\documentclass[preprint,11pt]{elsarticle}

\usepackage[T1]{fontenc}
\usepackage[utf8]{inputenc}
\usepackage{amsmath}
\usepackage{amsthm}
\usepackage{amssymb}
\usepackage{amsfonts}
\usepackage{mathrsfs}
\usepackage{graphicx}
\usepackage{array}
\usepackage{url}
\usepackage{float}
\usepackage{booktabs}
\usepackage{tabularx}
\usepackage[lmargin=2.5cm,tmargin=2.5cm,bmargin=2.5cm,rmargin=2.5cm]{geometry}
\usepackage{indentfirst}
\usepackage[hidelinks]{hyperref}
\hypersetup{colorlinks=true,linkcolor=red,citecolor=blue,urlcolor=blue}
\usepackage{xcolor}
\usepackage{subcaption}
\biboptions{numbers,sort&compress}
\journal{}

\newif\ifshowparnum
\showparnumfalse
\newcounter{parnum}
\newcommand{\PN}{\ifshowparnum\refstepcounter{parnum}{\footnotesize\textcolor{gray!70!black}{[\theparnum]}}~\fi}

\newcommand{\CLtf}{\mathfrak{X}_{\text{\scriptsize T}}}
\newcommand{\CLequi}{\mathfrak{X}_{\text{\scriptsize S}}}
\newcommand{\CLzero}{\mathfrak{X}_{\text{\scriptsize S}}^{0}}
\newcommand{\CL}{\mathfrak{X}_{\text{\scriptsize$L$}}}
\DeclareMathOperator{\Fix}{Fix}

\newtheorem{theorem}{Theorem}
\newtheorem{definition}{Definition}
\newtheorem{proposition}{Proposition}
\newtheorem{corollary}{Corollary}
\newtheorem{lemma}{Lemma}
\theoremstyle{definition}

\newtheorem{remark}{Remark}

\begin{document}

\begin{frontmatter}

\title{Global classification of oscillatory dynamics in symmetric zero-divergence 3D piecewise-linear Filippov systems with a visible--visible two-fold}

\author{Samuel Carlos S. Ferreira}
\ead{ferreira.samuelcarlos@gmail.com}
\address{Independent researcher, Goi\^ania, Brazil}
\begin{abstract}
\PN A global classification of the asymptotic oscillatory dynamics is established for a symmetric zero-divergence class of three-dimensional piecewise-linear Filippov systems, up to a set of initial conditions of zero Lebesgue measure. The affine fields are related by an involution, and the switching plane contains a visible--visible two-fold. In canonical coordinates, the eigenvalues are \(\mu\pm i\) and \(-2\mu\), while \(H\) measures the focal-line inclination. For every \(\mu>0\), a simple-period crossing cycle exists if and only if \(H\in\mathcal I_\mu\), and is unique, symmetric, hyperbolic, and orbitally asymptotically stable. Its half-period parametrizes \(\mathcal I_\mu\) and determines the crossing points, period, and Floquet multipliers. Global dissipation excludes crossing cycles with any higher number of crossings and makes the classified cycle the \(\omega\)-limit set of every crossing-only trajectory. If attractive sliding has no interior pseudo-equilibria, sliding is transient unless the trajectory reaches the two-fold. The initial conditions leading to the two-fold lie in a countable union of analytic surfaces and curves. Under sufficiently small perturbations in the Whitney \(C^1\) topology, the cycle persists as the unique simple-period crossing limit cycle and remains hyperbolic and orbitally asymptotically stable.
\end{abstract}


\end{frontmatter}

\PN Self-sustained oscillations are among the fundamental recurrent regimes in dynamical systems, and an attracting limit cycle provides their standard mathematical representation \cite{andronov1966,guckenheimer1983}. Their qualitative study involves determining when these periodic regimes exist, whether they are unique, whether they attract or repel nearby trajectories or exhibit saddle-type behavior, and how they respond to perturbations. In a parametrized family, these questions are also linked to the dependence of the oscillatory regime on the system parameters. The resulting problem is to identify the region of parameter space that supports the oscillatory regime and to characterize the corresponding limit cycle for each admissible parameter.

\PN Piecewise-smooth dynamical systems arise naturally in the modeling of physical systems and processes whose governing dynamics changes abruptly. In such systems \cite{Filippov,diBernardo2008}, a trajectory passes from one dynamical regime to another when it crosses the switching manifold, so the orbits are determined by successive flight segments generated by different vector fields. The switching structure gives rise to dynamical phenomena including isolated limit cycles in the linear case \cite{JBM,RBM,ferreira2026}, invariant cones foliated by periodic orbits \cite{Carmona,freitas3}, and bifurcations associated with two-fold singularities \cite{cristiano2019}. These systems also arise in relay control, switched electronic circuits, mechanical systems with impacts or friction, and idealized neuronal dynamics; see, for example, \cite{Brogliato,belykh,vanSoest,coombes2023}.

\PN Among the periodic regimes that arise in piecewise-smooth systems, crossing limit cycles are characterized by successive smooth flights through the regions separated by the switching manifold, with transverse intersections at each change of regime. Their study has used local bifurcation theory, averaging, and asymptotic methods near finite or infinite equilibria \cite{Freire,Freire2020_1,Freire2020_2,Freire2023,llibre2015a,llibre2015b}. A complementary route uses first integrals to reduce closure to algebraic matching on the switching manifold \cite{jaume,jaumedurval,jaumemarco}. The present work first gives a complete classification of simple-period crossing cycles within the symmetric zero-divergence class, including the admissible parameter region, uniqueness, symmetry, stability, and orbit reconstruction, and then extends this analysis to the global asymptotic dynamics.

\PN This paper considers a class of three-dimensional piecewise-linear discontinuous systems defined by two affine linear vector fields \(X\) and \(Y\) on the half-spaces separated by the switching manifold \(\Sigma=\{(x,y,z)\in\mathbb{R}^{3}\mid z=0\}\). The matrices \(DX\) and \(DY\) each have a nonzero real eigenvalue and a nonreal complex-conjugate pair. The tangency sets \(L_X\) and \(L_Y\) are nonempty lines, and every point of \(L_X\) and \(L_Y\) is a visible quadratic fold for the corresponding vector field. The invariant affine planes \(W^X\) and \(W^Y\), associated with the respective complex-conjugate eigenvalue pairs, intersect \(\Sigma\) along the focal lines \(r^X\) and \(r^Y\). The four lines \(L_X,L_Y,r^X,r^Y\) are pairwise distinct and concurrent at a visible--visible two-fold point \(p_\ast\in\Sigma\), with \(r^X\setminus\{p_\ast\}\) and \(r^Y\setminus\{p_\ast\}\) contained in the crossing region. The two vector fields are related by the involutive symmetry
\[
S(x,y,z)=(-y,-x,-z),
\qquad
Y=S\circ X\circ S.
\]
Consequently, \(L_Y=S(L_X)\) and \(W^Y=S(W^X)\). The dynamics generated by \(Y\) is therefore determined from that of \(X\), reducing the return construction to the analysis of a single vector field.

\PN Under these hypotheses, the symmetric class admits a normal form obtained by a common affine change of coordinates, together with positive amplitude and time rescalings, chosen so as to preserve the Filippov sliding dynamics. This form is parametrized by \(\lambda\), \(\mu\), \(H\), and \(\kappa\). The parameter \(\kappa\) retains the information relevant to the sliding dynamics, whereas the crossing flights are represented by the canonical member \(\kappa=0\). In this representative, the eigenvalues of \(DX\) are \(\lambda\) and \(\mu\pm i\), while \(H\) measures the inclination of the focal line \(r^X\). Figure~\ref{fig:focal_geometry_3d} illustrates this configuration.
\begin{figure}[ht]
	\centering
	\includegraphics[width=0.5\textwidth]{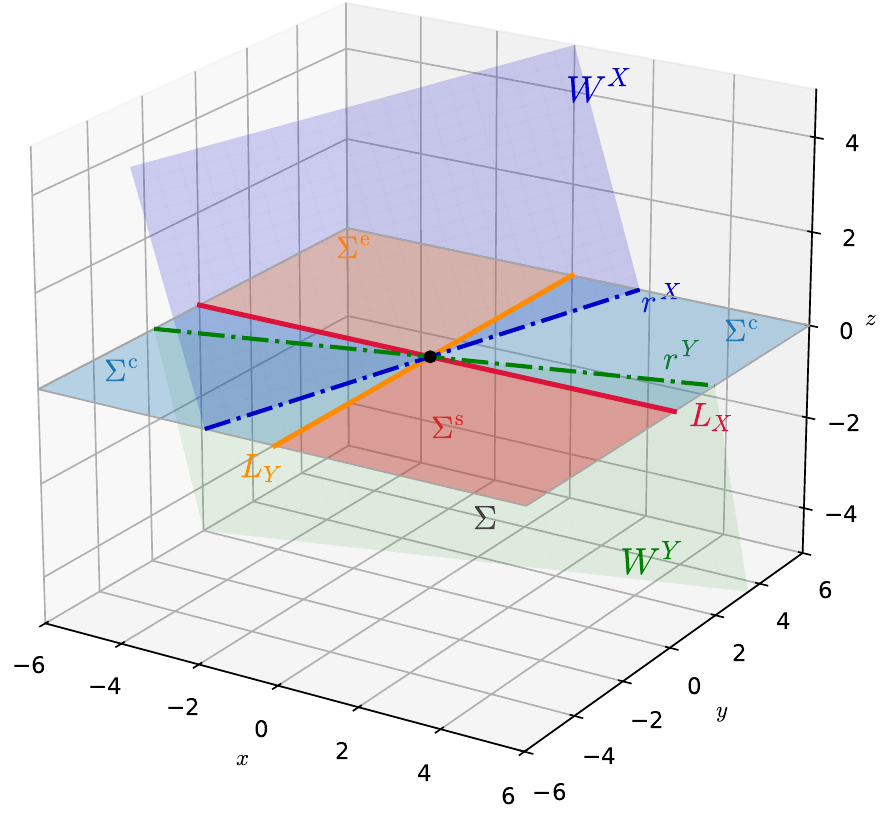}
\caption{Canonical geometric configuration. The tangency lines $L_X$ and $L_Y$ divide $\Sigma$ into crossing, sliding, and escaping regions, while the focal planes $W^X$ and $W^Y$ intersect $\Sigma$ along $r^X$ and $r^Y$.}
	\label{fig:focal_geometry_3d}
\end{figure}

\PN In the zero-divergence subclass, the absence of singular points in attractive sliding is equivalent to \(\kappa\le0\). This condition will enter only when sliding is incorporated into the global analysis.

\PN Theorem~\ref{teo:main1} identifies the interval \(\mathcal I_\mu\) in which a simple-period crossing cycle exists. For each admissible \((\mu,H)\), this cycle is unique within the simple-period crossing class, symmetric, hyperbolic, and orbitally asymptotically stable; its half-period determines its crossing points, period, and nontrivial Floquet multipliers.

\PN When the attractive sliding region has no interior pseudo-equilibria, every sliding episode ends in finite time, either at a regular exit fold or at the two-fold. If the trajectory does not reach the two-fold, only finitely many sliding episodes occur and it eventually becomes crossing-only. Theorem~\ref{thm-global-crossing} shows that every crossing-only trajectory has the classified crossing cycle as its \(\omega\)-limit set. Theorem~\ref{thm-global-hybrid} extends this conclusion to trajectories with sliding episodes. Proposition~\ref{prop-terminal-null} shows that the set of initial conditions whose trajectories reach the two-fold in finite time has zero three-dimensional Lebesgue measure. Consequently, Theorem~\ref{thm-almost-global} establishes that the classified cycle attracts every initial condition in \(\mathbb R^3\), with the exception of a set of zero Lebesgue measure. Corollary~\ref{cor:C1_persistence} gives local \(C^1\) persistence of the classified cycle, while Corollary~\ref{cor:global_C1_uniqueness} shows that its uniqueness within the simple-period crossing class is preserved under sufficiently small perturbations in the Whitney \(C^1\) topology. 

\PN Section~\ref{sec:dvf} defines the Filippov class and derives its canonical form. Sections~\ref{sec:caralimitcyle} and~\ref{sec:parametrization} construct the return map and classify the simple-period crossing cycles. Section~\ref{sec:stability} establishes their stability and local persistence. Section~\ref{sec-global-crossing} develops the global crossing and sliding dynamics, establishes the global asymptotic classification up to a set of initial conditions of zero Lebesgue measure, and analyzes the robustness of uniqueness in the Whitney \(C^1\) topology. Section~\ref{sec:edges} describe the oscillatory region in parameter space.

\section{Filippov class and canonical reduction}\label{sec:dvf}

\PN This section introduces the Filippov setting and the notion of simple-period crossing orbits, then derives the canonical form for the symmetric class. In these coordinates, \(\lambda\) is the real eigenvalue and \(\mu\) is the real part of the complex eigenvalue pair and \(H\) is the focal-line inclination. The zero-divergence subclass is then identified, reducing the closure problem to an algebraic relation. For general background on piecewise-smooth vector fields, see \cite{Filippov}.

\subsection{Filippov setting and simple-period orbits}\label{subsec:filippov_setting}

\PN Let \(\mathfrak X\) be the set of affine linear vector fields on \(\mathbb R^3\). Thus \(X=(X_1,X_2,X_3)\in\mathfrak X\), with \(X_j(x,y,z)=A_jx+B_jy+C_jz+D_j\) for real constants \(A_j,B_j,C_j,D_j\). A point \(p\in\mathbb R^3\) is a singularity of \(X\) when \(X(p)=0\).

\PN The analysis is carried out within the following class of piecewise-linear vector fields.

\begin{definition}\label{def:cl2z} \PN Let \(\CL\) be the set of the piecewise--linear vector fields \(Z\), 
\begin{equation}\label{eq:geral}
Z(\mathbf{x})=\left(X(\mathbf{x}),Y(\mathbf{x})\right)=\begin{cases} X(\mathbf{x})\ \text{if} \ \mathbf{x} \in \Sigma^{+}, \\ Y(\mathbf{x})\ \text{if} \ \mathbf{x} \in \Sigma^{-}, \end{cases} 
\end{equation} 
where \(\mathbf{x}=\left(x,y,z\right)\), \(X,Y\in\mathfrak{X}\) such that the matrices $DX$ and $DY$ each have a nonzero real eigenvalue and a pair of conjugate complex eigenvalues with a nonzero imaginary part. The  switching set is the plane \(\Sigma=\{\mathbf{x}\in\mathbb{R}^3 \ | \ f\left(\mathbf{x}\right)=z=0\}\),  \(\Sigma^{+}=\{\mathbf{x}\in\mathbb{R}^3 \ | \ f\left(\mathbf{x}\right)\geq 0\}\) and  \(\Sigma^{-}=\{\mathbf{x}\in\mathbb{R}^3 \ | \ f\left(\mathbf{x}\right)< 0\}\).
\end{definition}

\PN The switching plane \(\Sigma\) contains three distinct regions, determined by the way the vector fields \(X\) and \(Y\) meet \(\Sigma\). Their classification depends on the Lie derivatives \(Xf(p)=\left<X(p),\nabla f(p)\right>\) and \(Yf(p)=\left<Y(p),\nabla f(p)\right>\).

\begin{definition}\label{def:switching_regions}
\PN For \(Z=(X,Y)\in\CL\), define
\[
\Sigma^{\rm c}:=\{p\in\Sigma\mid Xf(p)Yf(p)>0\},
\]
\[
\Sigma^{\rm s}:=\{p\in\Sigma\mid Xf(p)<0<Yf(p)\},
\qquad
\Sigma^{\rm e}:=\{p\in\Sigma\mid Yf(p)<0<Xf(p)\}.
\]
These are called the \emph{crossing}, \emph{sliding}, and \emph{escaping} regions, respectively.
\end{definition}

\PN The \emph{tangency points} are the points \(p\in\Sigma\) at which \(Xf(p)=0\) or \(Yf(p)=0\). In the present piecewise-linear setting these points form lines, denoted by \(L_X\) and \(L_Y\) for the fields \(X\) and \(Y\), respectively.
 
\begin{definition}\label{def:tangencia_quadratica}
\PN A point $p\in L_{X}$ ($q\in L_{Y}$) is said to be a fold if $X^2f(p)\neq0$ ($Y^2f(q)\neq0$). A fold $p\in L_{X}$  ($q\in L_{Y}$) is called a visible fold point for the vector field $X$ (resp., $Y$) if $X^2f(p)>0$ (resp., $Y^2f(q)<0$) and invisible fold if $X^2f(p)<0$ (resp., $Y^2f(q)>0$). A point $p\in\Sigma$ is a double fold  point if $p$ is a tangency point for both vector fields $X$ and $Y$. A point $p\in L_{X}$ ($q\in L_{Y}$) is said to be a cusp if $X^2f(p)=0$ and $X^3f(p)\neq0$ ($Y^2f(q)=0$ and $Y^3f(q)\neq0$). 
\end{definition}

\PN For $Z=(X,Y)\in \CL$, write \(X(\mathbf{x}) = \sum_{j=1}^{3} (A_j x + B_j y + C_j z + D_j) \frac{\partial}{\partial x_j}\) and \(Y(\mathbf{x}) = \sum_{j=1}^{3} (a_j x + b_j y + c_j z + d_j) \frac{\partial}{\partial x_j}\), with \(\mathbf{x} = (x, y, z)\) and \((x_1,x_2,x_3)=(x,y,z)\). The tangency lines are then given by \(L_X = \{(x, y, z) \in \Sigma: A_3 x + B_3 y + D_3 = 0\}\) and \(L_Y = \{(x, y, z) \in \Sigma: a_3 x + b_3 y + d_3 = 0\}\). Within \(\Sigma\), the lines \(L_X\) and \(L_Y\) can be concurrent, parallel, or coincident.

\PN Since the configuration studied here requires nonempty tangency lines, it is assumed throughout that $A_3^2+B_3^2\neq 0$ and $a_3^2+b_3^2\neq 0$, so that $L_X$ and $L_Y$ are well-defined lines.

\PN Consider $X\in\mathfrak{X}$ and the differential system
\begin{equation}\label{eq:sistema_diferencial}
\dot{\mathbf{x}}\left(t\right)=X\left(\mathbf{x}\left(t\right)\right), 
\end{equation} 
where $\mathbf{x}=\left(x,y,z\right)$. Let $\varphi_X\left(t,p_0\right)$ be the flow of \eqref{eq:sistema_diferencial} through $p_0\in\mathbb{R}^3$. Thus, \[ \begin{cases} \dfrac{d}{dt}\varphi_X\left(t,p_0\right)=X\left(\varphi_X\left(t,p_0\right)\right),\\ \varphi_X\left(0,p_0\right)=p_0. \end{cases} \]

\PN The flow $\varphi_X(t,p_0)$ is defined on its maximal interval $I(p_0,X)$. The notation $\varphi_Y$ is used analogously. The classification below uses crossing trajectories. In each open half-space they follow the corresponding smooth flow, and at a point of $\Sigma^{\rm c}$ the incoming and outgoing arcs concatenate uniquely according to the signs of $Xf$ and $Yf$. The resulting maximal piecewise trajectory is denoted by $\varphi_Z(t,p_0)$. The later global result also uses the Filippov convex combination on $\Sigma^{\rm s}$ and the outgoing smooth characteristic at a regular visible fold. That construction is stopped at the two-fold. No continuation through the two-fold or within $\Sigma^{\rm e}$ is selected. These conventions agree with the standard Filippov construction \cite{Filippov}.

\begin{definition}\label{def:orbit}
	\PN Let \(p_0\) lie on a crossing trajectory described above, and let \(I_{p_0}\) be the maximal interval containing \(0\) on which the trajectory
	\(t\mapsto\varphi_Z(t,p_0)\) is defined.
	\begin{itemize}
		\item The \emph{orbit} through \(p_0\) is
		\[
		\gamma(p_0)=\{\varphi_Z(t,p_0):t\in I_{p_0}\}.
		\]
		
		\item A nonconstant orbit \(\gamma(p_0)\) is \emph{periodic} if
		\(I_{p_0}=\mathbb R\) and there exists \(T>0\) such that
		\[
		\varphi_Z(t+T,p_0)=\varphi_Z(t,p_0)
		\]
		for every \(t\in\mathbb R\). The least such \(T\) is called its \emph{period}.
	\end{itemize}
\end{definition}

\begin{definition}\label{def:simple_period}
	\PN Let \(\gamma\) be a periodic orbit of \(Z\).
	\begin{itemize}
		\item The orbit \(\gamma\) is a \emph{crossing periodic orbit} if
		\[
		\varnothing\neq\gamma\cap\Sigma\subset\Sigma^{\rm c}.
		\]
		
		\item A crossing periodic orbit \(\gamma\) has \emph{simple period} if
		\[
		\gamma\cap\Sigma=\{p_0,p_1\},
		\]
		with \(p_0\neq p_1\).
		
		\item The orbit \(\gamma\) is a \emph{limit cycle} if it is isolated among the periodic orbits of \(Z\).
	\end{itemize}
\end{definition}

\PN A parameter value belongs to the \emph{simple-period oscillatory regime} when the corresponding system possesses a hyperbolic, orbitally asymptotically stable crossing limit cycle of simple period. The parameter region, the family of cycles that it supports, and each individual cycle will be distinguished throughout the analysis.

\subsection{Symmetry and canonical form}\label{sec:cf}

\PN A canonical form is derived for the class of three-dimensional symmetric piecewise-linear Filippov systems under consideration. Since \(Y=S\circ X\circ S\), it is enough to normalize \(X\) on \(\Sigma^+\) by transformations that preserve the switching manifold and the symmetry; the corresponding field on \(\Sigma^-\) is then determined by equivariance.

\PN The geometric configuration described in the Introduction is first isolated within \(\CL\). At this stage no symmetry between the two sides is imposed: the defining conditions involve only the tangency and focal lines on \(\Sigma\). The \(S\)-equivariant subclass is introduced afterward.

\begin{definition}\label{def:CLperp} 
\PN Let \(\CLtf\subset\CL\) denote the class of vector fields \(Z=(X,Y)\) for which every point of \(L_X\) and \(L_Y\) is a visible fold for the corresponding vector field, in the sense of Definition~\ref{def:tangencia_quadratica}. The four lines \(L_X,L_Y,r^X,r^Y\) are pairwise distinct and concurrent at the visible--visible two-fold point \(p_\ast\in\Sigma\), and the focal-plane traces \(r^X\setminus\{p_\ast\}\) and \(r^Y\setminus\{p_\ast\}\) lie in the crossing region \(\Sigma^{\rm c}\).
\end{definition}

\PN At the level of configurations of four ordered lines in \(\Sigma\), the concurrency condition has codimension two: two distinct lines determine the common point, and requiring each of the remaining two lines to pass through it imposes one additional scalar condition. The geometric conditions in Definition~\ref{def:CLperp} determine the local configuration of the switching manifold but do not yet relate the dynamics on its two sides. The latter is achieved by imposing a linear involutive symmetry, which defines the equivariant subclass considered throughout this work.

\begin{definition}\label{def:equiva}
	\PN Let \(\CLequi\subset\CLtf\) be the class of systems \(Z=(X,Y)\) that are equivariant under the involution \(S\) introduced above, that is,
	\begin{equation}\label{eq:def_equiv_final}
		Y=S\circ X\circ S.
	\end{equation}
	Since \(S\) preserves \(\Sigma\) and exchanges the two open half-spaces, the vector field on \(\Sigma^-\) is completely determined by the vector field on \(\Sigma^+\).
\end{definition}

\PN A periodic orbit \(\gamma\) of a system \(Z\in\CLequi\) is called \emph{symmetric} if
\[
S(\gamma)=\gamma.
\]

\PN The involution relates the two tangency configurations and constrains which sidewise translations preserve their common switching geometry. These restrictions permit normalization of one field while retaining the symmetry of the pair.

\begin{remark}\label{rem:fix_S}
	\PN The involution \(S\) satisfies
	\[
	S^2=I,
	\qquad
	\det S=1,
	\]
	and is therefore the half-turn about the axis
	\[
	\Fix(S)=\{(x,-x,0)\mid x\in\mathbb{R}\}\subset\Sigma.
	\]
	Since \(f\circ S=-f\), equivariance implies
	\[
	Yf=-\,(Xf)\circ S,
	\qquad
	Y^2f=-\,(X^2f)\circ S.
	\]
	Hence \(p\in L_X\) is a visible fold for \(X\) if and only if \(S(p)\in L_Y\) is a visible fold for \(Y\). Thus, within the equivariant class, the two visibility requirements in Definition~\ref{def:CLperp} are equivalent. Moreover, a translation
	\[
	\mathbf{x}\mapsto\mathbf{x}+\mathbf{c}
	\]
	commutes with \(S\) if and only if
	\[
	\mathbf{c}\in\Fix(S).
	\]
\end{remark}

\PN The canonical reduction pairs affine coordinate changes on the two half-spaces by conjugation with \(S\). Their restrictions to \(\Sigma\) agree and commute with \(S|_\Sigma\), which preserves the equivariance of the pair.

\begin{proposition}\label{prop:equivariance_preserved}
\PN Suppose \(T_{+}:\mathbb R^3\to\mathbb R^3\) is an affine diffeomorphism satisfying \(T_{+}(\Sigma)=\Sigma\) and \(T_{+}(\Sigma^{+})=\Sigma^{+}\), and associate with it the map
\[
T_{-}:=S\circ T_{+}\circ S\quad\text{on }\Sigma^{-}.
\]
Assume that \(T_{+}|_{\Sigma}\) commutes with \(S|_{\Sigma}\), so that \(T_{+}\) and \(T_{-}\) agree on \(\Sigma\). If \(Z=(X,Y)\in\CLequi\), then the piecewise pushforward obtained from \(T_{+}\) on \(\Sigma^{+}\) and \(T_{-}\) on \(\Sigma^{-}\) is again \(S\)-equivariant.
\end{proposition}

\begin{proof}
\PN Write \(\widetilde X=(T_{+})_{*}X\) and
\(\widetilde Y=(T_{-})_{*}Y\). Since
\(T_{-}=S\circ T_{+}\circ S\), one has
\[
T_{-}^{-1}\circ S=S\circ T_{+}^{-1},
\qquad
DT_{-}[S]=[S]DT_{+}.
\]
Hence, using \(Y=S\circ X\circ S\), for
\(\mathbf u\in\Sigma^{+}\),
\[
\begin{aligned}
	\widetilde Y(S\mathbf u)
	&=
	DT_{-}\,
	Y\bigl(T_{-}^{-1}(S\mathbf u)\bigr)\\
	&=
	DT_{-}[S]\,
	X\bigl(T_{+}^{-1}(\mathbf u)\bigr)\\
	&=
	[S]DT_{+}\,
	X\bigl(T_{+}^{-1}(\mathbf u)\bigr)\\
	&=
	S\bigl(\widetilde X(\mathbf u)\bigr).
\end{aligned}
\]
Therefore the transformed piecewise vector field is \(S\)-equivariant.
\end{proof}

\PN The normalization will separate a common coordinate reduction, which preserves the Filippov field, from a final sidewise shear, which fixes the smooth-flight representative. Equivariance determines the lower field throughout.

\begin{proposition}\label{prop:canonical_form}
\PN Consider \(Z=(X,Y)\in\CLequi\). There exists an affine change of coordinates applied to both half-spaces, preserving \(\Sigma^\pm\) and commuting with \(S\), together with a positive amplitude rescaling and a positive linear rescaling of time, such that the transformed Filippov system has the Filippov-preserving normal form \(Z_\kappa=(X_\kappa,Y_\kappa)\), where
\begin{equation}\label{eq:residual_normal_form}
\begin{aligned}
X_\kappa(x,y,z)
&=\bigl(
\lambda x-\kappa y-
\{H[(\lambda-\mu)^2+1]-\lambda\kappa\}z+H,\;
2\mu y-(\mu^2+1)z+1,\;y
\bigr),\\
Y_\kappa(x,y,z)&=S\circ X_\kappa\circ S(x,y,z).
\end{aligned}
\end{equation}
Here \(\lambda\) is the normalized real eigenvalue, the remaining eigenvalues are \(\mu\pm i\), \(H\) is the inclination of the focal line
\[
r^X=\{(x,y,0)\in\Sigma:x=Hy\},
\]
and the corresponding focal plane is
\[
W_\kappa^X=
\{(x,y,z)\in\mathbb R^3:
 x-Hy+[\kappa+H(2\mu-\lambda)]z=0\}.
\]
Moreover, the sidewise shears
\begin{equation}\label{eq:residual_shear_maps}
T_+^\kappa(x,y,z)=(x+\kappa z,y,z),\qquad
T_-^\kappa(x,y,z)=(x,y+\kappa z,z)
\end{equation}
satisfy \(T_-^\kappa=S\circ T_+^\kappa\circ S\), agree with the identity on \(\Sigma\), and send \(Z_\kappa\) to the canonical representative \(Z_0\), whose upper field is
\begin{equation}\label{eq:canonica_ortogonal}
X_0(\mathbf{x})
=
\bigl(
\lambda x-H\bigl(((\lambda-\mu)^2+1)z-1\bigr),
\;2\mu y-(\mu^2+1)z+1,
\;y
\bigr),
\end{equation}
with \(Y_0=S\circ X_0\circ S\). The focal plane of this representative is
\[
W_0^X=\{(x,y,z)\in\mathbb R^3:x-Hy+H(2\mu-\lambda)z=0\}.
\]
The Filippov-preserving normal form \(Z_\kappa\) and the canonical representative \(Z_0\) have the same normalized smooth flight times and endpoints on \(\Sigma\), and therefore induce the same section maps wherever these maps are defined.
\end{proposition}
\begin{proof}
\PN First use the common translation and tangency alignment already dictated by the geometric hypotheses. After placing the concurrency point at the origin, the same linear change on both sides sends \(L_X\) to the \(x\)-axis and commutes with \(S\). As in the visibility argument, one obtains
\[
X=(A_1x+B_1y+C_1z+D_1,\;B_2y+C_2z+\Lambda,\;y+C_3z),
\qquad \Lambda>0.
\]
Let \(\lambda_0\) be the real eigenvalue and \(\alpha\pm i\Omega\), with \(\Omega>0\), the complex-conjugate pair. The triangular linear part gives
\[
A_1=\lambda_0,\qquad B_2+C_3=2\alpha,\qquad
B_2C_3-C_2=\alpha^2+\Omega^2.
\]

\PN Apply now the same affine shear on the two half-spaces,
\[
G_0(x,y,z)=(x+C_3z,y+C_3z,z).
\]
It is invertible, preserves \(\Sigma^\pm\), and commutes with \(S\). Its pushforward of the upper field is
\[
(G_0)_*X=
\bigl(
\lambda_0x+(B_1+C_3)y+[C_1-C_3(\lambda_0+B_1)]z+D_1,\;
2\alpha y-(\alpha^2+\Omega^2)z+\Lambda,\;y
\bigr).
\]
Use the positive normalizations
\[
\widehat t=\Omega t,\qquad
(\widehat x,\widehat y,\widehat z)
=(\Omega x/\Lambda,\Omega y/\Lambda,\Omega^2z/\Lambda)
\]
and set
\[
\lambda=\frac{\lambda_0}{\Omega},\qquad
\mu=\frac{\alpha}{\Omega},\qquad
\kappa=-\frac{B_1+C_3}{\Omega}.
\]
After dropping the hats, the upper field has the form
\[
X=(\lambda x-\kappa y+Cz+D,\;2\mu y-(1+\mu^2)z+1,\;y).
\]

\PN Let \(\ell\) define the invariant focal plane associated with the complex pair. Its coefficient of \(x\) is nonzero, because otherwise the planar \((y,z)\)-block would have the real eigenvalue \(\lambda\); hence normalize it as
\[
\ell=x-Hy+az.
\]
The trace passes through the origin by concurrency. Since \(X\ell=\lambda\ell\), comparison of the constant, \(y\)-, and \(z\)-coefficients gives
\[
D=H,\qquad
a=\kappa+H(2\mu-\lambda),\qquad
C=-H[(\lambda-\mu)^2+1]+\lambda\kappa.
\]
This is precisely \eqref{eq:residual_normal_form}; the lower field follows from equivariance.

\PN Every coordinate change used up to this point is common to both sides. For a common affine diffeomorphism \(G\), linearity gives
\[
DG[\theta X+(1-\theta)Y]
=\theta\,DGX+(1-\theta)\,DGY.
\]
Because \(G\) preserves \(\Sigma\) and its two sides, the normal components acquire the same positive factor, so the Filippov weight cancelling the normal component is transported. Thus the common reduction carries the sliding field, its oriented trajectories, and its zeros to those of \(Z_\kappa\). The positive constant time rescaling also preserves oriented trajectories and finiteness of physical times.

\PN Finally, direct substitution in \eqref{eq:residual_normal_form} gives
\[
DT_+^\kappa X_\kappa=X_0\circ T_+^\kappa,\qquad
DT_-^\kappa Y_\kappa=Y_0\circ T_-^\kappa.
\]
The maps in \eqref{eq:residual_shear_maps} keep \(z\) fixed, are the identity on \(\Sigma\), and do not change the normalized time. They therefore identify the smooth flights, their first-return times, and their endpoints on \(\Sigma\), yielding the canonical representative \eqref{eq:canonica_ortogonal}. Their normal derivatives differ when \(\kappa\ne0\), so this final sidewise step does not assert a conjugacy of the sliding fields.
\end{proof}

\PN The common reduction preserves Filippov sliding, whereas the final sidewise shears identify the smooth flights with those of \(Z_0\). We use \(Z_0\) for crossing calculations and retain \(\kappa\) when noncrossing motion enters the argument.

\PN The planar focus of the canonical field has \(z\)-coordinate \(\zeta_\mu\), and its quadratic matching factor contains the shift \(\beta_\mu\). Their values are
\begin{equation}\label{eq:zeta_mu}
\zeta_\mu:=\frac{1}{1+\mu^2}
\end{equation}
and
\begin{equation}\label{eq:beta_mu}
\beta_\mu:=\frac{2\mu}{1+\mu^2}=2\mu\zeta_\mu.
\end{equation}

\PN The focal-line traces must be distinct and lie in the crossing region. Their positions in canonical coordinates therefore restrict the inclination \(H\) directly.

\begin{corollary}\label{cor:admissible_H}
\PN Every \(Z\in\CLequi\) written in the canonical form~\eqref{eq:canonica_ortogonal} has inclination parameter
\[
H\in(0,1)\cup(1,\infty).
\]
\end{corollary}
\begin{proof}
\PN On \(\Sigma\), one has \(Xf=y\) and \(Yf=x\), while
\[
r^X=\{x=Hy,\ z=0\}, \qquad r^Y=\{y=Hx,\ z=0\}.
\]
Since \(r^X\setminus\{0\}\subset\Sigma^{\rm c}\) by Definition~\ref{def:CLperp}, a nonzero point \((Hy,y,0)\in r^X\) satisfies \(Xf\,Yf=Hy^2>0\), and therefore \(H>0\). The same definition requires \(r^X\neq r^Y\); with \(H>0\), this excludes \(H=1\).
\end{proof}

\subsection{The zero-divergence subclass and the main result}\label{subsec:main_result}

\PN The classification below is restricted to \(\mu>0\) and \(\operatorname{div}X=0\). The case \(\mu=0\), for which the complex pair is purely imaginary, and the case \(\mu<0\), for which the planar focus is contracting in forward time, are not considered here. Systems with \(\operatorname{div}X\neq0\) also lie outside the present classification. Since \(\operatorname{div}X=\lambda+2\mu\), the zero-divergence condition is equivalent to \(\lambda=-2\mu\).

\begin{definition}\label{def:zero_divergence_class}
	\PN Let \(\CLzero\subset\CLequi\) be the class of systems written in the canonical form~\eqref{eq:canonica_ortogonal} with \(\mu>0\) and \(\operatorname{div}X=0\).
\end{definition}

\PN Setting \(\lambda=-2\mu\) in the Filippov-preserving normal form \eqref{eq:residual_normal_form} gives
\begin{equation}\label{eq:zero_divergence_residual_fields}
\begin{aligned}
X_\kappa(x,y,z)
&=\bigl(
-2\mu x-\kappa y-[H(9\mu^2+1)+2\mu\kappa]z+H,\;
2\mu y-(1+\mu^2)z+1,\;y
\bigr),\\
Y_\kappa(x,y,z)
&=\bigl(
2\mu x-(1+\mu^2)z-1,\;
-2\mu y-\kappa x-[H(9\mu^2+1)+2\mu\kappa]z-H,\;x
\bigr).
\end{aligned}
\end{equation}
The member \(\kappa=0\) is the canonical representative used throughout the crossing classification:
\begin{equation}\label{eq:zero_divergence_fields}
\begin{aligned}
X(x,y,z)
&=
\bigl(
-2\mu x-H(9\mu^2+1)z+H,
\ 2\mu y-(1+\mu^2)z+1,
\ y
\bigr),\\
Y(x,y,z)
&=
\bigl(
2\mu x-(1+\mu^2)z-1,
\ -2\mu y-H(9\mu^2+1)z-H,
\ x
\bigr).
\end{aligned}
\end{equation}
In this representative the eigenvalues and the Darboux cofactors are explicit. Proposition~\ref{prop:canonical_form} shows that choosing \(\kappa=0\) loses no smooth crossing-flight information, while \(Z_\kappa\) retains the information needed when sliding is considered.

\PN For the canonical member, the Filippov combination \(Z^{\rm s}=(xX-yY)/(x-y)\big|_{z=0}\) in \(x>0>y\), with \(s=x-y\) and \(d=x+y\), gives
\begin{equation}\label{eq:canonical_sliding_sd}
\dot s=H-1-2\mu\frac{d^2}{s},
\qquad
\dot d=d\left(\frac{1+H}{s}-2\mu\right).
\end{equation}

\PN The global argument requires attractive sliding without interior pseudo-equilibria. This requirement can be imposed on the original system, before choosing coordinates:
\begin{equation}\tag{NS}\label{eq:nonsingular_sliding}
Z^{\rm s}(p)\ne0\qquad\text{for every }p\in\Sigma^{\rm s}.
\end{equation}
Here \(\Sigma^{\rm s}\) is the strict attractive sliding region of Definition~\ref{def:switching_regions}, excluding its boundary. The common stage of Proposition~\ref{prop:canonical_form} preserves this condition; the following calculation characterizes this condition in the Filippov-preserving normal form.

\begin{lemma}\label{lem:nonsingular_sliding_parameter}
\PN Let \(\mu>0\) and \(0<H<1\). For the family \eqref{eq:zero_divergence_residual_fields}, condition \eqref{eq:nonsingular_sliding} holds if and only if \(\kappa\le0\).
\end{lemma}
\begin{proof}
\PN On strict attractive sliding, \(x>0>y\), the normal components are \(X_\kappa f=y\) and \(Y_\kappa f=x\). Hence
\[
Z_\kappa^{\rm s}
=\left.\frac{xX_\kappa-yY_\kappa}{x-y}\right|_{z=0}.
\]
With \(s=x-y>0\), \(d=x+y\), and \(|d|<s\), this gives
\begin{equation}\label{eq:residual_sliding_sd}
\dot s=H-1-2\mu\frac{d^2}{s}
+\frac{\kappa}{2s}(s^2-d^2),\qquad
\dot d=d\left(\frac{1+H}{s}-2\mu\right).
\end{equation}
If \(\kappa\le0\), then \(s^2-d^2>0\) and
\[
\dot s\le H-1=-(1-H)<0,
\]
so no interior pseudo-equilibrium exists. If \(\kappa>0\), the point
\[
(s_*,d_*)=\left(\frac{2(1-H)}{\kappa},0\right),
\qquad
p_\kappa=\left(\frac{1-H}{\kappa},-\frac{1-H}{\kappa},0\right)
\]
lies in strict attractive sliding and makes both components vanish. This proves both directions.
\end{proof}
The stationary sliding trajectory at \(p_\kappa\) prevents a two-fold-or-crossing-cycle alternative for every sliding initial condition when \(\kappa>0\). In escaping, the lateral fields point away from \(\Sigma\); no tangential forward continuation is selected. Sliding is used below as a transient mechanism in the global argument, while the periodic-orbit classification concerns crossing cycles.

\PN All crossing calculations below are carried out for \(Z_0\). For the cycle \(\gamma_{(\mu,H)}\) classified in Theorem~\ref{teo:main1}, denote by \(\gamma_{(\mu,H)}^\kappa\) its inverse sidewise image under \(T_\pm^\kappa\), and by \(\gamma_Z\) its further inverse image under the common coordinate reduction. The identical section maps preserve the crossing points in normalized coordinates, the normalized period, and the return spectrum for every \(\kappa\); the relation \(T_-^\kappa=S\circ T_+^\kappa\circ S\) also preserves symmetry. Orbital convergence of crossing flight arcs transfers because each inverse shear is a fixed linear map, and the common affine change is invertible. No conjugacy of the sliding fields with \(Z_0\) is asserted.

\PN Corollary~\ref{cor:admissible_H} restricts the inclination of the canonical geometry, but does not determine which values yield a crossing flight closing after two intersections with \(\Sigma\). In the zero-divergence subclass, both the closure condition and the stability of the resulting cycles admit a characterization.

\begin{theorem}\label{teo:main1}
\PN Consider systems \(Z\in\CLzero\) in the representation
given by~\eqref{eq:canonica_ortogonal}. For each \(\mu>0\) let \(t_{\mathrm{g}}(\mu)\in(\pi,2\pi)\) denote the solution of
\[
\frac{\cos t-e^{-\mu t}}{\sin t}=\mu,
\]
and define
\begin{equation}\label{eq:H_g_definition}
H_{\mathrm{g}}(\mu)=e^{-2\mu t_{\mathrm{g}}(\mu)}
\end{equation}
and
\begin{equation}\label{eq:H_crit_definition}
H_{\mathrm{crit}}(\mu)
=\frac{1}{2\cosh(\pi \mu)-1}.
\end{equation}
Set \(\mathcal I_\mu:= \bigl(H_{\mathrm{g}}(\mu),H_{\mathrm{crit}}(\mu)\bigr).\) Then the following statements hold.
\begin{enumerate}
\item[(a)] For \(t\in(\pi,t_{\mathrm g}(\mu))\), define
\begin{equation}\label{eq:xy_of_t}
	y_0(t):=\zeta_\mu\left(\frac{\cos t-e^{-\mu t}}{\sin t}-\mu\right),
	\qquad
	x_0(t):=\zeta_\mu\left(\mu-\frac{e^{\mu t}-\cos t}{\sin t}\right),
\end{equation}
and
\begin{equation}\label{eq:H_of_t}
	\eta(t):=
	\frac{x_0(t)y_0(t)+\zeta_{\mu}}
	{x_0(t)^2+y_0(t)^2-x_0(t)y_0(t)
		-\beta_{\mu}\bigl(x_0(t)-y_0(t)\bigr)+\zeta_{\mu}}.
\end{equation}

Then \(\eta:(\pi,t_{\mathrm g}(\mu))\to\mathcal I_{\mu}\) is a strictly decreasing bijection. The map \(t\mapsto p_0(t)=(x_0(t),y_0(t),0)\) parametrizes a curve of crossing points in \(\Sigma\). The trajectory from \(p_0(t)\) is a crossing periodic orbit of simple period if and only if \(H=\eta(t)\). It returns to \(\Sigma\) at \(S(p_0(t))\), and its period is \(2t\). For each \(H\in\mathcal I_\mu\), this orbit is unique within the simple-period crossing class. No crossing periodic orbit of simple period exists when \(H\notin\mathcal I_\mu\). Denote the classified orbit by \(\gamma_{(\mu,H)}\).
	
\item[(b)] Let \(\Pi_X\) denote the half-return map of \(X\) and define the reduced map
\begin{equation}\label{eq:reduced_map}
	F:=S\circ\Pi_X.
\end{equation}
For every \(H\in\mathcal I_{\mu}\), let \(t=\eta^{-1}(H)\) and set
\begin{equation}\label{eq:dif_reduced_map}
	N(t):=DF\bigl(p_0(t)\bigr).
\end{equation}
Let \(\lambda_1(t)\) and \(\lambda_2(t)\) be the roots of the characteristic polynomial of \(N(t)\). Then the two nontrivial Floquet multipliers of \(\gamma_{(\mu,H)}\) are \(\mu_j=\lambda_j(t)^2\), \(j=1,2\), and both have modulus strictly smaller than \(1\). Hence \(\gamma_{(\mu,H)}\) is isolated, hyperbolic, and orbitally asymptotically stable.
\end{enumerate}
\end{theorem}

\section{Algebraic closure and the reduced map}\label{sec:caralimitcyle}

\PN For a symmetric orbit of simple period, closure is determined by a single \(X\)-flight. After returning to \(\Sigma\), the trajectory must land at the image under \(S\) of its initial crossing point. In the zero-divergence subclass, the first integrals provide an algebraic matching relation on \(\Sigma\), while the planar subsystem determines the landing coordinate. This leads naturally to the reduced map \(F\) defined in~\eqref{eq:reduced_map}, whose fixed points correspond to symmetric closures. At this stage, the fixed-point analysis therefore classifies only the symmetric family. The remaining simple-period alternative is given by points of minimal period two of \(F\), which are treated later and shown not to occur.

\subsection{Symmetry and the closure condition}

\PN The geometric closure condition depends only on the equivariance of the system and therefore holds throughout the symmetric class \(\CLequi\), independently of the zero-divergence condition.

\begin{proposition}\label{prop:symmetric_closure} 
\PN For \(Z\in\CLequi\), a periodic orbit \(\gamma\) of simple period is symmetric if and only if its successive intersection points \((x_0,y_0,0)\) and \((x_1,y_1,0)\) with \(\Sigma\) satisfy
\[ 
x_1=-y_0 \qquad\text{and}\qquad y_1=-x_0. 
\] 
\end{proposition}
\begin{proof}
\PN Suppose first that \(\gamma\) is symmetric. Since \(S(\Sigma)=\Sigma\), the set
\[
\gamma\cap\Sigma=\{p_0,p_1\}
\]
is invariant under \(S\). No crossing point is fixed by \(S\). Indeed, if
\(p\in\Fix(S)\cap\Sigma\), then equivariance gives
\[
Y(p)=S(X(p)).
\]
Since \(S\) reverses the \(z\)-component,
\[
Yf(p)=-Xf(p),
\]
and consequently
\[
Xf(p)Yf(p)=-\bigl(Xf(p)\bigr)^2\leq0.
\]
Thus \(p\notin\Sigma^{\rm c}\). Hence \(S\) cannot fix either of the two
crossing points and must exchange them, so
\[
p_1=S(p_0)=(-y_0,-x_0,0),
\]
which gives
\[
x_1=-y_0,
\qquad
y_1=-x_0.
\]
Conversely, suppose that \(p_1=S(p_0)\). By equivariance, \(S\) maps the flight arc from \(p_0\) to \(p_1\) to an orbit arc of the opposite side from
\(S(p_0)=p_1\) to \(S(p_1)=p_0\). By uniqueness of the crossing trajectory, this is the complementary flight arc of \(\gamma\). Hence \(S(\gamma)=\gamma\), so \(\gamma\) is symmetric.
\end{proof}

\PN Proposition~\ref{prop:symmetric_closure} determines how a symmetric orbit of simple period must close, but it does not yet provide an algebraic condition for such a closure. The first integrals of \(X\) and \(Y\) provide this condition on \(\Sigma\).

\subsection{First integrals and the matching relation}\label{sec:invcurve}

\PN A differentiable function \(f\) is a Darboux factor of \(X\) when \(Xf=Kf\); the function \(K\) is its cofactor. Zero divergence gives the two factors below the opposite cofactors \(+2\mu\) and \(-2\mu\). Their cancellation makes the product a polynomial first integral of \(X\), and symmetry transfers a corresponding integral to \(Y\). Restricting these integrals to \(\Sigma\) supplies the algebraic matching relation. For related uses of first integrals in piecewise systems, see \cite{jaumedurval,jaume}.

\PN For fixed \(\mu>0\), set
\begin{equation}\label{eq:Q_mu}
Q_{\mu}(w):=w^2+\beta_{\mu} w+\zeta_{\mu}.
\end{equation}
Its discriminant satisfies
\[
\operatorname{disc}(Q_{\mu})=\beta_{\mu}^2-4\zeta_{\mu}=-4\zeta_{\mu}^2<0.
\]
Since \(Q_{\mu}\) is monic, \(Q_{\mu}(w)>0\) for every \(w\in\mathbb R\).

\begin{proposition}\label{prop:zero_divergence_first_integrals}
\PN For \(Z=(X,Y)\in\CLzero\), define
\[
f_1(x,y,z)=(z-\zeta_{\mu})^2+
\bigl(y-\mu(z-\zeta_{\mu})\bigr)^2
\]
and
\[
f_2(x,y,z)=x-Hy+4\mu Hz.
\]
Then
\[
X(f_1)=2\mu f_1,
\qquad
X(f_2)=-2\mu f_2.
\]
Consequently,
\[
P_X=f_1f_2,
\qquad
P_Y=-P_X\circ S
\]
are cubic first integrals of \(X\) and \(Y\), respectively. Their restrictions to the switching plane are
\[
P_X(x,y,0)=Q_{\mu}(y)(x-Hy),
\qquad
P_Y(x,y,0)=Q_{\mu}(-x)(y-Hx).
\]
Moreover, \(f_2\) defines the focal plane \(W^X\), and \(W^X\cap\Sigma=r^X=\{x=Hy\}\).
\end{proposition}
\begin{proof}
\PN Since \(\operatorname{div}X=0\), one has \(\lambda=-2\mu\). Substitution in the canonical field and direct differentiation give the two cofactor identities. Their sum vanishes in the derivative of \(f_1f_2\), so \(X(P_X)=0\). Equation~\eqref{eq:def_equiv_final} and the chain rule give \(Y(P_X\circ S)=0\), and therefore \(P_Y=-P_X\circ S\) is a first integral of \(Y\). On \(\Sigma\), one has
\[
f_1(x,y,0)=\zeta_{\mu}^2+(y+\mu\zeta_{\mu})^2=Q_{\mu}(y)
\]
and \(f_2=x-Hy\), which gives the displayed restrictions. The expression for \(f_2\) is the focal-plane generator obtained in Proposition~\ref{prop:canonical_form} after setting \(\lambda=-2\mu\). Since \(f_1\) is quadratic and \(f_2\) is linear, both products are cubic.
\end{proof}

\PN The same factors expose the geometry behind the matching relation. Set
\[
v=z-\zeta_{\mu},
\qquad
w=y-\mu v,
\qquad
\xi=x-Hy+4\mu Hz.
\]
Along the upper field,
\begin{equation}\label{eq:spectral_coordinates}
\dot\xi=-2\mu\xi,
\qquad
\dot v=\mu v+w,
\qquad
\dot w=-v+\mu w.
\end{equation}
Thus \(f_1=v^2+w^2\), \(f_2=\xi\), and the first integral is
\[
P_X=\xi(v^2+w^2).
\]
The focal plane is \(W^X=\{\xi=0\}\), the upper equilibrium and its stable line are
\[
q_X=\left(-4\mu H\zeta_{\mu},0,\zeta_{\mu}\right),
\qquad
W^s(q_X)=\{(x,0,\zeta_{\mu})\mid x\in\mathbb R\}.
\]
Equation~\eqref{eq:spectral_coordinates} gives a contraction toward \(W^X\) by the factor \(e^{-2\mu t}\), while the adapted radial distance \(R=(v^2+w^2)^{1/2}\) grows by \(e^{\mu t}\). The contraction and expansion explain the cancellation of the two Darboux cofactors in \(P_X\).

\PN Away from \(R=0\), set \(R=\sqrt{v^2+w^2}\) and choose a local branch of \(\theta=\arg(v+iw)\). Equation~\eqref{eq:spectral_coordinates} gives \(\dot R=\mu R\) and \(\dot\theta=-1\), so \(\theta+\mu^{-1}\log R\) is a locally defined transcendental first integral independent of \(P_X\). Its angular branch makes it less convenient globally, and it supplies no simple polynomial matching condition on \(\Sigma\). The single-valued integral \(P_X=\xi(v^2+w^2)\), whose restriction to \(\Sigma\) is algebraic, is therefore the one used for closure.

\PN By Proposition~\ref{prop:symmetric_closure}, a symmetric periodic orbit meets \(\Sigma\) at \((x,y,0)\) and \((-y,-x,0)\), so each first integral takes the same value at these two points. Thus
\[
P_X(x,y,0)-P_X(-y,-x,0)=0,
\qquad
P_Y(x,y,0)-P_Y(-y,-x,0)=0.
\]
These conditions coincide because \(P_Y=-P_X\circ S\).

\subsection{The planar subsystem}\label{subsec:PX_structure}

\PN The canonical form~\eqref{eq:canonica_ortogonal} contains the autonomous planar subsystem
\begin{equation}\label{eq:planar_subsystem}
\dot y=2\mu y-(\mu^{2}+1)z+1,
\qquad
\dot z=y.
\end{equation}
It is independent of the zero-divergence condition and has a focus at
\[
(y,z)=(0,\zeta_{\mu})
\]
with eigenvalues \(\mu\pm i\). Eliminating \(y\) gives
\[
\ddot z-2\mu\dot z+(1+\mu^2)z=1.
\]
For the solution issued from \((y_0,0)\), one obtains
\begin{equation}\label{eq:z_planar_solution}
z(t;y_0)
=
\zeta_{\mu}\left[1-e^{\mu t}\bigl(\cos t-\mu\sin t\bigr)\right]
+y_0e^{\mu t}\sin t.
\end{equation}
The corresponding \(y\)-component is recovered from
\[
y=\dot z.
\]

\PN The auxiliary coordinate \(\widetilde y=y-2\mu z\) agrees with \(y\) on \(\Sigma\) and satisfies
\[
\dot{\widetilde y}=1-(1+\mu^2)z,
\qquad
\dot z=\widetilde y+2\mu z.
\]
It therefore leaves the first-return time, the endpoint coordinate, and the chord map unchanged. The planar and three-dimensional flight illustrations below use this sidewise auxiliary coordinate, whose focus is \((-\beta_{\mu},\zeta_{\mu})\).

\PN Since the planar subsystem is independent of \(x\), the first-return time and the corresponding planar landing coordinate depend only on \(y_0\). For \(t\in(\pi,2\pi)\), imposing the return condition \(z(t;y_0)=0\) in \eqref{eq:z_planar_solution} gives
\[
y_0
=
\zeta_{\mu}
\left(
\frac{\cos t-e^{-\mu t}}{\sin t}-\mu
\right).
\]
This motivates the definition
\begin{equation}\label{eq:Phi_definition}
\Phi_{\mu}(t):=
\frac{\cos t-e^{-\mu t}}{\sin t},
\qquad
t\in(\pi,2\pi).
\end{equation}
for which the return equation becomes
\[
y_0=\zeta_{\mu}\bigl(\Phi_{\mu}(t)-\mu\bigr).
\]
Once the planar flight is determined, the \(x\)-component is obtained along the same arc, yielding the triangular structure of the \(X\)-half-return map \(\Pi_X\).

\begin{lemma}\label{lem:Phi_properties}
\PN For every \(\mu>0\), the function \(\Phi_{\mu}\) defined by \eqref{eq:Phi_definition} is strictly decreasing on \((\pi,2\pi)\) and satisfies
\[
\lim_{t\downarrow\pi}\Phi_{\mu}(t)=+\infty,
\qquad
\lim_{t\uparrow2\pi}\Phi_{\mu}(t)=-\infty.
\]
Consequently, for every \(x\in\mathbb R\), the equation
\[
\Phi_{\mu}(t)=x
\]
has a unique solution \(t\in(\pi,2\pi)\).
\end{lemma}

\begin{proof}
\PN Differentiation gives
\[
\Phi_{\mu}'(t)
=
\frac{e^{-\mu t}(\cos t+\mu\sin t)-1}{\sin^2 t}.
\]
The function
\[
t\longmapsto e^{-\mu t}(\cos t+\mu\sin t)
\]
is strictly increasing on \((\pi,2\pi)\), since
\[
\frac{d}{dt}\left[e^{-\mu t}(\cos t+\mu\sin t)\right]
=
-(1+\mu^2)e^{-\mu t}\sin t>0,
\]
and its value at \(2\pi\) is \(e^{-2\pi \mu}<1\). Hence
\[
\Phi_{\mu}'(t)<0
\]
throughout \((\pi,2\pi)\). The endpoint limits follow directly from the definition of \(\Phi_{\mu}\), and the final assertion follows from continuity and strict monotonicity.
\end{proof}
\PN By \eqref{eq:Phi_definition}, the equation defining \(t_{\mathrm g}(\mu)\) in Theorem~\ref{teo:main1} is equivalent to
\[
\Phi_{\mu}\bigl(t_{\mathrm g}(\mu)\bigr)=\mu.
\]
In view of the return relation above, this corresponds to the initial value \(y_0=0\).

\begin{definition}\label{def:grazing_time}
\PN The quantity \(t_{\mathrm g}(\mu)\) characterized above is called the \emph{grazing time}.
\end{definition}

\PN The following lemma identifies \(t_{\mathrm g}(\mu)\) as the endpoint of the planar first-return times and establishes uniqueness and transversality of the return.

\begin{lemma}\label{lem:first_return}
\PN For \(Z\in\CLzero\) and \(y_0\geq0\), consider the solution \((y(t),z(t))\) of \eqref{eq:planar_subsystem} issued from \((y_0,0)\). Its second component has a unique zero \(t^X(y_0)\in(\pi,t_{\mathrm g}(\mu)]\), with \(t^X(y_0)=t_{\mathrm g}(\mu)\) if and only if \(y_0=0\). Moreover, \(z(t)>0\) for \(0<t<t^X(y_0)\), and \(\dot z(t^X(y_0))<0\).
\end{lemma}
\begin{proof}
\PN For \(t\in(\pi,2\pi)\), equation~\eqref{eq:z_planar_solution} can be written as
\[
z(t;y_0)
=
-\zeta_{\mu} e^{\mu t}\sin t
\left[
\Phi_{\mu}(t)-\mu-(1+\mu^2)y_0
\right].
\]
Hence \(z(t;y_0)=0\) if and only if \(\Phi_{\mu}(t)=\mu+(1+\mu^2)y_0\). Lemma~\ref{lem:Phi_properties} gives a unique solution \(t^X(y_0)\in(\pi,2\pi)\). Since the right-hand side is at least \(\mu\), strict monotonicity of \(\Phi_{\mu}\) yields \(t^X(y_0)\leq t_{\mathrm g}(\mu)\), with equality only when \(y_0=0\).

\PN For \(0<t\leq\pi\), equation~\eqref{eq:z_planar_solution} and \(y_0\geq0\) give
\[
z(t;y_0)
\geq
\zeta_{\mu}
\left[
1-e^{\mu t}(\cos t-\mu\sin t)
\right].
\]
The function \(e^{\mu t}(\cos t-\mu\sin t)\) decreases strictly from \(1\) on \((0,\pi)\), since its derivative is \(-(1+\mu^2)e^{\mu t}\sin t<0\). Thus \(z(t;y_0)>0\) on \((0,\pi]\). For \(t\in(\pi,t^X(y_0))\), Lemma~\ref{lem:Phi_properties} gives \(\Phi_{\mu}(t)>\Phi_{\mu}(t^X(y_0))\), and the preceding representation of \(z\) again gives \(z(t;y_0)>0\). Finally,
\[
\dot z\bigl(t^X(y_0)\bigr)
=
-\zeta_{\mu} e^{\mu t^X(y_0)}
\sin\bigl(t^X(y_0)\bigr)
\Phi_{\mu}'\bigl(t^X(y_0)\bigr)
<0.
\]
\end{proof}

\PN The unique return established in Lemma~\ref{lem:first_return} defines the planar chord map and the corresponding \(X\)-return to \(\Sigma\).

\begin{definition}\label{def:chord_map}
\PN Let \(Z\in\CLzero\). For \(y_0\geq0\), let \(t^{X}(y_0)\) denote the first-return time given by Lemma~\ref{lem:first_return}. The \emph{planar chord map} is
\[
g(y_0):=y\bigl(t^{X}(y_0);y_0\bigr),
\]
the \(y\)-coordinate at which the arc of \(X\) issued from \((y_0,0)\) lands back on \(\{z=0\}\). For \(p=(x_0,y_0,0)\in\Sigma\) with \(y_0\geq0\), the \(X\)-half-return map is defined by
\[
\Pi_X(p):=\varphi^X\bigl(t^X(y_0);p\bigr).
\]
\end{definition}

\PN The endpoint \(y_0=0\) of the planar chord map corresponds to the grazing return shown below. By Lemma~\ref{lem:first_return}, its flight time is \(t_{\mathrm g}(\mu)\) and \(g(0)<0\). Define
\begin{equation}\label{eq:x_g_definition}
	x_{\mathrm g}(\mu):=-g(0)>0.
\end{equation}
The corresponding grazing threshold is \(H_{\mathrm g}(\mu)\), as defined in Theorem~\ref{teo:main1}.

\begin{lemma}\label{lem:grazing_data}
	\PN The maps \(t^X\) and \(g\) are real-analytic for \(y_0>0\) and extend real-analytically across \(y_0=0\). Moreover, \(H_{\mathrm g}(\mu)=\zeta_{\mu}/Q_{\mu}(-x_{\mathrm g}(\mu))\).
\end{lemma}
\begin{proof}
\PN Lemma~\ref{lem:first_return} gives \(\dot z(t^X(y_0);y_0)<0\) for \(y_0\geq0\). Since \(z(t;y_0)\) is real-analytic, the real-analytic implicit function theorem \cite{Teschl2012,KrantzParks2002} and uniqueness of the first-return time imply that \(t^X\), and hence \(g(y_0)=y(t^X(y_0);y_0)\), extend real-analytically across \(y_0=0\). At \(y_0=0\), one has \(t^X(0)=t_{\mathrm g}(\mu)\) and \(g(0)=-x_{\mathrm g}(\mu)\). Since \(f_1(x,y,0)=Q_{\mu}(y)\) and \(X(f_1)=2\mu f_1\), Proposition~\ref{prop:zero_divergence_first_integrals} gives
	\[
	Q_{\mu}(-x_{\mathrm g})=e^{2\mu t_{\mathrm g}}\zeta_{\mu},
	\qquad
	\frac{\zeta_{\mu}}{Q_{\mu}(-x_{\mathrm g})}=e^{-2\mu t_{\mathrm g}}=H_{\mathrm g}(\mu).
	\]
\end{proof}

\PN Figure~\ref{fig:first_return_geometry} illustrates the planar first-return construction and the corresponding three-dimensional \(X\)-flight.

\begin{figure}[H]	
	\centering
	\begin{subfigure}[t]{0.47\textwidth}
		\centering
		\includegraphics[width=\linewidth]{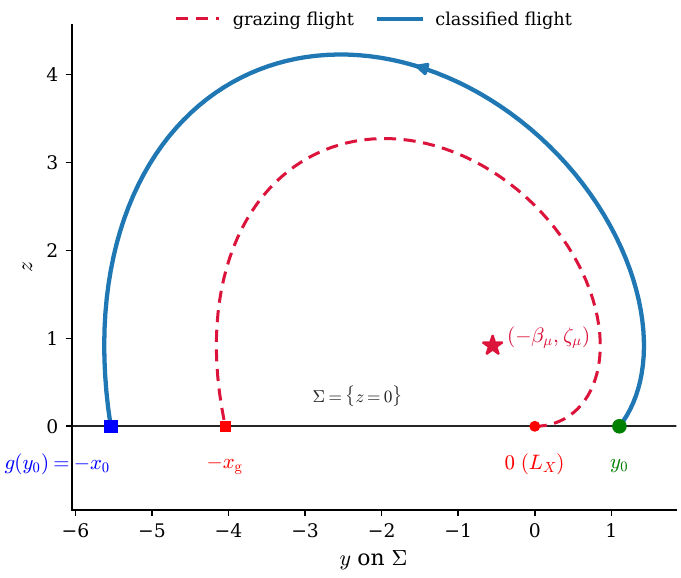}
		\caption{Planar first-return geometry in the auxiliary coordinate \(\widetilde y=y-2\mu z\), with the chord \(y_0\mapsto g(y_0)\), the grazing return from \(y_0=0\), and the unstable focus \((-\beta_{\mu},\zeta_{\mu})\).}
		\label{fig:first_return_planar}
	\end{subfigure}
	\hfill
	\begin{subfigure}[t]{0.50\textwidth}
		\centering
		\includegraphics[width=\linewidth]{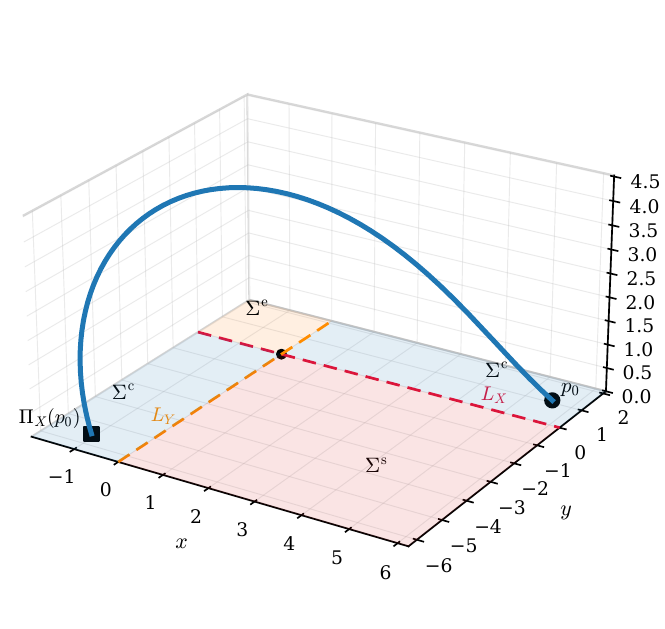}
		\caption{Three-dimensional \(X\)-flight in the sidewise auxiliary coordinates from \(p_0\) to its first return \(\Pi_X(p_0)\) on \(\Sigma\). The auxiliary and canonical coordinates agree at both endpoints.}
		\label{fig:first_return_3d}
	\end{subfigure}
	\caption{Geometry of the first \(X\)-return for \(\mu=0.3\), \(t=3.9007\), and \(H\approx0.2903\).}
	\label{fig:first_return_geometry}
\end{figure}

\subsection{The reduced map}
\label{subsec:halfmap}
\PN The chord map of Definition~\ref{def:chord_map} allows the triangular structure of \(\Pi_X\) to be written without retaining the flight time explicitly. The two Darboux factors along an \(X\)-arc relate the initial and return crossing coordinates directly, leading to the following formula for the half-return map.

\begin{lemma}\label{lem:PX_closed_form}
\PN Let \(Z\in\CLzero\). If an \(X\)-arc with flight time \(t\) joins
\[
(x_0,y_0,0)
\quad\text{to}\quad
(x_1,y_1,0)
\]
on \(\Sigma\), then
\begin{equation}\label{eq:factor_relations}
Q_{\mu}(y_1)=e^{2\mu t}Q_{\mu}(y_0),
\qquad
x_1-Hy_1=e^{-2\mu t}(x_0-Hy_0).
\end{equation}
Consequently, for the first \(X\)-return issued from \((x_0,y_0,0)\) with \(y_0\geq0\),
\begin{equation}\label{eq:PX_closed_form}
\Pi_X(x_0,y_0)
=
\left(
H\,g(y_0)
+
\frac{Q_{\mu}(y_0)}{Q_{\mu}(g(y_0))}
(x_0-Hy_0),
\;
g(y_0)
\right).
\end{equation}
\end{lemma}
\begin{proof}
\PN On \(\Sigma\), the Darboux factors of Proposition~\ref{prop:zero_divergence_first_integrals} satisfy
\[
f_1(x,y,0)=Q_{\mu}(y),
\qquad
f_2(x,y,0)=x-Hy.
\]
Their cofactors are \(2\mu\) and \(-2\mu\), respectively, and therefore along an \(X\)-arc
\[
f_1(t)=e^{2\mu t}f_1(0),
\qquad
f_2(t)=e^{-2\mu t}f_2(0).
\]
Evaluating at the endpoints gives \eqref{eq:factor_relations}. Since \(Q_{\mu}>0\), the two identities yield
\[
x_1
=
Hy_1+
\frac{Q_{\mu}(y_0)}{Q_{\mu}(y_1)}
(x_0-Hy_0).
\]
For the first return, \(y_1=g(y_0)\), which gives \eqref{eq:PX_closed_form}.
\end{proof}

\PN A symmetric closure requires the endpoint of the \(X\)-return to be the image under \(S\) of its initial crossing point. The reduced map \(F=S\circ\Pi_X\), defined in~\eqref{eq:reduced_map}, encodes that condition. Its iterates must start in the crossing component on which successive half-returns are defined, which fixes the natural domain below.

\begin{definition}\label{def:reduced_map_domain}
\PN Define the upward crossing domain
\[
\mathcal D:=\{(x,y,0)\in\Sigma\mid x>0,\ y>0\}.
\]
On \(\mathcal D\), consider the reduced map \(F\) defined in~\eqref{eq:reduced_map}, where \(\Pi_X\) is the first \(X\)-return introduced in Definition~\ref{def:chord_map}.
\end{definition}

\PN Set
\begin{equation}\label{eq:E_definition}
\mathcal E(y):=
\frac{Q_{\mu}(y)}{Q_{\mu}(g(y))}
=
e^{-2\mu t^X(y)}.
\end{equation}
By \eqref{eq:PX_closed_form},
\begin{equation}\label{eq:F_explicit}
F(x,y)
=
\bigl(
-g(y),
-Hg(y)-\mathcal E(y)(x-Hy)
\bigr).
\end{equation}
\PN At a symmetric crossing closure, the second intersection lies in the interior of the opposite crossing component. Thus the reduced map is locally iterable, and equivariance identifies its second iterate with the full return.

\begin{proposition}\label{prop:F_square}
\PN Let \(Z\in\CLzero\), and let \(p_\ast=(x_0,y_0,0)\in\mathcal D\) be the upward crossing point of a symmetric periodic orbit \(\gamma\) of simple period. Then \(F(p_\ast)=p_\ast\), and there exists a neighborhood \(U\subset\mathcal D\) of \(p_\ast\) such that \(F(U)\subset\mathcal D\). On this neighborhood,
\[
\Pi(p)=(F\circ F)(p).
\]
Consequently,
\[
D\Pi(p_\ast)=\bigl(DF(p_\ast)\bigr)^2.
\]
\end{proposition}
\begin{proof}
\PN Proposition~\ref{prop:symmetric_closure} gives
\[
\Pi_X(p_\ast)=S(p_\ast)=(-y_0,-x_0,0).
\]
Since \(x_0,y_0>0\), this point lies in the interior of the downward crossing component. Continuity of \(\Pi_X\) gives a neighborhood \(U\subset\mathcal D\) whose image remains in that component. The involution maps this component onto \(\mathcal D\), so \(F(U)\subset\mathcal D\) and \(F(p_\ast)=p_\ast\).
Equivariance gives
\[
\Pi_Y=S\circ\Pi_X\circ S
\]
on the corresponding crossing neighborhoods. Hence, on \(U\),
\[
\Pi_Y\circ\Pi_X
=
S\circ\Pi_X\circ S\circ\Pi_X
=
F\circ F.
\]
Therefore
\[
D\Pi(p_\ast)
=
DF\bigl(F(p_\ast)\bigr)\,DF(p_\ast)
=
\bigl(DF(p_\ast)\bigr)^2.
\]
\end{proof}

\PN The fixed-point equation for \(F\) splits into the planar chord condition and the first-integral matching condition. The next lemma makes this equivalence precise.

\begin{lemma}\label{lem:closure_identity}
	\PN For \(Z\in\CLzero\) and \(p_0=(x_0,y_0,0)\in\mathcal D\),
	\[
	F(p_0)=p_0
	\quad\Longleftrightarrow\quad
	g(y_0)=-x_0
	\quad\text{and}\quad
	P_X(p_0)=P_X(S(p_0)).
	\]
\end{lemma}
\begin{proof}
	By \eqref{eq:F_explicit}, the equality \(F(p_0)=p_0\) is equivalent to \(g(y_0)=-x_0\) together with \(\mathcal E(y_0)(x_0-Hy_0)=Hx_0-y_0\). Under the first relation, \eqref{eq:E_definition} shows that the second condition is equivalent to
	\[
	Q_{\mu}(y_0)(x_0-Hy_0)=Q_{\mu}(-x_0)(Hx_0-y_0),
	\]
	which, by the restriction of \(P_X\) to \(\Sigma\), is \(P_X(p_0)=P_X(S(p_0))\). 
\end{proof}

\subsection{Half-period parametrization}
\label{subsec:symmetric_existence}

\PN Fix \(\mu>0\) throughout this subsection. Recall that \(t_{\mathrm g}(\mu)\in(\pi,2\pi)\) is the grazing time of Definition~\ref{def:grazing_time}. Lemma~\ref{lem:closure_identity} reduces symmetric closure to the chord relation and the first-integral matching condition. The flight time provides a natural parameter for the corresponding crossing data.

\begin{lemma}\label{lem:param_properties}
	\PN For \(t\in(\pi,t_{\mathrm g}(\mu))\), one has \(y_0(t)>0\) and \(x_0(t)-y_0(t)>\beta_{\mu}\). Moreover, \(y_0\) is strictly decreasing from \(+\infty\) to \(0\) on this interval, and \(\eta(t)\) is well defined with \(0<\eta(t)<1\).
\end{lemma}

\begin{proof}
	\PN Since \(y_0(t)=\zeta_{\mu}(\Phi_{\mu}(t)-\mu)\), Lemma~\ref{lem:Phi_properties} and \(\Phi_{\mu}(t_{\mathrm g}(\mu))=\mu\) show that \(y_0\) is strictly decreasing from \(+\infty\) to \(0\) on \((\pi,t_{\mathrm g}(\mu))\). In particular, \(y_0(t)>0\). Moreover, \(\sin t<0\), \(e^{\mu t}-\cos t>0\), and \(e^{\mu t}-e^{-\mu t}>0\), so \eqref{eq:xy_of_t} gives \(x_0(t)-y_0(t)>2\mu\zeta_{\mu}=\beta_{\mu}\).
	
	\PN Finally, the denominator of \eqref{eq:H_of_t} can be written as \(Q_{\mu}(-(x_0-y_0))+x_0y_0>0\). Its numerator \(x_0y_0+\zeta_{\mu}\) is positive, while the difference between denominator and numerator is \((x_0-y_0)(x_0-y_0-\beta_{\mu})>0\). Hence \(0<\eta(t)<1\).
\end{proof}

\PN The preceding properties place the half-period data in the crossing domain and make the closure condition well defined. The next proposition shows that, for \(H=\eta(t)\), these data satisfy the reduced fixed-point condition and therefore generate a symmetric crossing periodic orbit.

\begin{proposition}\label{prop:symmetric_family}
	\PN Consider \(Z\in\CLzero\). Let \(\mu>0\) and \(t\in(\pi,t_{\mathrm g}(\mu))\) with \(H=\eta(t)\). Then \(p_0(t)\in\mathcal D\) is a fixed point of \(F\), and the trajectory through \(p_0(t)\) is a symmetric crossing periodic orbit of simple period, with crossing points \(p_0(t)\) and \(S(p_0(t))\) and period \(2t\).
\end{proposition}
\begin{proof}
\PN Lemma~\ref{lem:param_properties} gives \(y_0(t)>0\) and \(x_0(t)-y_0(t)>\beta_{\mu}\), hence \(p_0(t)\in\mathcal D\). Substitution into the planar solution gives \(z(t;y_0(t))=0\), so Lemma~\ref{lem:first_return} yields \(t=t^X(y_0(t))\). Using \(y=\dot z\) at the return and \eqref{eq:xy_of_t}, one obtains
\[
g(y_0(t))
=
\zeta_{\mu}\left(\frac{e^{\mu t}-\cos t}{\sin t}-\mu\right)
=
-x_0(t).
\]
By \eqref{eq:H_of_t}, the choice \(H=\eta(t)\) satisfies the first-integral matching condition in Lemma~\ref{lem:closure_identity}. Hence \(F(p_0)=p_0\), or equivalently \(\Pi_X(p_0)=S(p_0)\). Equivariance closes the second flight with the same flight time \(t\), and Proposition~\ref{prop:symmetric_closure} gives symmetry. Thus the crossing points are \(p_0(t)\) and \(S(p_0(t))\), and the period is \(2t\).
\end{proof}
\PN The determinant and trace of the reduced-map derivative, as well as the comparison needed to exclude two-cycles, depend on how the planar landing changes with its starting coordinate. The flight time, chord map, and factor \(\mathcal E\) therefore require the differential properties established below.

\begin{lemma}\label{lem:chord_monotone}
	\PN For \(Z\in\CLzero\), the flight time \(t^{X}\) and the chord map \(g\) are strictly decreasing on \((0,\infty)\), whereas \(\mathcal E\) is strictly increasing. Moreover,
	\[
	g'(y)=\frac{y}{\mathcal E(y)g(y)},
	\qquad y>0.
	\]
\end{lemma}
\begin{proof}
	\PN Let \(y_0>0\), set \(t=t^{X}(y_0)\). Differentiating \eqref{eq:z_planar_solution} with respect to \(y_0\), and using \(y=\dot z\), gives
	\[
	\frac{\partial y}{\partial y_0}=e^{\mu t}(\cos t+\mu\sin t),
	\qquad
	\frac{\partial z}{\partial y_0}=e^{\mu t}\sin t.
	\]
	At the return to \(\Sigma\), one has \(\dot y=2\mu g(y_0)+1\) and \(\dot z=g(y_0)\). Implicit differentiation of \(z(t^{X}(y_0);y_0)=0\) therefore yields
	\[
	(t^{X})'(y_0)
	=
	-\frac{e^{\mu t}\sin t}{g(y_0)}
	<0,
	\]
	since \(t\in(\pi,2\pi)\), \(\sin t<0\), and \(g(y_0)<0\). Thus \(t^{X}\) is strictly decreasing.
	
	\PN From \(g(y_0)=y(t^{X}(y_0);y_0)\) and the value of \(\dot y\) at the return,
	\[
	g'(y_0)
	=
	e^{\mu t}(\cos t+\mu\sin t)+\bigl(2\mu g(y_0)+1\bigr)(t^{X})'(y_0)
	=
	e^{\mu t}\frac{(\cos t-\mu\sin t)g(y_0)-\sin t}{g(y_0)}.
	\]
	Using the return condition
	\[
	(y_0+\beta_{\mu})\sin t
	=
	\zeta_{\mu}(\cos t+\mu\sin t-e^{-\mu t})
	\]
	together with \(\cos^2 t+\sin^2 t=1\), this reduces to
	\[
	g'(y_0)
	=
	e^{2\mu t}\frac{y_0}{g(y_0)}
	=
	\frac{y_0}{\mathcal E(y_0)g(y_0)}.
	\]
	Since \(y_0>0\), \(\mathcal E(y_0)>0\), and \(g(y_0)<0\), it follows that \(g'(y_0)<0\).
	Hence \(g\) is strictly decreasing and the stated identity for \(g'\) follows.
	
	Since \(\mathcal E=e^{-2\mu t^X}\) and \(t^X\) is strictly decreasing, \(\mathcal E\) is strictly increasing.
\end{proof}

\PN Monotonicity of the chord map alone gives no amplitude control near grazing or at large crossing data. Quantitative bounds provide the missing estimates for stability and for excluding two-cycles of \(F\).

\begin{lemma}\label{lem:chord_amplitude}
	\PN For \(Z\in\CLzero\) and \(y\geq0\),
	\[
	\frac{-g(y)}{Q_{\mu}(y)}
	\leq
	\frac{\beta_{\mu}}{2\zeta_{\mu}}
	+
	\frac{1}{\sqrt{H_{\mathrm g}(\mu)\zeta_{\mu}}},
	\qquad
	x_{\mathrm g}(\mu)
	=
	\frac{\beta_{\mu}}{2}
	+
	\sqrt{\zeta_{\mu}\bigl(e^{2\mu t_{\mathrm g}(\mu)}-\zeta_{\mu}\bigr)}
	>
	2\beta_{\mu}.
	\]
	Moreover, \(-g(y)>x_{\mathrm g}(\mu)\) for every \(y>0\).
\end{lemma}
\begin{proof}
\PN By \eqref{eq:factor_relations}, \(Q_{\mu}(g(y))=e^{2\mu t^{X}(y)}Q_{\mu}(y)\), and hence \(g(y)^2+\beta_{\mu}g(y)+\zeta_{\mu}-e^{2\mu t^{X}(y)}Q_{\mu}(y)=0\). Since \(g(y)<0\), the negative root gives \(-g(y)=\beta_{\mu}/2+\sqrt{e^{2\mu t^{X}(y)}Q_{\mu}(y)-\zeta_{\mu}^2}\). Therefore, using \(t^{X}(y)\leq t_{\mathrm g}(\mu)\) and \(Q_{\mu}(y)\geq\zeta_{\mu}\),
\[
\frac{-g(y)}{Q_{\mu}(y)}
\leq
\frac{\beta_{\mu}}{2\zeta_{\mu}}
+
\frac{e^{\mu t^{X}(y)}}{\sqrt{Q_{\mu}(y)}}
\leq
\frac{\beta_{\mu}}{2\zeta_{\mu}}
+
\frac{1}{\sqrt{H_{\mathrm g}(\mu)\zeta_{\mu}}}.
\]
\PN At \(y=0\), \(x_{\mathrm g}(\mu)=\beta_{\mu}/2+\sqrt{\zeta_{\mu}\bigl(e^{2\mu t_{\mathrm g}(\mu)}-\zeta_{\mu}\bigr)}\). Since \(t_{\mathrm g}(\mu)>\pi\), \(e^{2\mu t_{\mathrm g}(\mu)}>1+2\pi \mu\). Moreover, \(\mu\zeta_{\mu}=\mu/(1+\mu^2)\leq1/2\), so \(2\pi \mu>8\mu^2\zeta_{\mu}\). Hence
\[
\zeta_{\mu}\bigl(e^{2\mu t_{\mathrm g}(\mu)}-\zeta_{\mu}\bigr)
>
\zeta_{\mu}\bigl(1+2\pi \mu-\zeta_{\mu}\bigr)
>
9\mu^2\zeta_{\mu}^2
=
\left(\frac{3\beta_{\mu}}{2}\right)^2.
\]
Therefore \(x_{\mathrm g}(\mu)>2\beta_{\mu}\). Finally, the strict decrease of \(g\) gives \(-g(y)>-g(0)=x_{\mathrm g}(\mu)\) for every \(y>0\).
\end{proof}


\section{Classification of simple-period crossing orbits}\label{sec:parametrization}

\PN To classify the symmetric family for prescribed \(H\), it remains to determine the image and injectivity of \(\eta\).

\subsection{Parameter range of the symmetric family}
\PN Fix \(\mu>0\) and \(t\in(\pi,t_{\mathrm g}(\mu))\), set \(H=\eta(t)\), and let \(p_0=p_0(t)\) be the upward crossing point of the corresponding symmetric periodic orbit. Write
\[
N:=DF(p_0),\qquad \mathcal E:=\mathcal E(y_0(t))=e^{-2\mu t},\qquad \rho:=\frac{y_0}{x_0}.
\]
Its determinant and trace determine the Schur inequalities for stability and link the half-period parametrization to the spectrum of the return map.

\begin{proposition}\label{prop:invariants}
	\PN For \(\mu>0\) and \(t\in(\pi,t_{\mathrm g}(\mu))\), with \(H=\eta(t)\), the determinant and trace of the reduced-map derivative \(N\) satisfy
	\begin{equation}\label{eq:det_N}
		\det N=\rho,
	\end{equation}
	and
	\begin{equation}\label{eq:trace_N}
		\operatorname{tr}N
		=
		H\left(\frac{\rho}{\mathcal E}+\mathcal E\right)
		-
		\beta_{\mu}(1+\rho)\,
		\frac{x_0-Hy_0}{Q_{\mu}(-x_0)}.
	\end{equation}
\end{proposition}
\begin{proof}
	\PN By Lemma~\ref{lem:PX_closed_form}, the derivative \(D\Pi_X\) is upper triangular, with diagonal entries \(\mathcal E\) and \(g'(y_0)\). At the fixed point \(p_0=(x_0,y_0,0)\), Lemma~\ref{lem:closure_identity} gives
	\[
	g(y_0)=-x_0.
	\]
	Hence Lemma~\ref{lem:chord_monotone} yields
	\[
	g'(y_0)
	=
	\frac{y_0}{\mathcal E g(y_0)}
	=
	-\frac{\rho}{\mathcal E}.
	\]
	\PN Since
	\[
	D\Pi_X=
	\begin{pmatrix}
		\mathcal E&\dfrac{\partial x_1}{\partial y_0}\\
		0&g'
	\end{pmatrix},
	\qquad
	S|_\Sigma=
	\begin{pmatrix}
		0&-1\\
		-1&0
	\end{pmatrix},
	\]
	one has
	\[
	N=S|_\Sigma D\Pi_X=
	\begin{pmatrix}
		0&-g'\\
		-\mathcal E&-\dfrac{\partial x_1}{\partial y_0}
	\end{pmatrix}.
	\]
	Therefore
	\[
	\det N=-\mathcal E g'=\rho,
	\qquad
	\operatorname{tr}N=-\frac{\partial x_1}{\partial y_0}.
	\]
	The first identity proves \eqref{eq:det_N}.
	
	\PN It remains to compute the trace. Differentiating the first component of \eqref{eq:PX_closed_form} with respect to \(y_0\), with \(x_0\) fixed, gives
	\[
	\frac{\partial x_1}{\partial y_0}
	=
	Hg'
	+
	\frac{Q_{\mu}'(y_0)(x_0-Hy_0)-HQ_{\mu}(y_0)}
	{Q_{\mu}(y_1)}
	-
	\frac{Q_{\mu}(y_0)(x_0-Hy_0)Q_{\mu}'(y_1)g'}
	{Q_{\mu}(y_1)^2}.
	\]
	At the fixed point, \(y_1=-x_0\). Using
	\[
	\frac{Q_{\mu}(y_0)}{Q_{\mu}(-x_0)}=\mathcal E,
	\qquad
	g'=-\frac{\rho}{\mathcal E},
	\]
	the preceding derivative becomes
	\[
	\frac{\partial x_1}{\partial y_0}
	=
	-\frac{H\rho}{\mathcal E}
	-H\mathcal E
	+
	\frac{x_0-Hy_0}{Q_{\mu}(-x_0)}
	\left(
	Q_{\mu}'(y_0)+\rho Q_{\mu}'(-x_0)
	\right).
	\]
	Since
	\[
	Q_{\mu}'(y_0)+\rho Q_{\mu}'(-x_0)
	=
	\beta_{\mu}(1+\rho),
	\]
	and \(\operatorname{tr}N=-\partial x_1/\partial y_0\),
	equation~\eqref{eq:trace_N} follows.
\end{proof}
\PN By Lemma~\ref{lem:param_properties}, \(0<y_0<x_0\), and hence \(0<\rho<1\). Therefore, by \eqref{eq:det_N}, the Schur criterion \cite[p.~247]{Elaydi2005} for a real \(2\times2\) matrix gives
\[
\operatorname{spec}(N)\subset\{z\in\mathbb C:|z|<1\}
\quad\Longleftrightarrow\quad
\det(I-N)>0,
\qquad
\det(I+N)>0.
\]
By Proposition~\ref{prop:F_square}, \(D\Pi(p_0)=N^2\). Hence hyperbolicity and orbital asymptotic stability follow once the two Schur inequalities above are established.

\begin{proposition}\label{teo:inequalities}
	\PN For \(\mu>0\) and \(t\in(\pi,t_{\mathrm g}(\mu))\),
	\begin{align}
		\det(I-N)&=1+\rho-\operatorname{tr}N>0, \tag{S1}\label{ineq:IminusN_positive}\\
		\det(I+N)&=1+\rho+\operatorname{tr}N>0. \tag{S2}\label{ineq:IplusN_positive}
	\end{align}
\end{proposition}
\begin{proof}
	
	\PN For any \(2\times2\) matrix \(N\),
	\[
	\det(I-N)=1-\operatorname{tr}N+\det N,
	\qquad
	\det(I+N)=1+\operatorname{tr}N+\det N.
	\]
	Hence, by \eqref{eq:det_N},
	\[
	\det(I-N)=1+\rho-\operatorname{tr}N,
	\qquad
	\det(I+N)=1+\rho+\operatorname{tr}N.
	\]
	
	\PN The relation
	\[
	\mathcal E(x_0-Hy_0)=Hx_0-y_0
	\]
	gives
	\[
	H=
	\frac{\mathcal E x_0+y_0}{x_0+\mathcal E y_0},
	\qquad
	x_0-Hy_0=
	\frac{x_0^2-y_0^2}{x_0+\mathcal E y_0}.
	\]
	Using \eqref{eq:det_N} and \eqref{eq:trace_N},
	\[
	\begin{aligned}
		\det(I-N)
		&=
		1+\rho-\operatorname{tr}N\\
		&=
		1+\rho
		-
		H\left(\frac{\rho}{\mathcal E}+\mathcal E\right)
		+
		\beta_{\mu}(1+\rho)
		\frac{x_0-Hy_0}{Q_{\mu}(-x_0)}.
	\end{aligned}
	\]
	Since \(\rho=y_0/x_0\), substitution of the two identities above yields
	\[
	\det(I-N)
	=
	\frac{(1-\mathcal E^2)
		(\mathcal E x_0^2-y_0^2)}
	{\mathcal E x_0(x_0+\mathcal E y_0)}
	+
	\frac{\beta_{\mu}(x_0-y_0)(x_0+y_0)^2}
	{x_0(x_0+\mathcal E y_0)Q_{\mu}(-x_0)}.
	\]
	
	\PN Since \(g(y_0)=-x_0\), \eqref{eq:E_definition} gives
	\[
	\mathcal E=\frac{Q_{\mu}(y_0)}{Q_{\mu}(-x_0)}.
	\]
	Hence
	\[
	\mathcal E x_0^2-y_0^2
	=
	\frac{(x_0+y_0)
		\bigl[\beta_{\mu}x_0y_0+\zeta_{\mu}(x_0-y_0)\bigr]}
	{Q_{\mu}(-x_0)}.
	\]
	By Lemma~\ref{lem:param_properties}, \(y_0>0\) and \(x_0-y_0>\beta_{\mu}\), so \(x_0>y_0\) and the right-hand side is positive. Since \(\mathcal E\in(0,1)\), both terms in the preceding expression for \(\det(I-N)\) are positive, and therefore
	\[
	\det(I-N)>0.
	\]
	This proves \eqref{ineq:IminusN_positive}.
	
	\PN Likewise, \eqref{eq:det_N} and \eqref{eq:trace_N} give
	\[
	\begin{aligned}
		\det(I+N)
		&=
		1+\rho+\operatorname{tr}N\\
		&=
		1+\rho
		+
		H\left(\frac{\rho}{\mathcal E}+\mathcal E\right)
		-
		\beta_{\mu}(1+\rho)
		\frac{x_0-Hy_0}{Q_{\mu}(-x_0)}.
	\end{aligned}
	\]
	Substituting the same expressions for \(H\) and \(x_0-Hy_0\) gives
	\[
	\det(I+N)
	=
	\frac{
		(1+\mathcal E^2)(\mathcal E x_0^2+y_0^2)
		+
		2\mathcal E(1+\mathcal E)x_0y_0
	}
	{\mathcal E x_0(x_0+\mathcal E y_0)}
	-
	\frac{\beta_{\mu}(x_0-y_0)(x_0+y_0)^2}
	{x_0(x_0+\mathcal E y_0)Q_{\mu}(-x_0)}.
	\]
	Since \(g(y_0)=-x_0\) and \(y_0>0\),
	Lemma~\ref{lem:chord_amplitude} gives
	\[
	x_0>x_{\mathrm g}>2\beta_{\mu},
	\]
	and hence
	\[
	Q_{\mu}(-x_0)>\beta_{\mu}x_0.
	\]
	Together with \(0<x_0-y_0<x_0\), this implies
	\[
	\frac{\beta_{\mu}(x_0-y_0)}{Q_{\mu}(-x_0)}<1.
	\]
	Consequently,
	\begin{align*}
		\det(I+N)
		&>
		\frac{
			(1+\mathcal E^2)(\mathcal E x_0^2+y_0^2)
			+
			2\mathcal E(1+\mathcal E)x_0y_0
			-
			\mathcal E(x_0+y_0)^2
		}
		{\mathcal E x_0(x_0+\mathcal E y_0)}
		\\
		&=
		\frac{
			\mathcal E^3x_0^2
			+
			2\mathcal E^2x_0y_0
			+
			(\mathcal E^2-\mathcal E+1)y_0^2
		}
		{\mathcal E x_0(x_0+\mathcal E y_0)}
		>0.
	\end{align*}
	Thus \eqref{ineq:IplusN_positive} follows.
\end{proof}

\PN The first Schur inequality condition~\eqref{ineq:IminusN_positive},  in Proposition above, excludes the multiplier \(+1\). The same quantity controls the sign of \(\eta'(t)\), connecting spectral stability to the monotonicity needed for a one-to-one parametrization.

\begin{lemma}\label{lem:eta_monotone}
	\PN For each fixed \(\mu>0\), the function \(\eta\) is strictly decreasing on \((\pi,t_{\mathrm g}(\mu))\).
\end{lemma}

\begin{proof}
	\PN Since \(y_0(t)=\zeta_{\mu}\bigl(\Phi_{\mu}(t)-\mu\bigr)\), Lemma~\ref{lem:Phi_properties} gives \(y_0'(t)<0\) on \((\pi,t_{\mathrm g}(\mu))\), while Lemma~\ref{lem:param_properties} shows that \(y_0\) maps this interval onto \((0,\infty)\). Hence the family may be reparametrized by \(y=y_0(t)\). Let \(t=t(y)\) denote the inverse parametrization and set \(x(y):=x_0(t(y))\), \(\eta(y):=\eta(t(y))\), and \(p(y):=(x(y),y)\). By Proposition~\ref{prop:symmetric_family},
	\[
	F\bigl(p(y),\eta(y)\bigr)=p(y).
	\]
	Differentiating with respect to \(y\) gives
	\begin{equation}\label{eq:turning_point_system}
		(I-N)
		\begin{pmatrix}
			x'(y)\\
			1
		\end{pmatrix}
		=
		\frac{d\eta}{dy}\,\partial_HF.
	\end{equation}
	
	\PN Since \(g\) and \(\mathcal E\) do not depend on \(H\), formula~\eqref{eq:F_explicit}, together with \(g(y)=-x\) at the fixed point, gives \(\partial_HF=(0,x+\mathcal E y)^{\top}\), with \(x+\mathcal E y>0\). Moreover, the matrix formula in the proof of Proposition~\ref{prop:invariants} and Lemma~\ref{lem:chord_monotone} give \(-g'=y/(\mathcal E x)\). The first component of \eqref{eq:turning_point_system} therefore yields \(x'(y)=y/(\mathcal E x)\). Substitution into the second component gives
	\begin{equation}\label{eq:turning_point_identity}
		\frac{d\eta}{dy}\bigl(x+\mathcal E y\bigr)=1+\rho-\operatorname{tr}N=\det(I-N).
	\end{equation}
	By \eqref{ineq:IminusN_positive}, \(d\eta/dy>0\). Since \(y_0'(t)<0\),
	\[
	\eta'(t)=\frac{d\eta}{dy}\,y_0'(t)<0.
	\]
	Thus \(\eta\) is strictly decreasing.
\end{proof}

\PN As \(t\downarrow\pi\), the crossing data become unbounded, whereas as \(t\uparrow t_{\mathrm g}(\mu)\), the parametrization approaches the grazing configuration of Lemma~\ref{lem:grazing_data}.

\begin{lemma}\label{lem:edge_values}
	\PN For the half-period parametrization \eqref{eq:xy_of_t}--\eqref{eq:H_of_t}, let \(k:=e^{-\pi \mu}\). As \(t\downarrow\pi\), one has \(x_0(t),y_0(t)\to+\infty\), \(y_0(t)/x_0(t)\to k\), and \(\eta(t)\to H_{\mathrm{crit}}(\mu)\). As \(t\uparrow t_{\mathrm g}(\mu)\), one has \(y_0(t)\to0\), \(x_0(t)\to x_{\mathrm g}(\mu)\), and \(\eta(t)\to H_{\mathrm g}(\mu)\).
\end{lemma}

\begin{proof}
	\PN Let \(t=\pi+\varpi\), with \(\varpi\downarrow0\). From \eqref{eq:xy_of_t},
	\[
	y_0(t)=\frac{\zeta_{\mu}(1+k)}{\varpi}+O(1),
	\qquad
	x_0(t)=\frac{\zeta_{\mu}(1+k^{-1})}{\varpi}+O(1),
	\qquad
	\frac{y_0(t)}{x_0(t)}\to k.
	\]
	Substitution into \eqref{eq:H_of_t} gives \(\eta(t)\to k/(1-k+k^2)=H_{\mathrm{crit}}(\mu)\).
	
	\PN As \(t\uparrow t_{\mathrm g}(\mu)\), Lemma~\ref{lem:param_properties} gives \(y_0(t)\to0\). Using \(\Phi_{\mu}(t_{\mathrm g}(\mu))=\mu\) in \eqref{eq:xy_of_t} gives \(x_0(t)\to-e^{\mu t_{\mathrm g}}\sin t_{\mathrm g}\). On the other hand, \eqref{eq:z_planar_solution} at \(y_0=0\) gives \(g(0)=e^{\mu t_{\mathrm g}}\sin t_{\mathrm g}\). Hence \(x_0(t)\to-g(0)=x_{\mathrm g}(\mu)\), and \eqref{eq:H_of_t} together with Lemma~\ref{lem:grazing_data} yields \(\eta(t)\to H_{\mathrm g}(\mu)\).
\end{proof}

\PN The strict monotonicity of \(\eta\) and the endpoint limits above now determine the parameter range and complete the half-period parametrization.

\begin{proposition}\label{prop:eta_bijection}
	\PN For each fixed \(\mu>0\), the function \(\eta\) defined by~\eqref{eq:H_of_t} is a strictly decreasing bijection
	\[
	\eta:(\pi,t_{\mathrm g}(\mu))
	\longrightarrow
	\bigl(H_{\mathrm g}(\mu),H_{\mathrm{crit}}(\mu)\bigr).
	\]
	In particular,
	\[
	0<H_{\mathrm g}(\mu)<H_{\mathrm{crit}}(\mu)<1.
	\]
\end{proposition}
\begin{proof}
	\PN Lemma~\ref{lem:eta_monotone} shows that \(\eta\) is strictly decreasing on \((\pi,t_{\mathrm g}(\mu))\). Moreover, the explicit formula~\eqref{eq:H_of_t}, together with the positivity of its denominator from Lemma~\ref{lem:param_properties}, shows that \(\eta\) is continuous. By Lemma~\ref{lem:edge_values},
	\[
	\lim_{t\downarrow\pi}\eta(t)=H_{\mathrm{crit}}(\mu),
	\qquad
	\lim_{t\uparrow t_{\mathrm g}(\mu)}\eta(t)=H_{\mathrm g}(\mu).
	\]
	Hence \(H_{\mathrm g}(\mu)<H_{\mathrm{crit}}(\mu)\) and
	\[
	\eta\bigl((\pi,t_{\mathrm g}(\mu))\bigr)
	=
	\bigl(H_{\mathrm g}(\mu),H_{\mathrm{crit}}(\mu)\bigr).
	\]
	Strict decrease gives injectivity, while the displayed image gives surjectivity. Finally, \eqref{eq:H_g_definition} and \eqref{eq:H_crit_definition} show that both endpoints belong to \((0,1)\).
\end{proof}
\subsection{Exclusion of asymmetric periodic orbits}\label{subsec:no_twocycle}

\PN The preceding subsection completely determines the fixed points of \(F\) that generate the symmetric simple-period family. The full periodicity problem, however, is initially
\[
\Pi(p)=p,
\]
or equivalently \(F^2(p)=p\) on the corresponding crossing domains. Thus the symmetric classification is complete only if \(F^2(p)=p\) implies \(F(p)=p\).

\PN Indeed, let \(\gamma\) be a crossing periodic orbit of simple period, with consecutive crossings \(p_0\in\mathcal D\) and \(p_1\in S(\mathcal D)\), and set \(q_0:=S(p_1)=F(p_0)\). By equivariance, the return from \(p_1\) to \(p_0\) gives \(F(q_0)=p_0\), and hence \(F^2(p_0)=p_0\). Conversely, a two-cycle of \(F\) produces, by equivariance, the corresponding simple-period crossing orbit. Moreover, Proposition~\ref{prop:symmetric_closure} gives
\[
\gamma\text{ symmetric}
\quad\Longleftrightarrow\quad
F(p_0)=p_0.
\]
Thus asymmetric simple-period periodic orbits correspond to two-cycles of \(F\). It remains to show that such two-cycles do not exist.

\PN To exclude two-cycles of \(F\), the two corresponding crossing equations will be compared after separating the flight factor from the chord map. For \(y\geq0\), let \(g\) be the planar chord map of Definition~\ref{def:chord_map} and set
\[
h(y):=-\frac{g(y)}{\mathcal E(y)},
\qquad
m(y):=-g(y)+\mathcal E(y)y.
\]
Since \(t^X\) and \(g\) are real-analytic and \(Q_{\mu}>0\), the functions \(\mathcal E\), \(h\), and \(m\) are real-analytic on \((0,\infty)\). The factor relation \eqref{eq:factor_relations} gives the identity for \(\mathcal E\). The functions \(m\) and \(h\) are chosen so that subtraction of the two-cycle equations separates their dependence on \(H\) into a difference of \(m\) and the remaining crossed terms into a difference of \(h\). This reduction allows the mean value theorem to convert the two-cycle condition into a comparison between \(m'\) and \(h'\). Since \(g(y)<0\), one also has \(h(y)>0\). Equation~\eqref{eq:F_explicit} gives the corresponding reduced map without changing its crossing domain.

\PN By Lemma~\ref{lem:chord_monotone}, \(g'(y)<0\) and \((t^X)'(y)<0\) for \(y>0\). Since \(\mathcal E(y)=e^{-2\mu t^X(y)}\), one has \(\mathcal E'(y)>0\). Therefore \(m'(y)=-g'(y)+\mathcal E'(y)y+\mathcal E(y)>0\), so \(m\) is strictly increasing on \((0,\infty)\).

\PN The monotonicity above and the quantitative bounds of Lemma~\ref{lem:chord_amplitude} provide the estimate for \(h'\). The next lemma supplements this estimate with the uniform scalar bound needed to close the two-cycle argument.

\begin{lemma}\label{lem:twocycle_estimates}
	\PN For \(Z\in\CLzero\) and every \(y>0\),
	\[
	\mathcal E(y)(-h'(y))
	<
	\beta_{\mu}\frac{-g(y)}{Q_{\mu}(y)}.
	\]
	Moreover, if
	\[
	\mathcal K(\mu)
	:=
	2\mu^2\zeta_{\mu}e^{-2\pi \mu}
	+
	\frac{2\mu}{\sqrt{1+\mu^2}}
	e^{-\mu(2\pi-t_{\mathrm g}(\mu))},
	\]
	then \(\mathcal K(\mu)<1\) for every \(\mu>0\).
\end{lemma}
\begin{proof}
\PN Since \(h(y)=-g(y)Q_{\mu}(g(y))/Q_{\mu}(y)\), logarithmic differentiation, together with Lemma~\ref{lem:chord_monotone}, gives
\[
\frac{h'(y)}{h(y)}
=
\frac{1}{Q_{\mu}(y)}
\left[
y\left(
1+\frac{2\beta_{\mu}}{g(y)}
+\frac{\zeta_{\mu}}{g(y)^2}
\right)
-\beta_{\mu}
\right].
\]
By Lemma~\ref{lem:chord_amplitude}, \(-g(y)>2\beta_{\mu}\), and hence \(g(y)^2+2\beta_{\mu}g(y)+\zeta_{\mu}>0\). Therefore the term multiplied by \(y\) above is positive, and
\[
-h'(y)
=
\frac{h(y)}{Q_{\mu}(y)}
\left[
\beta_{\mu}
-
y\left(
1+\frac{2\beta_{\mu}}{g(y)}
+\frac{\zeta_{\mu}}{g(y)^2}
\right)
\right]
<
\beta_{\mu}\frac{h(y)}{Q_{\mu}(y)}.
\]
Multiplying by \(\mathcal E(y)>0\) and using \(\mathcal E(y)h(y)=-g(y)\) gives the first inequality stated in Lemma~\ref{lem:twocycle_estimates}.

\PN It remains to prove the bound for \(\mathcal K(\mu)\). Set \(\theta_{\mathrm g}=2\pi-t_{\mathrm g}(\mu)\in(0,\pi)\). The defining equation for \(t_{\mathrm g}(\mu)\) becomes
\[
\cos\theta_{\mathrm g}+\mu\sin\theta_{\mathrm g}
=
e^{-\mu(2\pi-\theta_{\mathrm g})}.
\]
Set \(\theta_0:=2\arctan \mu\) and \(\psi(\theta):=\cos\theta+\mu\sin\theta-e^{-\mu(2\pi-\theta)}\). Since \(\cos\theta_0+\mu\sin\theta_0=1\),
\[
\psi(\theta_0)
=
1-e^{-\mu(2\pi-\theta_0)}
>0,
\qquad
\psi(\pi)
=
-1-e^{-\pi \mu}
<0.
\]
By uniqueness of the grazing root, \(\theta_{\mathrm g}>\theta_0=2\arctan \mu\). The first term of \(\mathcal K(\mu)\) satisfies
\[
2\mu^2\zeta_{\mu}e^{-2\pi \mu}
=
\frac{2\mu^2}{1+\mu^2}e^{-2\pi \mu}
<
2\mu^2e^{-2\pi \mu}
\leq
\frac{2}{\pi^2e^2}
<
\frac1{18}.
\]
For the second term, using \(\theta_{\mathrm g}>2\arctan \mu\), \(\arctan \mu\geq \mu/(1+\mu^2)\), and \(r:=\mu/\sqrt{1+\mu^2}\in(0,1)\), this gives
\[
\frac{2\mu}{\sqrt{1+\mu^2}}e^{-\mu\theta_{\mathrm g}}
<
\frac{2\mu}{\sqrt{1+\mu^2}}e^{-2\mu\arctan \mu}
\leq
2re^{-2r^2}
\leq
e^{-1/2}
<
\frac23,
\]
where \(2re^{-2r^2}\) attains its maximum at \(r=1/2\). Therefore
\[
\mathcal K(\mu)
<
\frac1{18}+\frac23
<1.
\]
\end{proof}
\PN The bounds in Lemma~\ref{lem:twocycle_estimates} make the balance equations for a nontrivial two-cycle incompatible. This rules out the remaining asymmetric simple-period alternative.

\begin{proposition}\label{prop:all_simple_symmetric}
\PN For every \(Z\in\CLzero\), the reduced map \(F\) has no periodic orbit of minimal period two. Consequently, every crossing periodic orbit of \(Z\) of simple period is symmetric.
\end{proposition}
\begin{proof}
\PN Suppose that \(F\) has a periodic orbit of minimal period two,
\[
F(p_0)=q_0,
\qquad
F(q_0)=p_0,
\]
with \(p_0,q_0\in\mathcal D\). Let \(a>0\) and \(b>0\) be their second coordinates. The first components give
\[
p_0=(-g(b),a,0),
\qquad
q_0=(-g(a),b,0),
\]
while the second components give
\begin{equation}\label{eq:twocycle_system}
\begin{aligned}
b&=-Hg(a)-\mathcal E(a)\bigl(-g(b)-Ha\bigr),\\
a&=-Hg(b)-\mathcal E(b)\bigl(-g(a)-Hb\bigr).
\end{aligned}
\end{equation}
If \(a=b\), then \(p_0=q_0\). After interchanging the points if necessary, assume \(a<b\).

\PN By the definition of \(m\), system~\eqref{eq:twocycle_system} becomes
\[
Hm(a)=b-\mathcal E(a)g(b),
\qquad
Hm(b)=a-\mathcal E(b)g(a).
\]
Using \(h=-g/\mathcal E\) and subtracting,
\[
(b-a)
+
\mathcal E(a)\mathcal E(b)\bigl(h(b)-h(a)\bigr)
=
-H\bigl(m(b)-m(a)\bigr).
\]
Applying the mean value theorem separately to \(m\) and \(h\), there exist
\(\xi_1,\xi_2\in(a,b)\) such that
\[
m(b)-m(a)=m'(\xi_1)(b-a),
\qquad
h(b)-h(a)=h'(\xi_2)(b-a).
\]
Substituting these identities into the preceding equation and dividing by
\(b-a>0\) gives
\begin{equation}\label{eq:twocycle_balance}
\mathcal E(a)\mathcal E(b)\bigl(-h'(\xi_2)\bigr)
=
1+Hm'(\xi_1)>1,
\end{equation}
since \(H>0\) and \(m'>0\).

\PN Since \(\mathcal E\) is strictly increasing and \(a<\xi_2\),
\[
0<
\frac{\mathcal E(a)}{\mathcal E(\xi_2)}
<1.
\]
Moreover, \(b>0\) and Lemma~\ref{lem:first_return} give \(t^X(b)>\pi,\) and \(\mathcal E(b)<e^{-2\pi \mu}.\)
Using Lemma~\ref{lem:twocycle_estimates},
\[
\begin{aligned}
\mathcal E(a)\mathcal E(b)\bigl(-h'(\xi_2)\bigr)
=
\frac{\mathcal E(a)}{\mathcal E(\xi_2)}
\mathcal E(b)
\mathcal E(\xi_2)\bigl(-h'(\xi_2)\bigr)&<
\frac{\mathcal E(a)}{\mathcal E(\xi_2)}
\mathcal E(b)\beta_{\mu}
\frac{-g(\xi_2)}{Q_{\mu}(\xi_2)}\\
&<
\mathcal E(b)\beta_{\mu}
\frac{-g(\xi_2)}{Q_{\mu}(\xi_2)}.
\end{aligned}
\]
Lemma~\ref{lem:chord_amplitude} therefore yields
\[
\mathcal E(a)\mathcal E(b)\bigl(-h'(\xi_2)\bigr)
<
e^{-2\pi \mu}\beta_{\mu}
\left(
\frac{\beta_{\mu}}{2\zeta_{\mu}}
+
\frac{1}{\sqrt{H_{\mathrm g}(\mu)\zeta_{\mu}}}
\right)=
\mathcal K(\mu)
<1,
\]
where the equality follows from \(\beta_{\mu}=2\mu\zeta_{\mu}\) and \(H_{\mathrm g}(\mu)=e^{-2\mu t_{\mathrm g}(\mu)}\), while the final inequality follows from Lemma~\ref{lem:twocycle_estimates}.
This contradicts \eqref{eq:twocycle_balance}. Hence \(F\) has no periodic orbit of minimal period two. The correspondence established above then implies that every crossing periodic orbit of simple period is symmetric.
\end{proof}
\PN Proposition~\ref{prop:all_simple_symmetric} excludes points of minimal period two of \(F\), so \(F^2(p)=p\) reduces to \(F(p)=p\) in the simple-period crossing class. The fixed-point parametrization therefore covers every orbit in that class and gives part~\textup{(a)} of Theorem~\ref{teo:main1}.

\begin{proof}[Proof of part~\textup{(a)} of Theorem~\ref{teo:main1}]
	\PN Proposition~\ref{prop:symmetric_family} gives, for every \(t\in(\pi,t_{\mathrm g}(\mu))\) with \(H=\eta(t)\), a symmetric crossing periodic orbit with crossing points \(p_0(t)\) and \(S(p_0(t))\) and period \(2t\). Proposition~\ref{prop:eta_bijection} shows that \(\eta:(\pi,t_{\mathrm g}(\mu))\to\mathcal I_{\mu}\) is a bijection, so each \(H\in\mathcal I_{\mu}\) determines one member of this symmetric family.
	
\PN Conversely, Proposition~\ref{prop:all_simple_symmetric} shows that every crossing periodic orbit of simple period is symmetric. Let \(p_0=(x_0,y_0,0)\in\mathcal D\) be its upward crossing point and set \(t=t^X(y_0)\). Symmetry gives \(\Pi_X(p_0)=S(p_0)\), hence \(g(y_0)=-x_0\). The planar return relations then give \(y_0=y_0(t)\) and \(x_0=x_0(t)\), while Lemma~\ref{lem:closure_identity} and \eqref{eq:H_of_t} give \(H=\eta(t)\). Thus every crossing periodic orbit of simple period belongs to the family of Proposition~\ref{prop:symmetric_family}. Since \(\eta\) is bijective, for each \(H\in\mathcal I_{\mu}\) there is a unique \(t=\eta^{-1}(H)\). The formulas in~\eqref{eq:xy_of_t} then determine a unique upward crossing point \(p_0(t)\), and uniqueness of the sidewise flows determines a unique periodic orbit.
\end{proof}


\section{Stability}\label{sec:stability}

\PN The exclusion of asymmetric simple-period cycles completes the classification, so the stability analysis can now be carried out along the fixed-point family of the reduced map \(F\). Fix \(\mu>0\) and \(t\in(\pi,t_{\mathrm g}(\mu))\), set \(H=\eta(t)\), and consider the corresponding crossing periodic orbit with crossing point \(p_0=p_0(t)\). Recall from~\eqref{eq:dif_reduced_map} that \(N(t)=DF(p_0(t))\). The argument \(t\) is suppressed whenever no ambiguity can arise.

\PN For a crossing periodic orbit of simple period, the two nontrivial Floquet multipliers are the eigenvalues of the derivative of the Poincar\'e return map \cite{guckenheimer1983}. By Proposition~\ref{prop:F_square},
\[
D\Pi(p_0)=N^2.
\]
Let \(\lambda_1\) and \(\lambda_2\) be the roots of the characteristic polynomial of \(N\). Then the two nontrivial Floquet multipliers are \(\lambda_1^2\) and \(\lambda_2^2\). Therefore orbital asymptotic stability follows once \(\operatorname{spec}(N)\subset\{z\in\mathbb C:|z|<1\}\).

\begin{proof}[Proof of part~\textup{(b)} of Theorem~\ref{teo:main1}]
\PN Equation~\eqref{eq:det_N} and Lemma~\ref{lem:param_properties} give \(0<\det N=\rho<1\), while Proposition~\ref{teo:inequalities} gives \(\det(I-N)>0\) and \(\det(I+N)>0\). Hence the Schur criterion yields \(\operatorname{spec}(N)\subset\{\lambda\in\mathbb C:|\lambda|<1\}\). Since \(D\Pi(p_0)=N^2\), the spectral mapping theorem shows that the two nontrivial Floquet multipliers are the squares of the eigenvalues of \(N\), and therefore both have modulus strictly smaller than \(1\). Thus \(\gamma_{(\mu,H)}\) is hyperbolic and orbitally asymptotically stable.
	
\PN It remains to prove isolation. Since \(D\Pi(p_0)\) is Schur, there exists an equivalent norm \(\|\cdot\|_*\) on the local section such that \(\|D\Pi(p_0)\|_*<1\). By continuity, a closed ball \(B_r:=\{p:\|p-p_0\|_*\le r\}\) can be chosen inside the domain of the local return map with
\[
\sup_{q\in B_r}\|D\Pi(q)\|_*<1.
\]
Convexity of \(B_r\), the mean-value inequality, and \(\Pi(p_0)=p_0\) give
\[
\|\Pi(p)-p_0\|_*
\le\left(\sup_{q\in B_r}\|D\Pi(q)\|_*\right)\|p-p_0\|_*
<\|p-p_0\|_*
\qquad(p\in B_r\setminus\{p_0\}).
\]
Thus \(\Pi(B_r)\subset B_r\), and the distance to \(p_0\) strictly decreases along every other orbit. Such an orbit cannot be periodic, so \(p_0\) is the only periodic point of \(\Pi\) in \(B_r\).

\PN Choose an open neighborhood \(V_0\subset B_r\) of \(p_0\) in \(\Sigma\). The two crossings of \(\gamma_{(\mu,H)}\) are transverse. After removing sufficiently small neighborhoods of the crossing points, the remaining closed subarcs of the two flight arcs are compact and disjoint from \(\Sigma\). Transverse crossing neighborhoods and finitely many sidewise flow boxes can therefore be chosen to form a tubular neighborhood \(\mathcal U\) of \(\gamma_{(\mu,H)}\) such that every periodic orbit contained in \(\mathcal U\) intersects \(V_0\), with successive intersections generated by \(\Pi\). If \(\widetilde\gamma\subset\mathcal U\) is periodic, then \(\Pi^m(p)=p\) for some \(p\in V_0\) and \(m\ge1\). Since \(V_0\subset B_r\), the preceding argument gives \(p=p_0\), and uniqueness of the sidewise flows yields \(\widetilde\gamma=\gamma_{(\mu,H)}\). Thus \(\gamma_{(\mu,H)}\) is isolated, completing the proof.
\end{proof}

\PN Proposition~\ref{prop:all_simple_symmetric} excludes asymmetric periodic orbits within the simple-period class, whereas the preceding isolation argument is local and excludes any other periodic orbit in a sufficiently small neighborhood of each classified cycle, independently of its period.

\subsection{Consequences of the reduced-map analysis}

\subsubsection{Spectral consequences}

\PN Equation~\eqref{eq:turning_point_identity}, obtained in the proof of Lemma~\ref{lem:eta_monotone}, gives a spectral interpretation of the half-period parametrization. Since \(y_0'(t)\neq0\) and \(x+\mathcal E y>0\),
\[
\eta'(t)=0
\quad\Longleftrightarrow\quad
\det(I-N)=0
\quad\Longleftrightarrow\quad
1\in\operatorname{spec}(N).
\]
Hence a turning point of \(H=\eta(t)\) corresponds to the reduced-map derivative acquiring the eigenvalue \(+1\).

\begin{corollary}\label{cor:local_bifurcation_exclusion}
\PN Fix \(\mu>0\) and \(t\in(\pi,t_{\mathrm g}(\mu))\), set \(H=\eta(t)\), and let \(\gamma_{(\mu,H)}\) be the corresponding crossing limit cycle. Then \(N=DF(p_0(t))\) has no eigenvalue on the unit circle, excluding in particular the \(+1\) turning-point condition and the \(-1\) period-doubling condition for the reduced map. Since \(D\Pi(p_0(t))=N^2\), no nontrivial Floquet multiplier lies on the unit circle, so the corresponding \(+1\), period-doubling, and Neimark--Sacker spectral conditions for the Poincar\'e map are excluded.
\end{corollary}
\begin{proof}
\PN The Schur stability of \(N\) was established in the proof of part~\textup{(b)} of Theorem~\ref{teo:main1}. Together with Proposition~\ref{prop:F_square}, which gives \(D\Pi(p_0(t))=N^2\), this yields the stated spectral exclusions.
\end{proof}

\subsubsection{Spectral behavior at the boundary}\label{subsec:edges}

\PN We now examine the spectra at both ends of the half-period interval. At the upper boundary \(t\downarrow\pi\), retain the factor \(k\) from the half-period limits. Lemma~\ref{lem:edge_values} gives
\[
\rho\to k,
\qquad
H\to H_{\mathrm{crit}}=\frac{k}{1-k+k^2},
\qquad
x_0\to+\infty.
\]
Since \(y_0/x_0=\rho\), it follows that
\[
\frac{x_0-Hy_0}{Q_{\mu}(-x_0)}\to0.
\]
Therefore \eqref{eq:trace_N} and \eqref{eq:det_N} give
\[
\operatorname{tr}N\to1+k,
\qquad
\det N\to k.
\]
The characteristic polynomial of \(N\) consequently converges to
\[
(\lambda-1)(\lambda-k).
\]
Hence the eigenvalues of \(N\) converge to \(1\) and \(k\), while those of \(N^2\) converge to \(1\) and \(k^2\). Thus the upper boundary \(H_{\mathrm{crit}}(\mu)\) is approached through the \(+1\) spectral boundary as the amplitude becomes unbounded.

\begin{proposition}\label{prop:grazing_edge}
\PN Fix \(\mu>0\). Along the classified family \(H=\eta(t)\), let \(N(t)=DF(p_0(t))\). As \(t\uparrow t_{\mathrm g}(\mu)\),
\[
\det N(t)\to0,
\qquad
\operatorname{tr}N(t)\to
T_0:=H_{\mathrm g}(\mu)\bigl(H_{\mathrm g}(\mu)-2\mu x_{\mathrm g}(\mu)\bigr),
\qquad
T_0\in(-1,1).
\]
Consequently, the eigenvalues of \(N(t)\) converge to \(0\) and \(T_0\).
\end{proposition}
\begin{proof}
\PN By Lemma~\ref{lem:edge_values}, \(y_0\to0\), \(x_0\to x_{\mathrm g}\), and \(H\to H_{\mathrm g}\), and hence \(\rho\to0\). Equation~\eqref{eq:det_N} therefore gives \(\det N(t)\to0\). Passing to the limit in \eqref{eq:trace_N} gives
\[
\operatorname{tr}N(t)
\to
H_{\mathrm g}^{2}
-
\frac{\beta_{\mu}x_{\mathrm g}}{Q_{\mu}(-x_{\mathrm g})}
=
H_{\mathrm g}\bigl(H_{\mathrm g}-2\mu x_{\mathrm g}\bigr)
=T_0,
\]
where \(\beta_{\mu}=2\mu\zeta_{\mu}\) and \(H_{\mathrm g}=\zeta_{\mu}/Q_{\mu}(-x_{\mathrm g})\).

\PN It remains to locate \(T_0\). By Lemma~\ref{lem:chord_amplitude}, \(x_{\mathrm g}>2\beta_{\mu}\), and therefore
\[
Q_{\mu}(-x_{\mathrm g})-\beta_{\mu}x_{\mathrm g}
=
x_{\mathrm g}^2-2\beta_{\mu}x_{\mathrm g}+\zeta_{\mu}
>0.
\]
Consequently,
\[
0<
2\mu H_{\mathrm g}x_{\mathrm g}
=
\frac{\beta_{\mu}x_{\mathrm g}}{Q_{\mu}(-x_{\mathrm g})}
<1.
\]
Since \(0<H_{\mathrm g}<1\), it follows that \(H_{\mathrm g}^2-1<T_0<H_{\mathrm g}^2\), and hence \(T_0\in(-1,1)\).

The characteristic polynomial of \(N(t)\) consequently converges to
\[
\lambda^2-T_0\lambda
=
\lambda(\lambda-T_0),
\]
so the eigenvalues of \(N(t)\) converge to \(0\) and \(T_0\).
\end{proof}

\PN Toward grazing, Proposition~\ref{prop:grazing_edge} gives eigenvalue limits \(0\) and \(T_0\), with \(|T_0|<1\). Since \(D\Pi(p_0(t))=N(t)^2\), the eigenvalues of the Poincar\'e derivative converge to \(0\) and \(T_0^2\), both strictly inside the unit disk. Thus no unit-circle spectral boundary is approached at grazing. At \(t=t_{\mathrm g}(\mu)\), however, the limiting orbit is tangent to \(\Sigma\) at \((x_{\mathrm g},0,0)\in L_X\) and \((0,-x_{\mathrm g},0)\in L_Y\), so the standard saltation matrix is no longer defined at these grazing contacts \cite{dieci2011,diBernardo2008}. The grazing boundary is therefore geometric rather than spectral.

\PN Corollary~\ref{cor:local_bifurcation_exclusion} excludes unit-circle spectrum at every interior cycle. The boundary limits show that the only unit-circle spectral boundary approached along the admissible family is \(+1\), at the infinite-amplitude limit \(t\downarrow\pi\).

\subsection{Persistence under sidewise \texorpdfstring{\(C^1\)}{C1} perturbations}\label{subsec:C1_persistence}

\PN Hyperbolicity of the classified crossing cycles also yields their local persistence under sidewise \(C^1\) perturbations. Related persistence results for crossing periodic solutions in two-zone piecewise-smooth systems are given in \cite{GouveiaLlibreNovaesPessoa2016}, while the product-space formulation for Filippov systems treats the two sidewise vector fields as independent perturbation components \cite{GomideTeixeira2020}. The corresponding continuation statement for the cycles considered here is the following.

\begin{corollary}\label{cor:C1_persistence}
	\PN Let \(Z\in\CLzero\) with \(H\in\mathcal I_{\mu}\), and let \(\gamma_{(\mu,H)}\) be the corresponding crossing limit cycle. For every sufficiently small sidewise \(C^1\) perturbation \(\widetilde Z=(\widetilde X,\widetilde Y)\) of \(Z\) on a compact neighborhood of \(\gamma_{(\mu,H)}\), with the switching manifold \(\Sigma\) kept fixed, there exists a unique crossing limit cycle \(\widetilde\gamma\) of simple period near \(\gamma_{(\mu,H)}\). Moreover, \(\widetilde\gamma\) is hyperbolic and orbitally asymptotically stable.
\end{corollary}
\begin{proof}
\PN Let \(p_0=(x_0,y_0,0)\) and \(p_1=S(p_0)=(-y_0,-x_0,0)\) be the crossing points of \(\gamma_{(\mu,H)}\). Since both crossings are transverse and the interiors of the two flight arcs are disjoint from \(\Sigma\), the implicit function theorem \cite{KrantzParks2002} applied to the corresponding hitting equations gives, for every sufficiently small sidewise \(C^1\) perturbation, \(C^1\) first-crossing maps
	\[
	\widetilde\Pi_X:V_0\longrightarrow V_1,
	\qquad
	\widetilde\Pi_Y:V_1\longrightarrow V_0,
	\qquad
	\widetilde\Pi_X\longrightarrow\Pi_X,
	\qquad
	\widetilde\Pi_Y\longrightarrow\Pi_Y
	\quad\text{in }C^1.
	\]
Their composition defines the perturbed Poincar\'e map
\[
\Pi_{\widetilde Z}
=
\widetilde\Pi_Y\circ\widetilde\Pi_X
:V_0\longrightarrow V_0,
\qquad
\Pi_{\widetilde Z}\longrightarrow\Pi
\quad\text{in }C^1.
\]

\PN By Theorem~\ref{teo:main1}\textup{(b)}, \(p_0\) is an attracting hyperbolic fixed point of \(\Pi\). Persistence of hyperbolic fixed points under \(C^1\) perturbations \cite{guckenheimer1983} therefore gives a unique fixed point \(\widetilde p_0\in V_0\) near \(p_0\), with \(\widetilde p_0\to p_0\) and \(D\Pi_{\widetilde Z}(\widetilde p_0)\to D\Pi(p_0)\). For sufficiently small perturbations, \(\widetilde p_0\) remains attracting and hyperbolic.
	
\PN Set \(\widetilde p_1:=\widetilde\Pi_X(\widetilde p_0)\). Since \(\widetilde\Pi_Y(\widetilde p_1)=\widetilde p_0\), the corresponding \(\widetilde X\)- and \(\widetilde Y\)-flights form a periodic orbit \(\widetilde\gamma\). The first-crossing construction preserves transversality and the absence of intermediate intersections with \(\Sigma\), so \(\widetilde\gamma\) has simple period. Since \(\widetilde p_0\) is attracting and hyperbolic, so is \(\widetilde\gamma\), and the local uniqueness of \(\widetilde p_0\) gives the uniqueness of the nearby crossing limit cycle.
\end{proof}

\section{Global convergence in the crossing--sliding domain}\label{sec-global-crossing}

\PN The Floquet analysis establishes local attraction of the classified cycle, but does not determine the extent of its basin. For the same unperturbed zero-divergence family, we first seek a Lyapunov estimate for the reduced return dynamics and then control trajectories with sliding under \eqref{eq:nonsingular_sliding}. The discrete invariance mechanism has classical precedents \cite{LaSalle1976,Elaydi2005}; the sector estimate and centered energy used below are constructed for this system.

\subsection{Chord geometry and the restoring sector}

\PN Fix \(\mu>0\) and \(H\in\mathcal I_\mu\). The reduced map already uses the chord map \(g\), the flight factor \(\mathcal E\), and \(m(y)=-g(y)+\mathcal E(y)y\) from the two-cycle analysis of Section~\ref{sec:parametrization}. For \(y>0\), let \(t(y)\) be the inverse of \(y_0(t)\), as in the proof of Lemma~\ref{lem:eta_monotone}. Since \(g(y)<0\), one has \(m(y)>0\). At the classified fixed point \((-g(y),y)\) with \(H=\eta(t(y))\), equation~\eqref{eq:F_explicit} gives
\[
\eta(t(y))=\frac{y-\mathcal E(y)g(y)}{m(y)}.
\]
The factor \(k\), already fixed by the half-period limits, is the common upper bound in the secant and energy estimates below. Proposition~\ref{prop:eta_bijection} and Lemma~\ref{lem:param_properties} show that \(y\mapsto\eta(t(y))\) is a strictly increasing bijection from \((0,\infty)\) onto \(\mathcal I_\mu\). With \(t=\eta^{-1}(H)\), its unique solution of \(\eta(t(y))=H\) is the already classified \(y_0(t)\); on the reduced section the fixed point is represented by
\[
(x_0(t),y_0(t))=(-g(y_0(t)),y_0(t)).
\]

\PN The weighted chord \(\chi(y):=-\mathcal E(y)g(y)\), for \(y\ge0\), controls the derivatives entering the restoring force. Its strict convexity and limiting slope give the bound used below in the sector estimate. The same bound also controls chord secants weighted by the flight factor at one endpoint.

\begin{lemma}\label{lem-global-chord-energy}
\PN For \(x,y>0\) with \(x\ne y\),
\[
0<\mathcal E(x)\frac{g(y)-g(x)}{x-y}<k,
\qquad 0<-\mathcal E(y)g'(y)<k.
\]
On the chord \(x=-g(y)\),
\[
\chi'(y)=\frac{\beta_\mu x^2+y(x^2+\zeta_\mu)}{xQ_\mu(-x)}>0,
\]
\[
\chi''(y)=
\frac{\zeta_\mu(x+y)[(x-y)(x^2+\zeta_\mu)+2\beta_\mu xy]}
{x^3Q_\mu(-x)Q_\mu(y)}>0,
\qquad 0<\chi'(y)<k.
\]

\end{lemma}
\begin{proof}

\PN Put \(x=-g(y)\). By \eqref{eq:E_definition}, \eqref{eq:factor_relations}, and Lemma~\ref{lem:first_return},
\[
\mathcal E(y)
=
\frac{Q_\mu(y)}{Q_\mu(-x)}
<
k^2
<
1.
\]
Therefore \(Q_\mu(-x)>Q_\mu(y)\). Since
\[
Q_\mu(-x)-Q_\mu(y)
=(x+y)(x-y-\beta_\mu),
\]
one has \(x-y>\beta_\mu\). Lemma~\ref{lem:chord_monotone} gives
\[
-g'(y)
=
\frac{yQ_\mu(-x)}{xQ_\mu(y)}.
\]
Differentiating \(\chi(y)=-\mathcal E(y)g(y)\) with this identity yields the two displayed derivative formulas. Their numerators are positive, so \(\chi\) is strictly increasing and strictly convex.

\PN The return equation in Lemma~\ref{lem:first_return} gives \(t^X(y)\downarrow\pi\) as \(y\to\infty\). Hence \(\mathcal E(y)\to k^2\), and the factor relation gives \(-g(y)/y\to1/k\). Thus \(\chi(y)/y\to k\). Strict convexity now implies \(0<\chi'(y)<k\) for every \(y>0\).

\PN It remains to compare the endpoint-weighted secants with those of \(\chi\). If \(x>y\), monotonicity of \(\mathcal E\) gives
\[
\mathcal E(x)[g(y)-g(x)]<\chi(x)-\chi(y)<k(x-y).
\]
If \(x<y\), the same argument with the endpoints reversed gives
\[
\mathcal E(x)[g(x)-g(y)]<\chi(y)-\chi(x)<k(y-x).
\]
At coincident endpoints, \(-\mathcal E(y)g'(y)<\chi'(y)<k\). Positivity follows from the strict increase of \(-g\).

\end{proof}

\PN The bound on \(\chi'\) controls one part of the restoring force. A comparison with the derivative of \(m\) is still needed to obtain a strict sector throughout \(\mathcal I_\mu\). The next lemma supplies this comparison without imposing an additional parameter condition.

\PN The restoring force is
\begin{equation}\label{eq-global-restoring}
f_H(y)=y+\chi(y)-Hm(y),\qquad y\ge0.
\end{equation}
For \(y>0\), the fixed-point identity above gives
\[
f_H(y)=m(y)\bigl(\eta(t(y))-H\bigr).
\]
Consequently, with \(t=\eta^{-1}(H)\),
\[
f_H(y)(y-y_0(t))>0\quad(y>0,\ y\ne y_0(t)),
\qquad f_H(0)=x_{\mathrm g}(H_{\mathrm g}-H)<0.
\]
Dependence of \(x_{\mathrm g}\) and \(H_{\mathrm g}\) on \(\mu\) is suppressed in this section.

\begin{lemma}\label{lem-global-sector-quotient}
\PN For \(y\ge0\), the quotient
\[
\mathscr S_k(y)=\frac{-2kg'(y)+\chi'(y)-1}{m'(y)}
\]
is strictly increasing and has limit \(H_{\mathrm{crit}}(\mu)\) as \(y\to\infty\). In particular, \(\mathscr S_k(y)<H_{\mathrm{crit}}(\mu)\) for every finite \(y\).
\end{lemma}
\begin{proof}
\PN By Lemma~\ref{lem:chord_monotone} and~\eqref{eq:E_definition},
\[
g'(y)=\frac{yQ_\mu(g(y))}{g(y)Q_\mu(y)},\qquad
\mathcal E'(y)=\frac{\beta_\mu[y-g(y)]}{[-g(y)]Q_\mu(g(y))}>0,
\]
and differentiation gives
\[
-g''(y)=\frac{\zeta_\mu Q_\mu(g(y))[g(y)^2-y^2]}
{[-g(y)]^3Q_\mu(y)^2}>0.
\]
In particular, \(m'=-g'+\mathcal E+y\mathcal E'>0\), including at \(y=0\).
Expanding~\eqref{eq:xy_of_t} as \(t\downarrow\pi\) gives
\[
-g(y)=\frac yk+\frac{\beta_\mu(1+k)}k+O(y^{-1}),\qquad
\mathcal E(y)=k^2-\frac{\beta_\mu k^2(1+k)}y+O(y^{-2}).
\]
Since \(-g'=-y/(\mathcal E g)\to1/k\), strict convexity of \(-g\) gives
\(0\le-g'<1/k\). Hence \(-g(y)-y/k\) decreases strictly to
\(\beta_\mu(1+k)/k\), so
\[
-g(y)>\frac{y+\beta_\mu(1+k)}k,\qquad
0\le y<-kg(y),\qquad \beta_\mu<\frac{-kg(y)-y}{1+k}.
\]
Also,
\[
(1-k)^2\zeta_\mu-k\beta_\mu^2
=4k\zeta_\mu\left[\sinh^2\!\left(\frac{\pi\mu}{2}\right)-\mu^2\zeta_\mu\right]>0,
\]
because \(\sinh(\pi\mu/2)>\pi\mu/2>\mu\sqrt{\zeta_\mu}\).

\PN We next prove the comparison
\(-2kg''+\chi''>H_{\mathrm{crit}}m''\).
Differentiating \(y\mathcal E(y)\) gives
\[
(y\mathcal E)''=
\frac{\beta_\mu[y-g(y)]^2[-2\beta_\mu yg(y)-\zeta_\mu(2g(y)+y)]}
{[-g(y)]^3Q_\mu(g(y))Q_\mu(y)}>0,
\]
so \(m''=-g''+(y\mathcal E)''>0\).
The derivative formula from Lemma~\ref{lem-global-chord-energy} reads
\[
\chi''(y)=
\frac{\zeta_\mu[y-g(y)][-(g(y)+y)(g(y)^2+\zeta_\mu)-2\beta_\mu yg(y)]}
{[-g(y)]^3Q_\mu(g(y))Q_\mu(y)},
\]
while
\[
-g''(y)=\frac{\zeta_\mu[y-g(y)]}{[-g(y)]^3Q_\mu(g(y))Q_\mu(y)}
       \frac{[-g(y)-y]Q_\mu(g(y))}{\mathcal E(y)}.
\]
The identity preceding the estimates is
\[
\begin{aligned}
&\frac{[-g(y)]^3Q_\mu(g(y))Q_\mu(y)(1-k+k^2)}{\zeta_\mu[y-g(y)]}
  [-2kg''+\chi''-H_{\mathrm{crit}}m'']\\
&\quad=\left[\frac{1-2k+2k^2}{k}+1-k+k^2\right]
 [-g(y)-y][g(y)^2+\zeta_\mu]\\
&\qquad{}+\beta_\mu\Bigl\{-\frac{1-2k+2k^2}{k}g(y)[g(y)+y]
 -2(1-k+k^2)yg(y)\\
&\hspace{44mm}{}+k[y-g(y)][2g(y)+y]\Bigr\}\\
&\qquad{}+\frac{2k\beta_\mu^2}{\zeta_\mu}yg(y)[y-g(y)]\\
&\qquad{}+k(1-2k+2k^2)[-g(y)-y]Q_\mu(g(y))
 \left[\frac1{\mathcal E(y)}-\frac1{k^2}\right].
\end{aligned}
\]
The final term is positive: \(1-2k+2k^2=k^2+(1-k)^2>0\),
\(-g(y)-y>0\), and \(\mathcal E<k^2\). It may therefore be dropped
to obtain a strict lower bound. Since \(yg(y)[y-g(y)]\le0\),
replacing \(k\beta_\mu^2/\zeta_\mu\) by its upper bound
\((1-k)^2\) lowers the expression again.
The coefficient of the remaining term linear in \(\beta_\mu\) is negative:
\[
\begin{aligned}
&-\frac{1-2k+2k^2}{k}g(y)[g(y)+y]
  -2(1-k+k^2)yg(y)\\
&\qquad{}+k[y-g(y)][2g(y)+y]\\
&\quad=-\frac{1-k}{k}[k^2+(1-k)^2+3k^3]g(y)^2\\
&\qquad\quad{}-[-kg(y)-y]
 \left[-\left(\frac{1-k^2}{k}+3k^2\right)g(y)+ky\right]<0.
\end{aligned}
\]
Thus replacing \(\beta_\mu\) by its upper bound
\([-kg(y)-y]/(1+k)\) lowers the expression once more. Finally,
the positive term
\(\bigl[(1-2k+2k^2)/k+1-k+k^2\bigr][-g(y)-y]\zeta_\mu\)
may be discarded. Collecting the remaining terms gives
\begin{equation}\label{eq-global-sector-curvature}
\begin{aligned}
&\frac{[-g(y)]^3Q_\mu(g(y))Q_\mu(y)(1-k+k^2)}{\zeta_\mu[y-g(y)]}
  [-2kg''+\chi''-H_{\mathrm{crit}}m'']\\
&\quad>\frac{y-g(y)}{k^2(1+k)}\Bigl\{
 [-g(y)][-kg(y)-y]\bigl[1-k+k^2+k^2(1-k)^2\bigr]\\
&\hspace{36mm}{}-yg(y)(1-k)\bigl[(1-k)^2(1+k)+2k^4\bigr]\\
&\hspace{36mm}{}+k^3y[-kg(y)-y]\Bigr\}>0.
\end{aligned}
\end{equation}
All three terms in braces are nonnegative, and the first is strictly positive
because \(-g(y)>0\) and \(-kg(y)-y>0\). This proves the comparison.

\PN As \(y\to\infty\), the chord identities and
Lemma~\ref{lem-global-chord-energy} give
\(-g'\to1/k\), \(\chi'\to k\), and \(m'\to k^{-1}+k^2\).
Since \(H_{\mathrm{crit}}=(1+k)/(k^{-1}+k^2)\), the strictly increasing
expression \(-2kg'+\chi'-H_{\mathrm{crit}}m'\) has limit \(1\).
It is therefore strictly less than \(1\) at every finite \(y\ge0\), which gives
\(\mathscr S_k<H_{\mathrm{crit}}\). Differentiating the quotient now gives
\[
\mathscr S_k'
=\frac{-2kg''+\chi''-\mathscr S_k m''}{m'}
>\frac{(H_{\mathrm{crit}}-\mathscr S_k)m''}{m'}>0.
\]
The derivative limits above also give
\[
\lim_{y\to\infty}\mathscr S_k(y)
=\frac{1+k}{k^{-1}+k^2}=H_{\mathrm{crit}}.
\]
All formulas extend to \(y=0\), where \(g(0)<0\),
\(\mathcal E(0)>0\), and \(m'(0)>0\).
\end{proof}

\PN Write \((x_0,y_0)=(x_0(t),y_0(t))\) for the fixed point and define
\[
c=H(1-k+k^2)<k,\qquad b_*=y_0-cx_0,\qquad
a_H(y)=y+cg(y)-b_*.
\]
Then \(a_H(y_0)=0\) and \(a_H'>1-c/k>0\). Lemma~\ref{lem-global-sector-quotient} gives
\begin{equation}\label{eq-global-sector}
\begin{aligned}
2a_H'-f_H'
&=\left(1-\frac ck\right)(1-\chi')
+\frac ck m'[H_{\mathrm{crit}}-\mathscr S_k]\\
&>\left(1-\frac ck\right)(1-k)>0.
\end{aligned}
\end{equation}
Both \(f_H\) and \(2a_H-f_H\) vanish at \(y_0\). By the fixed-point identity for \(y>0\) and \(f_H(0)<0\), \(f_H\) has the sign of \(y-y_0\) on \([0,\infty)\), while the positive derivative in \eqref{eq-global-sector} gives the same sign to \(2a_H-f_H\). Thus
\begin{equation}\label{eq-global-Phi}
\Phi_H(y):=f_H(y)[2a_H(y)-f_H(y)]>0\quad(y\ne y_0),
\qquad \Phi_H(y_0)=0.
\end{equation}
This strict sector converts the restoring force into a nonnegative loss.

\subsection{A centered energy and the reachable cone}

\PN To make the sector loss appear in a flight balance, define a primitive on \([x_{\mathrm g},\infty)\) by
\begin{equation}\label{eq-global-Theta}
\Theta_H'(-g(y))=2cy+2\mathcal E(y)[f_H(y)-a_H(y)],\qquad
\Theta_H(x_0)=0.
\end{equation}
The chord is a bijection onto this interval, so the definition determines a \(C^1\) function, which is \(C^2\) for \(x>x_{\mathrm g}\). No finite value of \(\Theta_H''(x_{\mathrm g})\) is required. The crossing energy is
\begin{equation}\label{eq-global-energy}
\mathcal V_H(x,y)=\Theta_H(x)+y^2-2(cx+b_*)y+y_0^2,
\qquad x\ge x_{\mathrm g},\quad y\in\mathbb R.
\end{equation}
Allowing negative \(y\) records a landing before sliding; it does not extend the crossing domain of \(F\).

\PN The energy must have a lower bound independently of the convergence it will prove. The chord asymptotes give
\[
\frac{f_H(y)}y\to(1+k)\left(1-\frac ck\right),\qquad
\frac{a_H(y)}y\to1-\frac ck,
\]
and hence, using \(y/[-g(y)]\to k\) in \eqref{eq-global-Theta},
\[
\frac{\Theta_H'(-g(y))}{-g(y)}
\longrightarrow 2[k^4+ck(1-k^2)].
\]
Since \(-g(y)\to\infty\), l'Hôpital's rule gives
\[
\frac{\Theta_H(x)}{x^2}\longrightarrow k^4+ck(1-k^2).
\]
Subtracting \(c^2\) from this limit gives
\[
k^4+ck(1-k^2)-c^2=(k-c)(k^3+c)>0.
\]
Completion of the square in \eqref{eq-global-energy} therefore yields constants \(\epsilon>0\), \(C\) such that
\begin{equation}\label{eq-global-coercivity}
\mathcal V_H(x,y)\ge(y-cx-b_*)^2+\epsilon x^2-C
\qquad(x\ge x_{\mathrm g}).
\end{equation}
Thus its sublevel sets are bounded and it is bounded below.

\begin{proposition}\label{prop-global-crossing-dissipation}
\PN Let \(\mu>0\), \(H\in\mathcal I_\mu\), and set
\[
\begin{aligned}
\Lambda_{\mu,H}&=2\mu Hk^2[\pi(1-k+k^2)-1],\\
d_{\mu,H}&=2\mu Hk^2[\pi(1-k+k^2)-1-2k]>0.
\end{aligned}
\]
For \(y>0\),
\begin{equation}\label{eq-global-weighted-curvature}
\frac12\Theta_H''(-g(y))>\mathcal E(y)^2+\Lambda_{\mu,H}.
\end{equation}
On the cone
\[
\mathscr C_H=\{(x,y):x\ge x_{\mathrm g},\ 0\le y\le Hx\},
\]
the first normalized landing satisfies
\begin{equation}\label{eq-global-dissipation}
\mathcal V_H(F(x,y))-\mathcal V_H(x,y)
\le-\Phi_H(y)-d_{\mu,H}[x+g(y)]^2.
\end{equation}
The landing coordinate may be negative. The inequality is strict outside \((x_0,y_0)\).
\end{proposition}
\begin{proof}
\PN On the chord \(x=-g(y)\), put \(\varrho=y/x\), \(R=\mathcal E'(y)x\), and \(\theta=t^X(y)-\pi\in(0,\pi)\). Differentiating the flight time gives
\[
R=2\mu\sqrt{\mathcal E}\sin\theta<2\mu k.
\]
The factor relation gives
\[
\mathcal E x^2-y^2
=\beta_\mu(\mathcal E x+y)+\zeta_\mu(1-\mathcal E)>0,
\]
so \(0<\varrho<\sqrt{\mathcal E}<k\). Since \(\mathcal E^2e^{4\mu\theta}=k^4\),
\[
k^4-\mathcal E^2
>4\mu\theta\mathcal E^2
>4\mu\sin\theta\mathcal E^2
>2\mathcal E\varrho R.
\]
This comparison retains the dependence of the flight factor and chord ratio on the same flight time.

\PN Set
\[
\mathcal B=(1-k+k^2)-\mathcal E^2
-(k-k^2)\varrho-R(1+2\mathcal E\varrho).
\]
The preceding bounds and \(1-k^2=2k\sinh(\pi\mu)>2\pi\mu k\) give
\[
\mathcal B>(1-k+k^2)(1-k^2)-2\mu k
>2\mu k[\pi(1-k+k^2)-1].
\]
Moreover, \(-g(y_0)<x_{\mathrm g}+y_0/k\) implies \(b_*>-cx_{\mathrm g}\). Differentiating \eqref{eq-global-Theta} and using \(\mathcal E x'=\varrho\) gives
\[
\frac12\Theta_H''(x)-\mathcal E^2
=\frac{H\mathcal B+2\mathcal E R+R(b_*+cx)/x}{x'}.
\]
Every term after \(H\mathcal B\) in the numerator is positive, and \(1/x'>k\); this proves \eqref{eq-global-weighted-curvature}. Finally,
\[
\pi(1-k+k^2)-1-2k
=\pi\left(k-\frac12-\frac1\pi\right)^2
+\frac{3\pi}{4}-2-\frac1\pi>0,
\]
because \(\pi>25/8\) makes the last three terms greater than \(19/800\). Thus \(d_{\mu,H}=\Lambda_{\mu,H}-4\mu Hk^3>0\).

\PN For a launch \((x,y)\) with \(x\ge-g(y)\), monotonicity of \(\mathcal E\) and \eqref{eq-global-weighted-curvature} give \(\Theta_H''(s)/2>\mathcal E(y)^2+\Lambda_{\mu,H}\) along the intervening segment. If \(x_{\mathrm g}\le x<-g(y)\), put \(v=(-g)^{-1}(s)\le y\) for an interior point \(s\in[x,-g(y)]\). Since \(\mathcal E'(r)<2\mu k/[-g(r)]\),
\[
\mathcal E(y)^2-\mathcal E(v)^2
\le\frac{4\mu k^3}{s}(y-v)
\le4\mu Hk^3.
\]
The last step uses \(s\ge x\) and the cone condition \(y\le Hx\). Hence in both cases \(\Theta_H''(s)/2>\mathcal E(y)^2+d_{\mu,H}\). Taylor's formula with integral remainder, with endpoint limits allowed by continuity of \(\Theta_H'\), yields
\[
\Theta_H(x)-\Theta_H(-g(y))-\Theta_H'(-g(y))[x+g(y)]
\ge[\mathcal E(y)^2+d_{\mu,H}][x+g(y)]^2.
\]
Using \eqref{eq:F_explicit}, direct expansion of \eqref{eq-global-energy} gives
\[
\begin{aligned}
\mathcal V_H(F(x,y))-\mathcal V_H(x,y)
={}&-\Phi_H(y)+\mathcal E(y)^2[x+g(y)]^2\\
&-\bigl[\Theta_H(x)-\Theta_H(-g(y))
-\Theta_H'(-g(y))[x+g(y)]\bigr].
\end{aligned}
\]
This proves \eqref{eq-global-dissipation}, including negative landings.

\end{proof}

\PN The preceding dissipation estimate extends the simple-period classification to the full crossing dynamics. Every forward crossing trajectory for which the successive return dynamics remains defined is forced toward the classified cycle. In particular, crossing periodic orbits with a higher number of intersections with \(\Sigma\) are excluded. The resulting global convergence is stated in the following theorem.

\begin{theorem}\label{thm-global-crossing}
\PN Let \(\mu>0\), \(H\in\mathcal I_\mu\), \(t=\eta^{-1}(H)\), and
\[
\mathcal D_\infty
=
\{p\in\mathcal D \mid F^n(p)\in\mathcal D
\text{ for every }n\ge0\}.
\]
Then
\[
F^n(p)\to(x_0(t),y_0(t))
\]
for every \(p\in\mathcal D_\infty\). Hence every forward trajectory represented in \(\mathcal D_\infty\) converges orbitally to the classified simple-period cycle. This cycle is the unique crossing-only periodic orbit. In particular, no crossing-only periodic orbit with a higher number of crossings exists. The same convergence holds for any chain of flights starting at \(x>0\), \(y\ge0\), for which all subsequent landing ordinates remain nonnegative. Such a chain has only finitely many zero-duration contacts with folds.
\end{theorem}
\begin{proof}
\PN We first show that every forward orbit generated by a point in
\(\mathcal D_\infty\) eventually enters the cone. Put \(w_n=x_n-Hy_n\). The landing formula gives
\[
w_{n+1}=(1-H^2)[-g(y_n)]+H\mathcal E(y_n)w_n.
\]
While \(w_n<0\), one has \(w_{n+1}-w_n>(1-H^2)x_{\mathrm g}>0\), since \(H\mathcal E<1\). Thus there is a first index \(N\) with \(w_N\ge0\). Since the orbit remains in \(\mathcal D_\infty\), its next landing is positive and satisfies \(0<y_{N+1}\le Hx_{N+1}\) and \(x_{N+1}=-g(y_N)\ge x_{\mathrm g}\). Subsequent nonnegative landings stay in \(\mathscr C_H\).

\PN For every \(n\ge N\), Proposition~\ref{prop-global-crossing-dissipation} implies that \(\mathcal V_H(x_n,y_n)\) is nonincreasing. Its bounded sublevel sets, established in \eqref{eq-global-coercivity}, therefore keep the orbit bounded. Summing \eqref{eq-global-dissipation} gives
\[
\sum_n\Phi_H(y_n)<\infty,\qquad
\sum_n[x_n+g(y_n)]^2<\infty.
\]
In particular, \(\Phi_H(y_n)\to0\). The continuous function \(\Phi_H\) has only the zero \(y_0\) on \([0,\infty)\). Every accumulation point of the bounded sequence \(y_n\) is therefore \(y_0\), so \(y_n\to y_0\). The first component of \(F\) gives \(x_{n+1}=-g(y_n)\to x_0\).

\PN Any periodic orbit of the reduced map contained in \(\mathcal D_\infty\) must therefore be the fixed point. A crossing-only periodic trajectory of the Filippov system induces a periodic orbit of the reduced map, possibly after doubling the discrete period because of the symmetry normalization. It must consequently coincide with the classified simple-period crossing cycle. Hence no crossing-only periodic orbit with a higher number of crossings can exist.

\PN For a crossing trajectory represented by \(p\in\mathcal D_\infty\), the reduced crossing data converge to \((x_0(t),y_0(t))\). Continuity of \(t^X\), of the first-crossing maps, and of the sidewise flows on the finite flight intervals gives uniform convergence of the successive flight arcs after a phase shift. This is orbital convergence to \(\gamma_{(\mu,H)}\).

\PN The cone argument and the energy estimate also apply when some landing ordinates are zero. Since the limiting ordinate satisfies \(y_0>0\), all sufficiently late landing ordinates are strictly positive. Consequently, only finitely many zero-duration contacts with folds can occur. This extension concerns chains of flights without positive sliding segments. Trajectories containing such segments require the separate interface analysis developed next.
\end{proof}

\subsection{Robust uniqueness within the simple-period crossing class under
Whitney \texorpdfstring{\(C^1\)}{C1} perturbations}
\label{subsec:global_C1_uniqueness}

\PN The local continuation in Corollary~\ref{cor:C1_persistence} guarantees that the hyperbolic crossing cycle survives small sidewise \(C^1\) perturbations, but does not exclude the appearance of additional simple-period crossing cycles away from it. The global crossing dissipation obtained above provides precisely this additional control. Since such a conclusion concerns the whole crossing domain, perturbations are considered in the Whitney \(C^1\) topology on each side of \(\Sigma\), with the switching manifold fixed.

\begin{corollary}\label{cor:global_C1_uniqueness}
\PN Let \(Z\in\CLzero\) with \(H\in\mathcal I_\mu\), and let \(\gamma_{(\mu,H)}\) be the crossing limit cycle given by Theorem~\ref{teo:main1}. There exists a Whitney \(C^1\) neighborhood \(\mathscr U_Z\) of \(Z\), with the topology taken on each side of \(\Sigma\), such that, for every \(\widetilde Z=(\widetilde X,\widetilde Y)\in\mathscr U_Z\), the continuation \(\widetilde\gamma\) supplied by Corollary~\ref{cor:C1_persistence} is the unique simple-period crossing limit cycle of \(\widetilde Z\). Moreover, \(\widetilde\gamma\) remains hyperbolic and orbitally asymptotically stable.
\end{corollary}

\begin{proof}
\PN By Corollary~\ref{cor:C1_persistence}, there is a neighborhood \(U\) of the crossing point \(p_0\) in which every sufficiently small sidewise \(C^1\) perturbation has a unique fixed point \(\widetilde p_0\) of the perturbed Poincar\'e map \(\Pi_{\widetilde Z}\). This fixed point is hyperbolic and attracting and generates the continued simple-period crossing cycle \(\widetilde\gamma\).

\PN It remains to exclude fixed points of \(\Pi_{\widetilde Z}\) outside \(U\). The proof of Theorem~\ref{thm-global-crossing} provides a strict Lyapunov decrease for the unperturbed crossing return away from \(p_0\). More precisely, the associated loss is continuous, vanishes only at \(p_0\), and is coercive on the admissible crossing cone. Hence, after \(U\) is fixed, the decrease is separated from zero on every compact subset of the crossing domain disjoint from \(U\), while the coercive estimate controls the noncompact part.

\PN First-crossing maps depend continuously in \(C^1\) on the two sidewise vector fields wherever the crossings remain transverse. The Whitney \(C^1\) topology on each side of \(\Sigma\) therefore allows \(\mathscr U_Z\) to be chosen so that the strict Lyapunov decrease persists throughout the crossing domain outside \(U\). Consequently, if \(p\notin U\) belongs to the domain of the perturbed simple-period return, then
\[
\mathcal V_H\!\left(\Pi_{\widetilde Z}(p)\right)<\mathcal V_H(p).
\]
Such a point cannot be fixed by \(\Pi_{\widetilde Z}\).

\PN Thus every fixed point of the perturbed simple-period crossing return lies in \(U\), where Corollary~\ref{cor:C1_persistence} gives the unique fixed point \(\widetilde p_0\). Hence \(\widetilde\gamma\) is the unique simple-period crossing limit cycle of \(\widetilde Z\). Its hyperbolicity and orbital asymptotic stability follow from the local continuation result.
\end{proof}

\PN Corollary~\ref{cor:global_C1_uniqueness} strengthens local persistence by showing that, under sufficiently small perturbations in the Whitney \(C^1\) topology on each side of \(\Sigma\), the continued cycle remains the unique simple-period crossing limit cycle. The global asymptotic classification additionally relies on the geometric conditions governing sliding and on the characterization of the exceptional set of initial conditions whose trajectories reach the two-fold in finite time.

\subsection{Transient sliding and the energy of fold launches}

\PN In this subsection we work with \(Z_\kappa\) in \eqref{eq:zero_divergence_residual_fields}, with \(\kappa\le0\). Sliding follows the Filippov convention \cite{Filippov,teixeira1990,diBernardo2008}. All flight quantities and energy functions remain those defined above, since the normalized return map is unchanged. A flight launched in the cone may land in attractive sliding. Its landing energy is controlled, but the outgoing fold can have amplitude below \(x_{\mathrm g}\), outside the domain of \(\Theta_H\). A separate quadratic value at fold launches will permit comparison across this interface and supply a fixed loss at each subsequent emission.

\PN Only two boundary values are needed:
\[
\Delta_0=-f_H(0)=(H-H_{\mathrm g})x_{\mathrm g}>0,\qquad
L_0=-a_H(0)=cx_{\mathrm g}+b_*.
\]
The sector estimate gives
\begin{equation}\label{eq-global-fold-loss}
2L_0>\Delta_0>0,\qquad
\Phi_H(0)=\Delta_0(2L_0-\Delta_0)>0.
\end{equation}
Equation~\eqref{eq-global-Theta} also gives
\[
\Theta_H'(x_{\mathrm g})=2H_{\mathrm g}(L_0-\Delta_0).
\]
These values depend on the parameters, not on the initial condition.

\PN Two parameter margins are needed for the fold comparison. One controls a lower barrier for the sliding exit; the other compensates for the weaker barrier at large entries. Both hold throughout \(\mathcal I_\mu\).

\begin{lemma}\label{lem-global-fold-margins}
\PN For \(\mu>0\) and \(H\in\mathcal I_\mu\),
\begin{equation}\label{eq-global-fold-margins}
8\mu H_{\mathrm g}x_{\mathrm g}<1-H_{\mathrm g}^2,
\qquad
\Lambda_{\mu,H}>H_{\mathrm g}(H-2H_{\mathrm g}).
\end{equation}
Consequently,
\[
(H_{\mathrm g}^2+\Lambda_{\mu,H})x_{\mathrm g}
>H_{\mathrm g}\Delta_0>2H_{\mathrm g}(\Delta_0-L_0).
\]
\end{lemma}
\begin{proof}
\PN Put \(z=\mu t_{\mathrm g}\). The grazing formula gives \(x_{\mathrm g}\le e^z\) and \(H_{\mathrm g}=e^{-2z}\), so the first inequality follows from \(e^z-e^{-3z}>8\mu\). The derivative of \(e^z-e^{-3z}\) has minimum \(4/\sqrt3\). For \(0<\mu\le1\), the grazing equation gives \(t_{\mathrm g}>\pi+\arctan(1/\mu)\ge5\pi/4\), and hence the difference is greater than \(5\pi\mu/\sqrt3>8\mu\). For \(\mu\ge1\), it is greater than \(e^z-1>z+z^2/2+z^3/6>12\mu\), using \(z>\pi\mu\). This proves the first margin.

\PN For fixed \(\mu\), the difference
\(\Lambda_{\mu,H}-H_{\mathrm g}(H-2H_{\mathrm g})\) is affine in \(H\).
At \(H=H_{\mathrm g}\) it equals
\(\Lambda_{\mu,H_{\mathrm g}}+H_{\mathrm g}^2>0\).
It therefore suffices to check the other endpoint:
\[
\Lambda_{\mu,H_{\mathrm{crit}}}
>H_{\mathrm g}(H_{\mathrm{crit}}-2H_{\mathrm g}).
\]
Put \(\mu_a=\log2/(4\pi)\) and \(\mu_b=2/(3\pi-4)\); then \(0<\mu_a<1/10<\mu_b\). If \(\mu\le\mu_a\), the bound \(H_{\mathrm g}>k^4\ge1/2\) gives \(H_{\mathrm g}(H_{\mathrm{crit}}-2H_{\mathrm g})<0\). If \(\mu\ge\mu_b\), then \(1-k+k^2\ge3/4\) gives \(\Lambda_{\mu,H_{\mathrm{crit}}}\ge H_{\mathrm{crit}}k^2>H_{\mathrm{crit}}H_{\mathrm g}\). Both cases imply the required inequality.

\PN On the remaining interval, the identity
\[
H_{\mathrm g}(H_{\mathrm{crit}}-2H_{\mathrm g})
=\frac{H_{\mathrm{crit}}^2}{8}
 -2\left(H_{\mathrm g}-\frac{H_{\mathrm{crit}}}{4}\right)^2
\le\frac{H_{\mathrm{crit}}^2}{8}
\]
reduces the claim to
\begin{equation}\label{eq-global-intermediate-margin}
16\mu k(1-k+k^2)[\pi(1-k+k^2)-1]>1.
\end{equation}
Indeed, this inequality gives \(\Lambda_{\mu,H_{\mathrm{crit}}}/H_{\mathrm{crit}}>H_{\mathrm{crit}}/8\), the required upper bound at the endpoint.

\PN The left-hand side of \eqref{eq-global-intermediate-margin} is
strictly log-concave as a function of \(\mu>0\). Indeed, differentiating
with respect to \(\pi\mu\), using \(k=e^{-\pi\mu}\), gives
\[
\begin{aligned}
&\frac{d^2}{d(\pi\mu)^2}
 \log\!\left(16\mu k(1-k+k^2)
       [\pi(1-k+k^2)-1]\right)\\
&\quad\le-\frac1{(\pi\mu)^2}
+k(4k-1)\left[
 \frac1{1-k+k^2}
 +\frac{\pi}{\pi(1-k+k^2)-1}\right].
\end{aligned}
\]
The bracket is at most
\(4/3+\pi/(3\pi/4-1)<4\), since
\(1-k+k^2\ge3/4\) and \(\pi>8/3\).
If \(k\le1/4\), the displayed upper bound is negative.
If \(k>1/4\), then
\[
(\pi\mu)^2k<(1-k)^2,\qquad
\frac14-(1-k)^2(4k-1)
=(2k-1)^2\left(\frac54-k\right)\ge0.
\]
The first inequality is \(2\sinh(\pi\mu/2)>\pi\mu\).
Hence the upper bound is negative in this case as well.

\PN It remains to check the two endpoints. At
\(\mu_a=\log2/(4\pi)\), one has \(k=2^{-1/4}>5/6\) and
\(1-k+k^2>6/7\). The elementary bounds
\(\log2>2/3\) and \(25/8<\pi<22/7\) give
\[
16\mu_a k(1-k+k^2)[\pi(1-k+k^2)-1]
>
\frac{4(2/3)}{22/7}\frac56\frac67
\left[\frac{25}{8}\frac67-1\right]
=\frac{235}{231}>1.
\]
At \(\mu_b=2/(3\pi-4)\), the same bounds give
\(1/3<\mu_b<2/5\) and
\(\pi\mu_b<2\pi/5<4/3<\log4\).
Thus \(k>1/4\) and \(1-k+k^2\ge3/4\), so
\[
16\mu_b k(1-k+k^2)[\pi(1-k+k^2)-1]
>
16\frac13\frac14\frac34
\left[\frac{25}{8}\frac34-1\right]
=\frac{43}{32}>1.
\]
Strict log-concavity now proves
\eqref{eq-global-intermediate-margin} on \([\mu_a,\mu_b]\).
Adding \(H_{\mathrm g}^2\) to the second margin proves the first comparison in the consequence; the second follows from \(2L_0>\Delta_0\).
\end{proof}

\PN Lemma~\ref{lem:finite-sliding-exit} below proves that every off-diagonal sliding entry reaches a regular exit in finite physical time. Its exit amplitude is constrained between an upper bound and two lower barriers. The local lower barrier retains the actual entry amplitude up to \(2x_{\mathrm g}\); for larger entries, the weaker barrier will be combined with the excess curvature of \(\Theta_H\).

\begin{lemma}\label{lem-global-sliding-barriers}
\PN Let \(p=(x,y,0)\in\Sigma^{\rm s}\), with \(x+y\ne0\), be a sliding entry for \(Z_\kappa\) with \(\kappa\le0\) and \(H\in\mathcal I_\mu\). Write \(x_{\rm out}>0\) for the nonzero coordinate magnitude at its regular exit fold. Then
\[
0<x_{\rm out}<\max(x,-y)+H\min(x,-y).
\]
If \(x\ge x_{\mathrm g}\) and \(-H_{\mathrm g}x_{\mathrm g}<y<0\), then
\[
x_{\rm out}>x_{\mathrm g}+\frac{y}{H_{\mathrm g}}.
\]
If also \(x\le2x_{\mathrm g}\), the stronger bound holds:
\[
x_{\rm out}>x+\frac{y}{H_{\mathrm g}}>0.
\]

\end{lemma}
\begin{proof}
\PN Use the common normalized coordinates \(x>0>y\) of \(\Sigma^{\rm s}\). Multiplying the physical sliding field by \(s=x-y>0\) preserves its oriented trajectories; with \(dt/d\tau=s\) and primes denoting derivatives in \(\tau\), it gives
\begin{equation}\label{eq-global-sliding-pq}
\begin{aligned}
x'&=Hx+y-2\mu x(x+y)-\kappa xy,\\
y'&=x+Hy+2\mu y(x+y)+\kappa xy.
\end{aligned}
\end{equation}
The sign of \(x+y\) is preserved, as also follows from \eqref{eq:residual_sliding_sd}. Since \(\kappa\le0\) and \(xy<0\), one has \(\kappa xy\ge0\). For \(x+y>0\), also \(x+Hy>0\), and therefore
\[
(x-Hy)'=(1-H^2)y-2\mu(x+Hy)(x+y)
-(1+H)\kappa xy<0.
\]
The exit has \(y=0\), giving the upper bound. The other case follows by the involution \(S\).

\PN The two lower barriers use the same outward-pointing boundary.
Write the entry coordinates as \(x_{\rm in},y_{\rm in}\), with
\(x_{\rm in}\ge x_{\mathrm g}\) and
\(-H_{\mathrm g}x_{\mathrm g}<y_{\rm in}<0\), and put
\[
w=x+\frac{y}{H_{\mathrm g}},\qquad
w_{\rm in}=
\begin{cases}
x_{\rm in}+y_{\rm in}/H_{\mathrm g},&x_{\rm in}\le2x_{\mathrm g},\\
x_{\mathrm g}+y_{\rm in}/H_{\mathrm g},&x_{\rm in}>2x_{\mathrm g}.
\end{cases}
\]
Then \(w_{\rm in}>0\) and \(w\ge w_{\rm in}\) at entry.
The sign of \(x+y\) stays positive. While
\(y_{\rm in}\le y<0\), Lemma~\ref{lem-global-fold-margins}
gives \(2\mu(-y)<1\); together with \(\kappa xy\ge0\), this yields
\[
y'=(x+y)(1+2\mu y)-(1-H)y+\kappa xy>0.
\]
On either boundary \(w=w_{\rm in}\), \(y\ge y_{\rm in}\), one has
\(x\le2x_{\mathrm g}\). There
\(H_{\mathrm g}x+y>0\) and
\(0<(x+y)(H_{\mathrm g}x-y)\le2H_{\mathrm g}x^2\), so
\[
\begin{aligned}
H_{\mathrm g}w'
&=(1-H_{\mathrm g}^2)x+(H_{\mathrm g}+H)(H_{\mathrm g}x+y)
 -2\mu(x+y)(H_{\mathrm g}x-y)
 +(1-H_{\mathrm g})\kappa xy\\
&\ge x(1-H_{\mathrm g}^2-4\mu H_{\mathrm g}x)\\
&\ge x(1-H_{\mathrm g}^2-8\mu H_{\mathrm g}x_{\mathrm g})>0.
\end{aligned}
\]
Thus the region \(w\ge w_{\rm in}\), \(y\ge y_{\rm in}\)
is invariant until \(y=0\), with strict separation from
\(w=w_{\rm in}\) after entry. The positive barrier excludes the
diagonal and the origin. Lemma~\ref{lem:finite-sliding-exit}
gives a regular exit with \(x_{\rm out}>w_{\rm in}\),
proving both lower bounds.

\end{proof}

\PN Define the fold value directly on the entire half-line \(x\ge0\):
\begin{equation}\label{eq-global-fold-polynomial}
J_{\mathrm f}(x)=\Theta_H(x_{\mathrm g})+y_0^2
+2H_{\mathrm g}(L_0-\Delta_0)(x-x_{\mathrm g})
+H_{\mathrm g}^2(x-x_{\mathrm g})^2.
\end{equation}
This quadratic matches the value and derivative of the trace \(\mathcal V_H(x,0)\) at \(x_{\mathrm g}\). Integrating \eqref{eq-global-weighted-curvature} twice, using \(\mathcal E\ge H_{\mathrm g}\), gives
\begin{equation}\label{eq-global-Theta-P}
\mathcal V_H(x,0)\ge J_{\mathrm f}(x)+\Lambda_{\mu,H}(x-x_{\mathrm g})^2
\qquad(x\ge x_{\mathrm g}).
\end{equation}
The fold value is therefore available below \(x_{\mathrm g}\), where \(\Theta_H\) is not defined.

\PN A fold emission has normalized landing \(F(x_{\rm out},0)=(x_{\mathrm g},Hx_{\mathrm g}-H_{\mathrm g}x_{\rm out})\). Equations~\eqref{eq-global-fold-loss} and \eqref{eq-global-fold-polynomial} give
\begin{equation}\label{eq-global-fold-emission}
\mathcal V_H(x_{\mathrm g},Hx_{\mathrm g}-H_{\mathrm g}x_{\rm out})
=J_{\mathrm f}(x_{\rm out})-\Phi_H(0).
\end{equation}
The identity holds for any landing sign. Its physical use requires \(x_{\rm out}>0\); the polynomial value at zero selects no continuation from the two-fold.

\PN A uniform loss at emission becomes useful only if the preceding sliding interface does not increase the assigned fold value. The required comparison is between \(J_{\mathrm f}\) at the exit and \(\mathcal V_H\) at the landing. The lower barriers in Lemma~\ref{lem-global-sliding-barriers} and the curvature margin provide this comparison.

\begin{proposition}\label{prop-global-sliding-interface}
\PN For a sliding entry \((x,y,0)\) of \(Z_\kappa\), with \(\kappa\le0\), \(x\ge x_{\mathrm g}\), \(y<0\), and a regular exit of amplitude \(x_{\rm out}\),
\begin{equation}\label{eq-global-interface}
J_{\mathrm f}(x_{\rm out})<\mathcal V_H(x,y).
\end{equation}
For a direct landing of duration zero on a regular fold,
\[
J_{\mathrm f}(x)\le\mathcal V_H(x,0).
\]

\end{proposition}
\begin{proof}
\PN The zero-duration statement is \eqref{eq-global-Theta-P}. For \(y<0\), that estimate gives
\[
\mathcal V_H(x,y)\ge J_{\mathrm f}(x)+y^2-2(cx+b_*)y
+\Lambda_{\mu,H}(x-x_{\mathrm g})^2.
\]
Convexity of \(J_{\mathrm f}\) reduces the comparison to two endpoint values. The upper endpoint is the barrier \(\max\{x,-y\}+H\min\{x,-y\}\) from Lemma~\ref{lem-global-sliding-barriers}; the lower endpoint is zero or one of its two lower barriers. We check those endpoint values once, then apply convexity to the exit amplitude. If \(-x\le y<0\), the upper endpoint is \(x-Hy\), and
\[
\begin{aligned}
J_{\mathrm f}(x-Hy)-\mathcal V_H(x,y)
\le{}&-(1-H_{\mathrm g}^2H^2)y^2
+2(c-H_{\mathrm g}^2H)(x-x_{\mathrm g})y\\
&+2[(1-H_{\mathrm g}H)L_0+H_{\mathrm g}H\Delta_0]y<0.
\end{aligned}
\]
Here \(H_{\mathrm g}H<1\) and \(c>H_{\mathrm g}^2\), since \(H>H_{\mathrm g}\) and \(1-k+k^2>H_{\mathrm g}\). If \(y\le-x\), the upper endpoint is \(-y+Hx\). Using \(x\ge x_{\mathrm g}\) and \(y\le-x\) gives
\[
\frac12\partial_y[J_{\mathrm f}(-y+Hx)-\mathcal V_H(x,y)]
\ge(1-H_{\mathrm g}^2H)x+H_{\mathrm g}\Delta_0+(1-H_{\mathrm g})L_0>0.
\]
The difference is negative at \(y=-x\) by the preceding case, and remains negative for smaller \(y\).

\PN For the lower endpoint there are three cases. If \(y\le-H_{\mathrm g}x_{\mathrm g}\), use zero. The difference \(J_{\mathrm f}(0)-\mathcal V_H(x,y)\) increases with \(y\); at \(y=-H_{\mathrm g}x_{\mathrm g}\), it is at most
\[
-H_{\mathrm g}^2(x-x_{\mathrm g})^2
+2H_{\mathrm g}(\Delta_0-L_0-cx_{\mathrm g})(x-x_{\mathrm g})
-2H_{\mathrm g}x_{\mathrm g}(2L_0-\Delta_0)<0.
\]
Indeed, \(\Delta_0-L_0<\Delta_0/2<cx_{\mathrm g}\), using \(c\ge3H/4\).
If \(-H_{\mathrm g}x_{\mathrm g}<y<0\) and \(x\le2x_{\mathrm g}\), use the local endpoint \(x+y/H_{\mathrm g}\):
\[
J_{\mathrm f}(x+y/H_{\mathrm g})-\mathcal V_H(x,y)
\le2[(H_{\mathrm g}+c)(x-x_{\mathrm g})+2L_0-\Delta_0]y<0.
\]
For the remaining case \(x>2x_{\mathrm g}\), use \(x_{\mathrm g}+y/H_{\mathrm g}\) and retain the excess curvature:
\[
\begin{aligned}
J_{\mathrm f}(x_{\mathrm g}+y/H_{\mathrm g})-\mathcal V_H(x,y)
\le{}&-(H_{\mathrm g}^2+\Lambda_{\mu,H})(x-x_{\mathrm g})^2
+2H_{\mathrm g}(\Delta_0-L_0)(x-x_{\mathrm g})\\
&+2[2L_0-\Delta_0+c(x-x_{\mathrm g})]y<0.
\end{aligned}
\]
Indeed, the first two terms are strictly less than
\[
(x-x_{\mathrm g})
\bigl[-(H_{\mathrm g}^2+\Lambda_{\mu,H})x_{\mathrm g}
      +2H_{\mathrm g}(\Delta_0-L_0)\bigr]<0
\]
by \(x-x_{\mathrm g}>x_{\mathrm g}\) and Lemma~\ref{lem-global-fold-margins}.
The remaining term is negative because \(2L_0-\Delta_0+c(x-x_{\mathrm g})>0\)
and \(y<0\). In every case both endpoint values lie below
\(\mathcal V_H(x,y)\); Lemma~\ref{lem-global-sliding-barriers}
places the exit strictly between them, so convexity proves
\eqref{eq-global-interface}.

\end{proof}

\begin{lemma}\label{lem:finite-sliding-exit}
\PN Let \(0<H<1\), \(\mu>0\), and \(\kappa\le0\). Every trajectory of the sliding field of \(Z_\kappa\) starting in strict attractive sliding leaves that region in finite physical time. If \(d=x+y=0\), it reaches the two-fold; if \(d\ne0\), it reaches a regular exit fold.
\end{lemma}
\begin{proof}
On strict attractive sliding, the coordinates \(s=x-y>0\) and \(d=x+y\) satisfy \eqref{eq:residual_sliding_sd}. In particular, \(\dot s\le-(1-H)\), while the equation for \(d\) preserves its sign. If \(d=0\), then \(s\) reaches zero in physical time at most \(s(0)/(1-H)\), so the trajectory reaches the two-fold. Suppose \(d\ne0\) and no fold is reached. The bound \(\dot s\le-(1-H)\) forces \(s\to0\) at a finite maximal time \(T\). Since \(|d|<s\), one has \(d^2/s\le s\) and \((s^2-d^2)/(2s)\le s/2\), so
\[
-(1-H)-\left(2\mu+\frac{|\kappa|}{2}\right)s\le\dot s\le-(1-H).
\]
Writing \(C_0=2\mu+|\kappa|/2>0\), integration of the lower inequality with \(s(T)=0\) gives
\[
s(t)\le\frac{1-H}{C_0}\left(e^{C_0(T-t)}-1\right)\le C(T-t)
\]
near \(T\). Therefore \(\int_{T-\delta}^{T}dt/s(t)\ge C^{-1}\int_{T-\delta}^{T}dt/(T-t)=\infty\) for any sufficiently small \(\delta>0\). The equation for \(d\) would give
\[
|d(t)|
=
|d(0)|\exp\left((1+H)\int_0^t\frac{du}{s(u)}-2\mu t\right)
\longrightarrow\infty,
\]
contradicting \(|d|<s\to0\). Thus every off-diagonal sliding orbit reaches a regular fold in finite time.
\end{proof}

\PN The two regular exit states are \((x_{\rm out},0,0)\) and \((0,-x_{\rm out},0)\), with \(x_{\rm out}>0\). The involution \(S\) normalizes the second to the first.

\PN To state the resulting scope without choosing a continuation at the two-fold, a trajectory is called regular in this subsection when it concatenates smooth flights, Filippov sliding in \(\Sigma^{\rm s}\) \cite{Filippov,diBernardo2008}, and the outgoing smooth characteristic at each visible exit fold. It is stopped upon reaching the two-fold, whose local geometry is discussed in \cite{jeffrey2009}. A visible grazing of a trajectory already in a half-space is continued along its incident smooth field.

\PN Consider the event domain
\[
\mathscr A_H=
\{(x,y):x\ge x_{\mathrm g},\ 0<y\le Hx\}
\cup\{(x,0):x>0\},
\]
and assign
\begin{equation}\label{eq-global-event-energy}
\mathscr E_H(x,y)=
\begin{cases}
\mathcal V_H(x,y),&y>0,\\
J_{\mathrm f}(x),&y=0.
\end{cases}
\end{equation}
The possible downward jump at a fold is consistent with the interface comparison. Only a balance at successive launches will be used.

\begin{proposition}\label{prop-global-event-dissipation}
\PN For \(\kappa\le0\) and \(H\in\mathcal I_\mu\), every event from \(\mathscr A_H\) has its next launch in \(\mathscr A_H\).
The energy \(\mathscr E_H\) is bounded below there and has bounded
sublevel sets. Successive launches \(p_n=(x_n,y_n)\) and \(p_{n+1}\) satisfy
\begin{equation}\label{eq-global-event-crossing}
\mathscr E_H(p_{n+1})-\mathscr E_H(p_n)
\le-\Phi_H(y_n)-d_{\mu,H}[x_n+g(y_n)]^2\le0
\qquad(y_n>0),
\end{equation}
and
\begin{equation}\label{eq-global-event-fold}
\mathscr E_H(p_{n+1})-\mathscr E_H(p_n)
\le-\Phi_H(0)<0
\qquad(y_n=0).
\end{equation}
Consequently, if a finite sequence of events contains \(n\) fold launches,
\begin{equation}\label{eq-global-fold-count}
n\Phi_H(0)\le\mathscr E_H(p_0)-\inf_{\mathscr A_H}\mathscr E_H,
\end{equation}
so every regular trajectory starting in \(\mathscr A_H\) has only finitely many fold
launches. The bound depends on its initial state.
\end{proposition}
\begin{proof}
\PN A crossing launch in the cone satisfies \(x-Hy>0\), so
\eqref{eq:F_explicit} places every positive landing in the cone.
A zero or negative landing leads to a fold launch, after finite sliding
when needed. A fold emission has first coordinate \(x_{\mathrm g}\);
a positive second coordinate is at most \(Hx_{\mathrm g}\).
Thus the event domain is invariant.

\PN For \(y_n>0\), Proposition~\ref{prop-global-crossing-dissipation}
controls the landing energy. A direct fold landing uses
\eqref{eq-global-Theta-P}, while a negative landing followed by a
regular sliding exit uses Proposition~\ref{prop-global-sliding-interface}.
These three cases give \eqref{eq-global-event-crossing}.
For \(y_n=0\), \eqref{eq-global-fold-emission} and the same interface
comparison give \eqref{eq-global-event-fold}; at a zero landing use
\(J_{\mathrm f}(x_{\mathrm g})=\mathcal V_H(x_{\mathrm g},0)\).
A positive sliding segment adds a strict loss.

\PN The positive quadratic leading coefficient of \(J_{\mathrm f}\)
and \eqref{eq-global-coercivity} bound \(\mathscr E_H\) below and make
its sublevel sets bounded. Summing the event inequalities proves
\eqref{eq-global-fold-count}. Since \(\Phi_H(0)>0\), infinitely many
fold launches are impossible.
\end{proof}

\begin{theorem}\label{thm-global-hybrid}
\PN Let \(Z\in\CLequi\) have zero divergence, normalized parameters \(\mu>0\) and \(H\in\mathcal I_\mu\), and satisfy \eqref{eq:nonsingular_sliding}. Let \(\gamma_Z\) denote the corresponding crossing cycle in the original coordinates, with the notation fixed above from the reductions of Proposition~\ref{prop:canonical_form} and Theorem~\ref{teo:main1}. A regular trajectory initiated at a crossing point, in attractive sliding, or at a regular exit fold has one of two alternatives: it reaches the two-fold in finite physical time, or it has only finitely many noncrossing contacts and converges orbitally to \(\gamma_Z\). In particular, every such trajectory has finitely many positive sliding episodes and is eventually strictly crossing.

The same alternatives hold from the first transverse encounter with \(\Sigma\) for a trajectory initiated in either open half-space outside the stable affine line of the corresponding saddle-focus. If a regular visible grazing occurs first, the statement uses the characteristic continuation along the incident field specified above.
\end{theorem}
\begin{proof}

\PN By Proposition~\ref{prop:canonical_form}, it is enough to work with \(Z_\kappa\) in common normalized coordinates. Lemma~\ref{lem:nonsingular_sliding_parameter} gives \(\kappa\le0\), and the common affine reduction and positive time scaling preserve the selected Filippov trajectories and finiteness of physical times. Lemma~\ref{lem:finite-sliding-exit} shows that every trajectory in strict attractive sliding reaches either an exit fold or the two-fold in finite time. If a fold is reached, the event balance begins there. For an initial crossing point, either the trajectory remains crossing, in which case Theorem~\ref{thm-global-crossing} applies, or its first noncrossing contact is a direct fold landing or an entry into attractive sliding. In the latter case, sliding reaches a fold or the two-fold in finite time. Accordingly, every trajectory that continues beyond this stage enters the event balance at a fold.

\PN Proposition~\ref{prop-global-event-dissipation} now gives only finitely many fold launches. Every regular positive sliding episode and every zero-duration fold landing contributes a launch, so both types of noncrossing contact are finite.

\PN After the last such contact, all landings are positive and remain in the cone. The resulting tail belongs to \(\mathcal D_\infty\), and Theorem~\ref{thm-global-crossing} gives convergence of the section data to \((x_0,y_0)\) for the unchanged section map. Each eventual smooth arc is the inverse image, under the corresponding fixed shear, of the canonical arc; hence the crossing convergence also gives orbital convergence to \(\gamma_{(\mu,H)}^\kappa\), and then to \(\gamma_Z\) in the original coordinates. No continuation at the two-fold is used in this argument.

\PN Each event includes a smooth flight of duration greater than \(\pi\) in normalized time, and every sliding segment has finite duration. The positive constant time scaling preserves a positive lower bound for flight durations in original time, so an infinite sequence of events cannot accumulate in finite physical time. If a trajectory reaches the two-fold, only finitely many events precede that arrival. Its final landing lies on the diagonal, and the subsequent sliding segment reaches the origin in finite time by Lemma~\ref{lem:finite-sliding-exit}. Thus any trajectory reaching the two-fold does so after finitely many events and in finite physical time.

\PN For the upper field \(X_\kappa\), the \((y,z)\)-subsystem is identical to that of the canonical member, and therefore
\[
z(t)
=
\zeta_\mu+e^{\mu t}(A\cos t+B\sin t).
\]
Outside the stable affine line one has \((A,B)\ne(0,0)\). Choose \(t_0\ge0\) with \(A\cos t_0+B\sin t_0<0\); then \(z(t_0+2\pi n)\to-\infty\), so a trajectory starting above \(\Sigma\) encounters it in finite forward time. If the first contact is a regular visible grazing, continuation in the incident field gives a later transverse encounter. At a transverse encounter from above one has \(y<0\), and the point lies in crossing, attractive sliding, or their fold boundary. The lower half-space follows by symmetry. The preceding argument now applies unless the trajectory reaches the two-fold first.
\end{proof}

\PN For the trajectories covered by Theorem~\ref{thm-global-hybrid}, finite-time arrival at the two-fold is the only obstruction to attraction by \(\gamma_Z\). The stable affine lines are incorporated separately into the exceptional set when arbitrary initial conditions are considered. It remains to characterize the initial conditions whose trajectories reach the two-fold. The energy estimate first restricts the fold launches that can lead to the two-fold, and the corresponding initial conditions are then traced through the analytic return and exit maps.

\begin{proposition}\label{prop-terminal-null}
\PN Under the hypotheses of Theorem~\ref{thm-global-hybrid}, let \(\mathcal T_Z\subset\mathbb R^3\) be the set of initial points whose regular trajectories reach the two-fold in finite forward time. Then \(\mathcal T_Z\) is contained in a countable union of real-analytic submanifolds of dimension at most two. In particular,
\[
\operatorname{Leb}_3(\mathcal T_Z)=0,
\qquad \dim_{\mathrm H}\mathcal T_Z=2.
\]
\end{proposition}
\begin{proof}
\PN Work with \(Z_\kappa\), \(\kappa\le0\), in the common normalized coordinates of Proposition~\ref{prop:canonical_form}. By Lemma~\ref{lem:finite-sliding-exit}, a sliding episode reaches the two-fold only along the diagonal \(y=-x\). By \eqref{eq:F_explicit} and Lemma~\ref{lem:grazing_data}, every normalized landing from a crossing or fold launch has first coordinate at least \(x_{\mathrm g}\). Set
\[
B_\dagger=\mathcal V_H(x_{\mathrm g},-x_{\mathrm g}).
\]
The function \(x\mapsto\mathcal V_H(x,-x)\) is strictly convex for \(x>x_{\mathrm g}\), by \eqref{eq-global-weighted-curvature}, and its right derivative at \(x=x_{\mathrm g}\) is
\[
\left.\frac{d}{dx}\mathcal V_H(x,-x)\right|_{x=x_{\mathrm g}}
=
2\bigl[(1+H_{\mathrm g})L_0+
(1+c-H_{\mathrm g}H+H_{\mathrm g}^2)x_{\mathrm g}\bigr]>0.
\]
Here \eqref{eq-global-Theta} and \eqref{eq-global-fold-loss} give the endpoint value and \(L_0>0\), while \(0<H_{\mathrm g}<H<1\). Hence every diagonal landing has energy at least \(B_\dagger\). A launch in \(\mathscr A_H\) with \(\mathscr E_H<B_\dagger\) therefore cannot lead to the two-fold. Indeed, the event inequalities \eqref{eq-global-event-crossing}--\eqref{eq-global-event-fold} keep the launch energy below \(B_\dagger\), while \eqref{eq-global-dissipation} and \eqref{eq-global-fold-emission} bound the next landing energy by the launch energy. Arrival at the two-fold would require a final diagonal sliding episode and hence a diagonal landing with energy at least \(B_\dagger\), which is impossible. A direct arrival at the origin from a launch is excluded by Lemma~\ref{lem:first_return}, since every complete flight lasts more than \(\pi\).

\PN The fold emission
\[
F(x,0)=(x_{\mathrm g},Hx_{\mathrm g}-H_{\mathrm g}x)
\]
singles out the amplitude
\begin{equation}\label{eq-terminal-fold-threshold}
x_\dagger=\frac{(1+H)x_{\mathrm g}}{H_{\mathrm g}}.
\end{equation}
If \(0<x<x_\dagger\), the landing ordinate satisfies
\[
-x_{\mathrm g}<y<Hx_{\mathrm g}<x_{\mathrm g},
\]
and
\[
B_\dagger-\mathcal V_H(x_{\mathrm g},y)
=
(x_{\mathrm g}+y)(x_{\mathrm g}-y+2L_0)>0.
\]
A positive landing therefore lies in the protected sublevel of \(\mathscr C_H\). If the landing ordinate is zero, then
\[
J_{\mathrm f}(x_{\mathrm g})
=
\mathcal V_H(x_{\mathrm g},0)
<
B_\dagger.
\]
If it is negative, Proposition~\ref{prop-global-sliding-interface} gives a regular fold exit with energy below \(B_\dagger\). None of these cases can lead to the two-fold. At \(x=x_\dagger\), the landing lies on the diagonal and reaches the two-fold through sliding. For \(x>x_\dagger\), one has \(x_{\mathrm g}+y<0\), so the landing enters sliding and reaches another regular fold. Consequently, once a trajectory that eventually reaches the two-fold has reached its first fold, its subsequent evolution consists of fold reinjections until the amplitude \(x_\dagger\) is attained.

\PN For \(x>x_\dagger\), let \(R_\kappa(x)>0\) denote the exit amplitude of the sliding trajectory starting from
\[
(x_{\mathrm g},Hx_{\mathrm g}-H_{\mathrm g}x,0).
\]
The desingularized field \eqref{eq-global-sliding-pq} is analytic and transverse to the exit boundary. At the exit point \((0,-R_\kappa(x),0)\), its first component is \(-R_\kappa(x)\ne0\). The physical exit time is finite by Lemma~\ref{lem:finite-sliding-exit}. During the sliding episode, \(s\) remains bounded below by the positive exit amplitude, and therefore the corresponding desingularized exit time is finite as well. Analytic dependence of the flow and the analytic implicit function theorem \cite{Teschl2012,KrantzParks2002} imply that \(R_\kappa\) is real-analytic on \((x_\dagger,\infty)\). Moreover,
\begin{equation}\label{eq-terminal-reinjection-limit}
\lim_{x\downarrow x_\dagger}R_\kappa(x)=0.
\end{equation}
To prove this limit, follow the diagonal trajectory corresponding to \(x=x_\dagger\) up to a time before its arrival at the two-fold for which \(s<\varepsilon/2\). By continuous dependence, trajectories with initial amplitude sufficiently close to \(x_\dagger\) remain in sliding up to that time and satisfy \(s<\varepsilon\). Equation~\eqref{eq:residual_sliding_sd} then implies that \(s\) continues to decrease, so the corresponding exit amplitudes are smaller than \(\varepsilon\). Hence \(R_\kappa\) is not constant.

\PN Since \(R_\kappa\) is nonconstant and real-analytic, every point fibre is discrete in \((x_\dagger,\infty)\) and therefore countable. Define
\[
\mathcal K_0=\{x_\dagger\},\qquad
\mathcal K_{n+1}
=
\{x>x_\dagger:R_\kappa(x)\in\mathcal K_n\},
\qquad
\mathcal K=\bigcup_{n\ge0}\mathcal K_n.
\]
Induction shows that each \(\mathcal K_n\), and therefore \(\mathcal K\), is countable. These are precisely the fold amplitudes whose successive reinjections eventually reach \(x_\dagger\). Conversely, every trajectory starting from a regular fold and reaching the two-fold must attain \(x_\dagger\) after finitely many reinjections, because arrival at the two-fold occurs in finite physical time and every complete smooth flight has a positive lower bound on its duration. Thus \(\mathcal K\) contains all fold amplitudes associated with trajectories that reach the two-fold. The argument uses only analyticity and nonconstancy of \(R_\kappa\), so critical points of \(R_\kappa\) cause no difficulty.

\PN Consider now initial points in an off-diagonal sliding sector. The map that assigns to each such point its regular exit amplitude is analytic. It is a submersion because the desingularized field is transverse to the exit folds, with normal component \(x\) at \((x,0,0)\) and \(-x\) at \((0,-x,0)\) for \(x>0\). Locally, exit amplitude and elapsed flow time provide analytic coordinates in which this map is the projection onto the amplitude coordinate. By the local submersion theorem \cite{Lee2012}, the inverse image of each amplitude in \(\mathcal K\) is therefore an analytic curve. Together with the diagonal \(y=-x\), these curves contain all sliding points whose trajectories reach the two-fold. The corresponding regular exit folds form a countable set, and second countability yields a countable collection of charts covering these sets.

\PN Crossing points whose trajectories subsequently reach the two-fold are obtained by pulling back the curves and points described above through the finite crossing itinerary preceding the first sliding contact. By \eqref{eq:F_explicit} and Lemma~\ref{lem:chord_monotone},
\[
\det DF(x,y)
=
-\mathcal E(y)g'(y)
=
\frac{y}{-g(y)}>0
\qquad (y>0).
\]
Analytic dependence of the flow and the implicit function theorem at each transverse endpoint \cite{Teschl2012,KrantzParks2002} imply that the return time is analytic. Hence \(F\), and every finite composition of the corresponding crossing maps, is locally an analytic diffeomorphism. Local inverse images therefore preserve the dimensions of the curves and points obtained above. Taking all finite crossing itineraries, all relevant charts, and their symmetric counterparts gives a countable cover of the set of section points whose trajectories reach the two-fold by analytic curves and points.

\PN It remains to lift this description from the switching manifold to the two open half-spaces. For an initial point whose first encounter with \(\Sigma\) is transverse to the incident vector field, the landing map is an analytic submersion of rank two. Locally, the landing point together with the backward flight time provides analytic coordinates. By the local submersion theorem \cite{Lee2012}, the inverse images of the section curves and points obtained above are analytic surfaces and curves. Initial conditions producing a visible grazing are covered separately. In the upper half-space they belong locally to the image of
\[
(x,t)\longmapsto
\varphi_{X_\kappa}(-t,(x,0,0)),
\qquad x\in\mathbb R,\quad t>0.
\]
This map is an analytic immersion of dimension two because the tangent vector \((1,0,0)\) to the fold line and the vector \(X_\kappa(x,0,0)\) are linearly independent, the latter having second component \(1\). Its image therefore admits a countable cover by embedded analytic surface patches. The same construction covers direct smooth arrivals at the two-fold, and the lower half-space follows by symmetry. Consequently, \(\mathcal T_Z\) is contained, in normalized coordinates, in a countable union of real-analytic submanifolds of dimension at most two. Such a union has zero three-dimensional Lebesgue measure \cite{Lee2012,Falconer2014}.

\PN The dimension bound is attained. Fix \(x>0\) and consider, for \(x\) in a sufficiently small interval and \(t>0\) sufficiently small,
\[
\varphi_{X_\kappa}(-t,(x,-x,0)).
\]
These points lie in the upper half-space and first reach \(\Sigma\) on the diagonal \(y=-x\), from which the sliding trajectory reaches the two-fold. At \(t=0\), the derivative with respect to \(x\) is the diagonal tangent \((1,-1,0)\), while the derivative with respect to \(t\) is the negative flow vector, whose normal component is \(x>0\). The parametrization therefore has rank two. Restricting the parameter domain gives an embedded analytic surface contained in the set of initial points whose trajectories reach the two-fold. Hence
\[
\dim_{\mathrm H}\mathcal T_Z\ge2.
\]
The countable cover by submanifolds of dimension at most two gives the reverse inequality, so
\[
\dim_{\mathrm H}\mathcal T_Z=2.
\]
Finally, the common invertible affine reduction transports both the countable analytic cover and the Hausdorff-dimension conclusion to the original system \(Z\). This completes the proof.
\end{proof}

\begin{theorem}\label{thm-almost-global}
\PN Let \(Z\in\CLequi\) have zero divergence, normalized parameters \(\mu>0\) and \(H\in\mathcal I_\mu\), and satisfy \eqref{eq:nonsingular_sliding}. There exists a set \(\mathcal N_Z\subset\mathbb R^3\) of zero three-dimensional Lebesgue measure such that every initial condition \(p\in\mathbb R^3\setminus\mathcal N_Z\) has a regular trajectory defined for all \(t\ge0\), with only finitely many noncrossing contacts, and
\begin{equation}\label{eq-almost-global-attraction}
\lim_{t\to\infty}\operatorname{dist}\bigl(\varphi_Z(t,p),\gamma_Z\bigr)=0.
\end{equation}
Consequently, \(\omega(p)=\gamma_Z\) for every \(p\in\mathbb R^3\setminus\mathcal N_Z\). The crossing cycle \(\gamma_Z\) is orbitally asymptotically stable.
\end{theorem}
\begin{proof}
\PN Let \(\ell_X\) and \(\ell_Y\) be the stable affine lines of the two saddle-foci in the original coordinates. Set
\[
\mathcal N_Z=\Sigma\cup\ell_X\cup\ell_Y\cup\mathcal T_Z.
\]
The switching plane and the two lines have zero three-dimensional measure, as does \(\mathcal T_Z\) by Proposition~\ref{prop-terminal-null}. Every point outside this union lies in an open half-space, outside the corresponding stable line, and its trajectory does not reach the two-fold in finite forward time. Theorem~\ref{thm-global-hybrid} therefore gives finitely many noncrossing contacts and orbital convergence to \(\gamma_Z\). The lateral affine fields have complete smooth flows; finitely many sliding episodes end in finite time, and the positive lower bound on complete flight times excludes finite-time accumulation of crossings. After the last noncrossing contact the trajectory converges to the cycle and stays bounded, so its regular continuation exists for all forward time. Orbital asymptotic stability is already given by Theorem~\ref{teo:main1} and the fixed coordinate changes.
\end{proof}

\PN The set \(\mathcal N_Z\) provides a zero-Lebesgue-measure exceptional cover sufficient for the global asymptotic classification, while the actual basin of \(\gamma_Z\) may be larger. Under the hypotheses of Theorem~\ref{thm-global-hybrid}, sliding is transient and periodic orbits containing sliding are excluded. The classification is formulated for the regular Filippov trajectories selected throughout the paper, with evolution stopped at the two-fold and without selecting a continuation in the escaping region.

\section{The oscillatory region in parameter space}\label{sec:edges}
\PN Within the symmetric zero-divergence subclass \(\CLzero\), Theorem~\ref{teo:main1} identifies \(\mathcal I_{\mu}\) for each \(\mu>0\). As \(\mu\) varies, these intervals form the \emph{simple-period oscillatory region}
\begin{equation}\label{eq:oscillatory_region}
\mathcal B_{\mathrm{osc}}
:=
\bigcup_{\mu>0}\{\mu\}\times\mathcal I_{\mu}
=
\bigl\{(\mu,H)\mid \mu>0,\ H_{\mathrm g}(\mu)<H<H_{\mathrm{crit}}(\mu)\bigr\}.
\end{equation}

\PN For the canonical member \(\kappa=0\), the convergence conclusions hold throughout \(\mathcal B_{\mathrm{osc}}\) without an additional energy condition. For a general system in the symmetric zero-divergence class, the almost-global conclusion also requires \eqref{eq:nonsingular_sliding}.

\PN To complete the description of \(\mathcal B_{\mathrm{osc}}\), it remains to determine the limiting behavior of its two boundary curves.

\begin{proposition}\label{prop:role_of_mu}
	\PN The boundary functions satisfy
	\[
	\lim_{\mu\to0^+}H_{\mathrm g}(\mu)
	=
	\lim_{\mu\to0^+}H_{\mathrm{crit}}(\mu)
	=
	1,
	\qquad
	\lim_{\mu\to\infty}H_{\mathrm g}(\mu)
	=
	\lim_{\mu\to\infty}H_{\mathrm{crit}}(\mu)
	=
	0.
	\]
\end{proposition}

\begin{proof}
	\PN The limits for \(H_{\mathrm{crit}}\) follow directly from~\eqref{eq:H_crit_definition}. Since \(\pi<t_{\mathrm g}(\mu)<2\pi\) and \(H_{\mathrm g}(\mu)=e^{-2\mu t_{\mathrm g}(\mu)}\), one has \(H_{\mathrm g}(\mu)\to1\) as \(\mu\to0^+\) and \(0<H_{\mathrm g}(\mu)<e^{-2\pi \mu}\), which gives \(H_{\mathrm g}(\mu)\to0\) as \(\mu\to\infty\).
\end{proof}

\PN Proposition~\ref{prop:role_of_mu} shows how the two boundary curves coalesce in the small- and large-\(\mu\) limits. Figure~\ref{fig:banda} displays the resulting simple-period oscillatory region.

\begin{figure}[H] \centering
\includegraphics[width=0.45\textwidth]{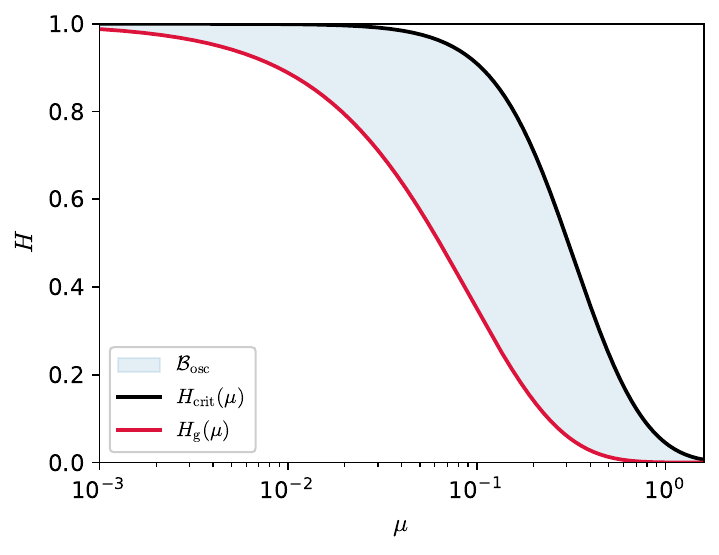} 
\caption{The simple-period oscillatory region \(\mathcal B_{\mathrm{osc}}\) on a logarithmic \(\mu\)-axis. The lower boundary \(H_{\mathrm g}(\mu)\) corresponds to grazing, while the upper boundary \(H_{\mathrm{crit}}(\mu)\) is approached in the infinite-amplitude limit.}
\label{fig:banda}
\end{figure}

\PN For any \((\mu,H)\in\mathcal B_{\mathrm{osc}}\), Proposition~\ref{prop:eta_bijection} gives a unique \(t=\eta^{-1}(H)\in(\pi,t_{\mathrm g}(\mu))\). The crossing coordinates are then obtained from~\eqref{eq:xy_of_t}, so \(p_0=(x_0(t),y_0(t),0)\), the second crossing is \(S(p_0)\), and the period is \(T=2t\). The corresponding canonical orbit is reconstructed from the \(X\)-arc issued from \(p_0\) and its symmetric image, while \eqref{eq:det_N} and \eqref{eq:trace_N} determine the eigenvalues of \(N(t)\), whose squares are the two nontrivial Floquet multipliers.
 Table~\ref{tab:cycles} gives representative cycles obtained directly from this parametrization.

\begin{table}[H]
	\centering
	\caption{Representative cycles of the canonical system obtained from the half-period parametrization with \(t=(\pi+t_{\mathrm g}(\mu))/2\), together with their nontrivial Floquet multipliers \(\mu_1\) and \(\mu_2\).}
	\label{tab:cycles}
	\small
	\begin{tabular}{cccccccc}
		\hline
		\(\mu\) & \(t\) & \(T\) & \(x_0\) & \(y_0\) & \(H\) & \(\mu_1\) & \(\mu_2\)\\
		\hline
		0.100 & 4.1921 & 8.3842 & 2.4016 & 1.2186 & 0.77071 & \(0.17492+0.47631i\) & \(0.17492-0.47631i\)\\
		0.300 & 3.9007 & 7.8015 & 5.5377 & 1.1053 & 0.29030 & \(-0.07005+0.18690i\) & \(-0.07005-0.18690i\)\\
		0.548 & 3.7181 & 7.4361 & 12.4281 & 0.9454 & 0.09294 & \(-0.01816+0.07387i\) & \(-0.01816-0.07387i\)\\
		1.000 & 3.5412 & 7.0823 & 46.0350 & 0.7213 & 0.01651 & \(0.04464\) & \(0.00550\)\\
		\hline
	\end{tabular}
\end{table}

\PN Figure~\ref{fig:cycle_exe1} illustrates the limit cycle associated with the admissible pair \(\mu=0.548\), \(H=0.09294\). 

\begin{figure}[H] \centering \includegraphics[width=0.4\textwidth]{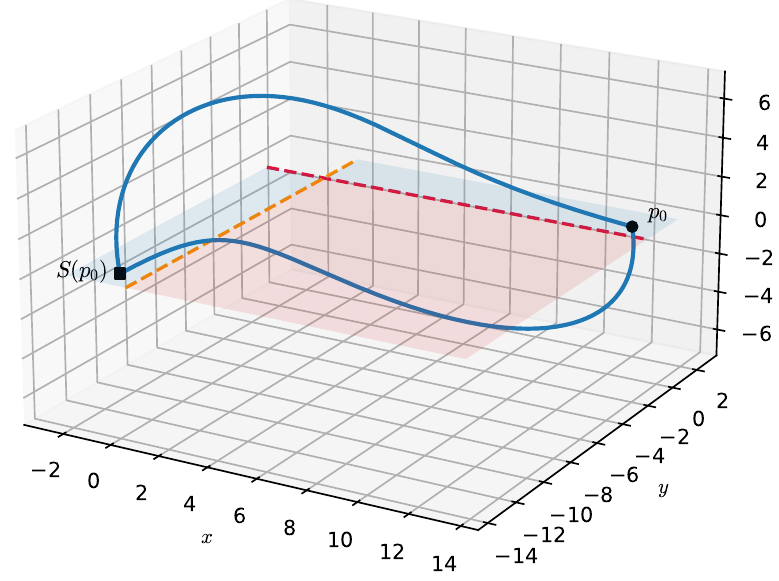}
	\caption{Symmetric limit cycle of the canonical system, with crossing regions in blue and the sliding region in red.} \label{fig:cycle_exe1} 
\end{figure}

\section{Concluding remarks}

\PN For every \(\mu>0\), the symmetric zero-divergence class has one symmetric simple-period crossing cycle when \(H\in\mathcal I_\mu\), and none outside this interval. The half-period parametrizes this family and determines the crossing points, period, and spectral data. Each cycle is hyperbolic and orbitally asymptotically stable. Stability remains strict toward grazing, whereas the \(+1\) spectral boundary is approached only at infinite amplitude. Each cycle persists locally under small \(C^1\) perturbations, and its uniqueness within the simple-period crossing class is preserved under sufficiently small perturbations in the Whitney \(C^1\) topology.

\PN The centered energy establishes convergence for crossing-only trajectories and excludes crossing-only periodic orbits with a higher number of crossings. Hence the classified simple-period cycle is the unique crossing-only periodic orbit. Under~\eqref{eq:nonsingular_sliding}, every attractive sliding episode ends in finite time at an exit fold or at the two-fold. If the trajectory does not reach the two-fold, only finitely many sliding episodes occur, the trajectory eventually becomes crossing-only, and the classified cycle is its \(\omega\)-limit set. The initial conditions whose trajectories reach the two-fold are contained in countably many analytic surfaces and curves. Together with the switching plane and the stable affine lines, they form a set of zero three-dimensional Lebesgue measure. Consequently, the classified cycle governs the asymptotic oscillatory dynamics for every initial condition in \(\mathbb R^3\), with the exception of a set of zero Lebesgue measure. This yields the global asymptotic classification established in this work. The exceptional set provides a zero-measure cover sufficient for the global classification, while the actual basin may be larger. The classification concerns the Filippov trajectories selected throughout the paper, with evolution stopped at the two-fold and with no continuation selected in the escaping region. 


\section*{Funding} \PN No funding was received to assist with the preparation of this manuscript.

\section*{Declaration of competing interest}
\PN The author declares no competing financial interests or personal relationships that could have appeared to influence the work reported in this paper.

\section*{Declaration on the use of generative AI.}
\PN During the development of this work, GPT-6 Astra (OpenAI) was used as an AI-assisted research tool to explore and refine the formulation and proof strategy underlying the global crossing dissipation and global attraction arguments. All mathematical arguments, computations, statements, and proofs resulting from this process were subsequently verified and critically audited by the author. The author takes full responsibility for the correctness, interpretation, and presentation of the results in the manuscript.

\section*{Acknowledgments}

\PN The author thanks Professor Armengol Gasull, Universitat Autònoma de Barcelona, for suggesting the use of first integrals in the search for limit cycles during the XII Workshop on Dynamical Systems, held at IBILCE/UNESP in São José do Rio Preto, Brazil, from November 6 to 9, 2023. This suggestion proved useful in the development of the closure analysis.

\bibliographystyle{elsarticle-num-names}
\bibliography{referencial}

\end{document}